\documentclass[11pt,b5paper,notitlepage]{article}
\usepackage[b5paper, margin={0.5in,0.65in}]{geometry}

\usepackage{amsmath,amscd,amssymb,amsthm,mathrsfs,amsfonts,layout,indentfirst,graphicx,caption,mathabx, stmaryrd,appendix,calc,imakeidx,upgreek,appendix} % mathabx for \widecheck
\usepackage[dvipsnames]{xcolor}
\usepackage{palatino}  %template
\usepackage{slashed} % Dirac operator
\usepackage{mathrsfs} % Enable using \mathscr
\usepackage{extarrows} % long equal sign, \xlongequal{blablabla}
\usepackage{enumitem} % enumerate label change e.g. [label=(\alph*)]  shows (a) (b) 

\usepackage{fancyhdr} % date in footer
\usepackage{verbatim}
\usepackage{halloweenmath}
\usepackage{simpler-wick}

\usepackage{tcolorbox}
\tcbuselibrary{theorems}
\usepackage{tikz-cd}
\usepackage[nottoc]{tocbibind}   % Add  reference to ToC

\makeindex

\usepackage{lipsum}
\let\OLDthebibliography\thebibliography
\renewcommand\thebibliography[1]{
	\OLDthebibliography{#1}
	\setlength{\parskip}{0pt}
	\setlength{\itemsep}{2pt} 
}

\allowdisplaybreaks  %allow aligns to break between pages
\usepackage{latexsym}
\usepackage{chngcntr}
\usepackage[colorlinks,linkcolor=blue,anchorcolor=blue, linktocpage,
]{hyperref}
\hypersetup{ urlcolor=cyan,
	citecolor=[rgb]{0,0.5,0}}

\counterwithin{figure}{section}

\theoremstyle{definition}
\newtheorem{df}{Definition}[section]
\newtheorem{eg}[df]{Example}

\newtheorem{rem}[df]{Remark}
\newtheorem{ass}[df]{Assumption}

\theoremstyle{plain}
\newtheorem{thm}[df]{Theorem}

\newtheorem{pp}[df]{Proposition}
\newtheorem{co}[df]{Corollary}
\newtheorem{lm}[df]{Lemma}

\newcommand{\fk}{\mathfrak}
\newcommand{\mc}{\mathcal}
\newcommand{\wtd}{\widetilde}
\newcommand{\wht}{\widehat}
\newcommand{\wch}{\widecheck}
\newcommand{\ovl}{\overline}
\newcommand{\Lbf}{\mathbf{L}}
\newcommand{\Sbb}{\mathbb{S}}

\newcommand{\Rbf}{\mathbf{R}}

\newcommand{\tr}{\mathrm{t}} %transpose
\newcommand{\Tr}{\mathrm{Tr}}
\newcommand{\End}{\mathrm{End}} %endomorphism
\newcommand{\idt}{\mathbf{1}}
\newcommand{\Hom}{\mathrm{Hom}}
\newcommand{\Conf}{\mathrm{Conf}}
\newcommand{\Res}{\mathrm{Res}}

\newcommand{\ML}{\mathcal{L}}

\newcommand{\SV}{\mathscr{V}}
\newcommand{\Span}{\mathrm{Span}}

\newcommand{\scr}{\mathscr}

\newcommand{\gk}{\mathfrak g}

\newcommand{\xk}{\mathfrak x}
\newcommand{\yk}{\mathfrak y}
\newcommand{\zk}{\mathfrak z}

\newcommand{\im}{\mathbf{i}}

\newcommand{\sgm}{\varsigma}
\newcommand{\SX}{{S_{\fk X}}}

\newcommand{\mbb}{\mathbb}
\newcommand{\mbf}{\mathbf}

\newcommand{\blt}{\bullet}

\newcommand{\Vbb}{\mathbb V}

\newcommand{\Xbb}{\mathbb X}
\newcommand{\Wbb}{\mathbb W}
\newcommand{\Mbb}{\mathbb M}
\newcommand{\Gbb}{\mathbb G}
\newcommand{\Cbb}{\mathbb C}
\newcommand{\Nbb}{\mathbb N}
\newcommand{\Zbb}{\mathbb Z}
\newcommand{\Pbb}{\mathbb P}

\newcommand{\cbf}{\mathbf c}

\newcommand{\btl}{\blacktriangleleft}
\newcommand{\btr}{\blacktriangleright}

\newcommand{\Sbf}{\mathbf{S}}

\newcommand{\pr}{\mathrm {pr}}

\newcommand{\ibf}{\mathbf 1}
\newcommand{\wbf}{\mathbf w}

\newcommand{\Ljss}{{L_j(0)_\mathrm{s}}}

\newcommand{\Ljni}{{L_j(0)_\mathrm{n}}}
\newcommand{\<}{\left\langle}
\renewcommand{\>}{\right\rangle}
\newcommand{\MO}{\mathcal{O}}
\newcommand{\MU}{\mathcal{U}}
\newcommand{\ME}{\mathcal{E}}

\newcommand{\MC}{\mathcal{C}}
\newcommand{\MB}{\mathcal{B}}

\newcommand{\fx}{\mathfrak{X}}

\usepackage{tipa} % wierd symboles e.g. \textturnh

\newcommand{\ST}{\mathscr{T}}
\newcommand{\SJ}{\mathscr{J}}
\newcommand{\SW}{\mathscr{W}}
\newcommand{\MV}{\mathcal{V}}
\usepackage{tipx}

\newcommand{\MD}{\mathcal{D}}
\newcommand{\MS}{\mathcal{S}}
\newcommand{\FA}{\mathfrak{A}}

\newcommand{\mk}{\mathfrak m}
\newcommand{\pre}{\mathrm{pre}}
\newcommand{\bk}[1]{\langle {#1}\rangle}

\newcommand{\bbs}{\boxbackslash}
\newcommand{\fq}{{\mathfrak Q}}

\newcommand{\id}{\mathrm{id}}
\newcommand{\obf}{\mathbf{0}}
\newcommand{\eps}{\varepsilon}
\newcommand{\dps}{\displaystyle}
\newcommand{\Lini}{{L_i(0)_\mathrm{n}}}
\newcommand{\Liss}{{L_i(0)_\mathrm{s}}}
\newcommand{\Lbni}{{L_\bullet(0)_\mathrm{n}}}
\newcommand{\Lbss}{{L_\bullet(0)_\mathrm{s}}}

\newcommand{\fn}{\mathfrak{N}}
\newcommand{\fy}{\mathfrak{Y}}

\usepackage{calligra}
\DeclareMathOperator{\shom}{\mathscr{H}\text{\kern -3pt {\calligra\large om}}\,}
\DeclareMathOperator{\sext}{\mathscr{E}\text{\kern -3pt {\calligra\large xt}}\,}
\DeclareMathOperator{\Rel}{\mathscr{R}\text{\kern -3pt {\calligra\large el}~}\,}
\DeclareMathOperator{\sann}{\mathscr{A}\text{\kern -3pt {\calligra\large nn}}\,}
\DeclareMathOperator{\send}{\mathscr{E}\text{\kern -3pt {\calligra\large nd}}\,}
\DeclareMathOperator{\stor}{\mathscr{T}\text{\kern -3pt {\calligra\large or}}\,}

\usepackage{aurical}
\DeclareMathOperator{\VVir}{\text{\Fontlukas V}\text{\kern -0pt {\Fontlukas\large ir}}\,}

\numberwithin{equation}{section}

\title{Analytic Conformal Blocks of $C_2$-cofinite Vertex Operator Algebras II: Convergence of Sewing and Higher Genus Pseudo-$q$-traces}
\author{{\sc Bin Gui, Hao Zhang}
}
\date{}
\begin{document}\sloppy % avoid stretch into margins
	\pagenumbering{arabic}
	%\pagenumbering{gobble}
	%\newpage
	%\setcounter{page}{1}
	\setcounter{section}{-1}

	\maketitle

%%%%%%%%%%%%%%%%%%%%%%%%%%%%%%%%%%%%%%%%%%%%%%%5
\newcommand\blfootnote[1]{%
	\begingroup
	\renewcommand\thefootnote{}\footnote{#1}%
	\addtocounter{footnote}{-1}%
	\endgroup
}
% Footnote without marker

%%%%%%%%%%%%%%%%%%%%%%%%%%%%%%%%%%%%%%%%%%%%%

%\hyperlink{currentwriting}{Current writing progress}~~~~~~ 

\begin{abstract}
Let $\Vbb=\bigoplus_{n\in\Nbb}\Vbb(n)$ be a $C_2$-cofinite vertex operator algebra. We prove the convergence of Segal's sewing of conformal blocks associated to analytic families of pointed compact Riemann surfaces and grading-restricted generalized $\Vbb^{\otimes N}$-modules (where $N=1,2,\dots$) that are not necessarily tensor products of $\Vbb$-modules, generalizing significantly the results on convergence in \cite{Gui-sewingconvergence}.

We show that ``higher genus pseudo-$q$-traces" (called \textbf{pseudo-sewing} in this article) can be recovered from the above generalization of Segal's sewing to $\Vbb^{\otimes N}$-modules. Therefore, our result on the convergence of the generalized Segal sewing implies the convergence of pseudo-sewing, and hence covers both the convergence of genus-$0$ sewing in \cite{Hua-differential-genus-0,HLZ7} and the convergence of pseudo-$q$-traces in \cite{Miy-modular-invariance} and \cite{Fio-genus-1}.

Using a similar method, we also prove the convergence of Virasoro uniformization, i.e., the convergence of conformal blocks deformed by non-autonomous meromorphic vector fields near the marked points. The local freeness of the \textit{analytic} sheaves of conformal blocks is a consequence of this convergence. It will be used in the third paper of this series to prove the sewing-factorization theorem.
\end{abstract}

\tableofcontents

%\vspace{-0.5cm}
%\blfootnote{Last major revision:  2021.6}

%%%%%%%%%%%%%%%%%%%%%%%%%%%%%%%%5
%\makeatletter
%\newcommand*{\toccontents}{\@starttoc{toc}}
%\makeatother
%\toccontents

% title and table of contents same page

%%%%%%%%%%%%%%%%%%%%%%%%%%%%%

	%%%%%%%%%%%%%%%%%%%%%%%%%%%%%%%%%%%%%%%%%%%%%%%%%%%%%%%%%
	
	%\newpage
	%$~$
	%\renewcommand\contentsname{} % the empty name
	
	%\begingroup
	%\let\clearpage\relax
	%\vspace{-2cm} % the removed space. Set as appropriate

	% Remove header of table of contents
	
	%%%%%%%%%%%%%%%%%%%%%%%%%%%%%%%%%%%%%%%%%%%%%%%%%%%%%%%

\section{Introduction: towards the geometry of pseudo-$q$-traces}

\nocite{HLZ1,HLZ2,HLZ3,HLZ4,HLZ5,HLZ6,HLZ7,HLZ8}

This paper is the second part of the series on conformal blocks associated to a $C_2$-cofinite $\Nbb$-graded vertex operator algebra (VOA) $\Vbb=\bigoplus_{n\in\Nbb}\Vbb(n)$ and analytic families of compact Riemann surfaces. The final goal of the series is to prove a \textbf{sewing-factorization theorem} (\textbf{SF theorem}). 

An outline of this theorem was presented in the Introduction of \cite{GZ1}. In genus $0$, this theorem encompasses the associativity of intertwining operators as proved by Huang-Lepowsky-Zhang in \cite{HLZ1,HLZ2}-\cite{HLZ8} and by Huang in \cite{Hua-projectivecover}. In genus $1$, we will show in a future paper that the SF theorem implies the modular invariance property studied in \cite{Zhu-modular-invariance,Miy-modular-invariance,Hua-differential-genus-1,AN-pseudo-trace,Hua-modular-C2}. In higher (i.e. arbitrary) genus, and when $\Vbb$ is also rational, the theorem specializes to the analytic version of the factorization theorem and the (formal) sewing theorem proved by \cite{DGT2}, namely, Thm. 12.1 of \cite{Gui-sewingconvergence}.

\subsection{Convergence of sewing conformal blocks}\label{lbb62}

In this paper, we prove the convergence of sewing conformal blocks, which constitutes half of the SF theorem. (The other half is that the sewing construction implements an isomorphism of two spaces of conformal blocks: the space of conformal blocks before sewing and the space of those after sewing.) The results proved in this paper are interesting in their own right. Therefore, we have structured this paper in a way that it can be read independently of the other parts of the series.

The sewing construction is the most fundamental operation in 2d conformal field theory, as highlighted in Segal's seminal works \cite{Segal-CFT1,Segal-CFT2}. For a concise (and necessarily incomplete) overview of the history behind the proof of convergence of sewing for rational VOAs, we refer readers to the Introduction of \cite{Gui-sewingconvergence}.

In fact, the proof of convergence of sewing already played an important role in the early history of VOAs. A notable feature of rational VOAs is that their characters $\Tr_\Mbb q^{L_0-\frac c{24}}$ (where $\Mbb$ runs through all irreducibles of $\Vbb$) form a vector-valued modular form. More generally, Zhu proved in \cite{Zhu-modular-invariance} the remarkable result that all $v\mapsto \Tr_\Mbb Y_\Mbb(v,z)q^{L_0-\frac c{24}}$ form an $\mathrm{SL}(2,\Zbb)$-invariant space. This modular invariance property can be reformulated in the following way:

\begin{enumerate}[label=(\arabic*)]
\item \textbf{Convergence of sewing}. For $0<|q|<1$, the sewing construction $v\mapsto \Tr_\Mbb Y_\Mbb(\mc U(\alpha)v,1)q^{L(0)}$ converges absolutely to a vacuum torus conformal block $\upphi_q$. (Here, $\mc U(\alpha)$ is a change of coordinate operator expressed in terms of the Virasoro operators.)
\item \textbf{Factorization}. For each fixed $q$, every vacuum torus conformal block arises from the sewing construction in (1). Moreover, this ``factorization" process can be holomorphic with respect to $q$.
\item \textbf{Flat connections}. The spaces of vacuum torus conformal blocks form a holomorphic vector bundle together with a canonical modular-invariant (flat) connection $\nabla$. 
\item \textbf{Parallelism}. The section $q\mapsto \upphi_q$ in (1) is parallel up to the projective term $\frac c{24}$ where $c$ is the central charge. Namely, $q\nabla_{\partial_q}\upphi_q=\frac c{24}\upphi_q$.
\end{enumerate}
A primary purpose of (3) and (4) is to ensure that the $S$ and $T$ matrices obtained by modular transforming $q^{-\frac c{24}}\upphi_q$ are ``independent of $q$".

The associativity of intertwining operators in \cite{HLZ1,HLZ2}-\cite{HLZ8} and \cite{Hua-projectivecover} can be phrased in a similar fashion, where $\Vbb$ is not assumed to be rational:

\begin{enumerate}[label=(\alph*)]
\item \textbf{Convergence of sewing}. The product $\mc Y_1(w_1,z_1)\mc Y_2(w_2,z_2)$ resp. the iterate $\mc Y_1(\mc Y_2(w_2,z_2-z_1)w_1,z_1)$ of intertwining operators ($\approx$ genus-$0$ three-pointed conformal blocks) converge absolutely on suitable domains to a genus-$0$ four-pointed conformal block $\upphi_{z_1,z_2}$ (with marked points $0,z_1,z_2,\infty$).
\item \textbf{Factorization}. For each fixed (suitable) $z_1,z_2$, every genus-$0$ conformal block with marked points $0,z_1,z_2,\infty$ can be factored into a product resp. an iterate of intertwining operators. Moreover, this factorization process can be holomorphic with respect to $z_1,z_2$.
\item \textbf{Flat connections}. The space of  genus-$0$ four-pointed conformal blocks form a holomorphic vector bundle together with a canonical M\"obius-invariant flat connection $\nabla$. 
\item \textbf{Parallelism}. The section $(z_1,z_2)\mapsto \upphi_{z_1,z_2}$ is parallel under $\nabla$.
\end{enumerate}
(c) and (d) ensure that the braiding and fusion matrices are ``independent of $z_1,z_2$". 

We emphasize that \textit{factorization is the inverse process of sewing}. Our SF theorem specializes to (1)+(2) in genus $1$ and (a)+(b) in genus $0$. The general theory related to (3)+(4) and (c)+(d) will also be given in this paper.

\subsection{Why do pseudo-$q$-traces appear in genus-$1$ but not in genus-$0$ ?}\label{lbb60}

One of the most puzzling phenomena for irrational $C_2$-cofinite VOAs is that (a)-(d) continue to hold, as indicated by the works of Huang-Lepowsky-Zhang and \cite{Hua-projectivecover}; however, (2) does not hold. Indeed, Miyamoto showed in \cite{Miy-modular-invariance} that to achieve a factorization as in (2), it is necessary to consider in (1) not only the standard $q$-traces, but also pseudo-$q$-traces $\Tr_\Mbb^\omega Y_\Mbb(\mc U(\alpha)v,1)q^{L(0)}$. We aim to give a \textit{conceptual} explanation of the following question: \textit{Why are pseudo-$q$-traces needed only in genus-$1$ but not in genus-$0$?}

The permutation-twisted/untwisted correspondence obtained in \cite{Gui-permutation} suggests that sewing \textit{and factorization} could be applied to higher genus Riemann surfaces even without (explicitly) appealing to pseudo-$q$-traces. %(Therefore, Segal's sewing is sufficient in finite logarithmic CFT!) 
That correspondence says roughly that if $G$ is a subgroup of $S_N$ acting by permutation on $\Vbb^{\otimes N}$, then the genus-$0$ conformal blocks for $G$-twisted  $\Vbb^{\otimes N}$-modules correspond to conformal blocks for (untwisted) $\Vbb^{\otimes K}$-modules (where $K\in\Zbb_+$ is possibly different from $N$) and branched coverings of $\Pbb^1$. Moreover, the sewing of the LHS corresponds to the sewing of the RHS. Therefore:
\begin{itemize}
\item The convergence of products/iterates and the associativity of intertwining operators among  $G$-twisted $\Vbb^{\otimes N}$-modules\footnote{This holds at least when $\Vbb$ is $C_2$-cofinite and $G$ is solvable. In that case, $(\Vbb^N)^G$ is $C_2$-cofinite \cite{Miy-C2-orbifold}. So one can combine Huang-Lepowski-Zhang with \cite{Hua-projectivecover} and \cite{McR-equiv} to prove the associativity for intertwining operators of $\Vbb^N$ among $G$-twisted modules. } can be translated to the sewing-factorization theorem for certain higher genus conformal blocks of untwisted $\Vbb^{\otimes K}$-modules via branched coverings.
\end{itemize}
In fact, \cite[Sec. 0.2]{Gui-permutation} explains how the associativity of $\Zbb_2$-twisted intertwining operators corresponds to the untwisted modular invariance. Although \cite{Gui-permutation} focuses on $\Vbb^{\otimes K}$-modules that are direct sums of those of the form $\Mbb_1\otimes\cdots\otimes\Mbb_K$ with each $\Mbb_i$ a $\Vbb$-module, most results in \cite{Gui-permutation} can be easily generalized to arbitrary $\Vbb^{\otimes K}$-modules.

\subsection{Segal's sewing for modules of tensor products of $\Vbb$}\label{lbb65}

In the Introduction of \cite{GZ1}, we have already presented the sewing construction in arbitrary genus that ensures the factorization holds. (In fact, our SF theorem is deeply inspired by the observation in Sec. \ref{lbb60}.) Let us briefly recall its formulation. For simplicity, instead of a family of Riemann surfaces, we consider a single fiber
\begin{align*}
\wtd\fx=(\wtd C|x_\blt\Vert x_\blt',x_\blt'')=(\wtd C|x_1,\dots,x_N\Vert x_1',\dots,x_R',x_1'',\dots,x_R'')
\end{align*}
where $\wtd C$ is a (possibly disconnected) compact Riemann surface, and $x_\blt,x_\blt',x_\blt''$ are distinct marked points on $\wtd C$. We assume that each connected component of $\wtd C$ contains at least one of $x_1,\dots,x_N$. At each $x_i,x_j',x_j''$ we associate an (analytic) local coordinate $\eta_i,\xi_j,\varpi_j$. After choosing suitable $\eps_j>0$, one can sew $\wtd\fx$ along each pair of points $(x_j',x_j'')$ via the relation $\xi_j\varpi_j=q_j$. The result of sewing is a family of nodal curves with marked points (and local coordinates)
\begin{align*}
\fx=(\pi:\MC\rightarrow\MB|\sgm_1,\dots,\sgm_N)
\end{align*}
Here, $\MB=\MD_{\eps_\blt}:=\MD_{\eps_1}\times\cdots\MD_{\eps_R}$ and $\MD_{\eps_j}=\{q_j\in\Cbb:|q_j|<\eps_j\}$. The marked points $x_j',x_j''$ disappear after sewing, and the constant extension of $x_i$ gives the section $\sgm_i$.Let $\MD_{\eps_j}^\times=\MD_{\eps_j}-\{0\}$. Then $\fx$ is a smooth family on $\MD^\times_{\eps_\blt}:=\MD^\times_{\eps_1}\times\cdots\times\MD^\times_{\eps_R}$.

Now, we choose a (grading-restricted) $\Vbb^{\otimes N}$-module $\Wbb$ associated simultaneously to the ordered marked points $x_\blt$ of $\wtd\fx$. Choose an $\Vbb^{\otimes R}$-module $\Mbb$ associated to $x_\blt'$, and associate the contragredient module $\Mbb'$ to $x_\blt''$.

We let $\scr T^*_{\wtd\fx}(\Wbb\otimes\Mbb\otimes\Mbb')$ be the space of conformal blocks associated to $\wtd\fx$ and $\Wbb\otimes\Mbb\otimes\Mbb'$. Its elements are linear functionals
\begin{align*}
\uppsi:\Wbb\otimes\Mbb\otimes\Mbb'\rightarrow\Cbb
\end{align*}
satisfying certain invariance property under the action of sheaf of VOAs $\scr V_{\wtd C}$. Let $L_j(n)$ be the Virasoro operator for the $j$-th component (cf. \eqref{eqb93}). Let $\Mbb=\bigoplus_{\lambda_\blt\in\Cbb^R}\Mbb_{[\lambda_\blt]}$ be the decomposition of $\Mbb$ with respect to the joint generalized eigenspaces of $L_1(0),\dots,L_R(0)$. Let $q_\blt^{L_\blt(0)}=q_1^{L_1(0)}\cdots q_R^{L_R(0)}$. Then the \textbf{sewing} \pmb{$\MS\uppsi$} of $\uppsi$ %(in the sense of Segal) 
is defined by taking contraction in the usual sense (\textit{but not under a pseudo-trace}), i.e.,
\begin{gather}\label{eq2}
\begin{gathered}
\MS\uppsi:\Wbb\rightarrow \Cbb\{q_\blt\}[\log q_\blt]=\Cbb\{q_1,\dots,q_R\}[\log q_1,\dots,\log q_R]\\
\MS\uppsi(w)=\sum_{\lambda_\blt\in\Cbb^R}\sum_{\alpha\in\fk A_{\lambda_\blt}} \uppsi\big(w_\blt\otimes q_\blt^{L_\blt(0)}m_{(\lambda_\blt,\alpha)}\otimes \wch m_{(\lambda_\blt,\alpha)}\big)
\end{gathered}
\end{gather}
where $(m_{(\lambda_\blt,\alpha)})_{\alpha\in\fk A_{\lambda_\blt}}$ is a (finite) basis of $\Mbb_{[\lambda_\blt]}$ with dual basis $(\wch m_{(\lambda_\blt,\alpha)})_{\alpha\in\fk A_{\lambda_\blt}}$. (Note that $L_j(0)$ preserves each $\Mbb_{[\lambda_\blt]}$.) See Subsec. \ref{lbb61} for the precise definition of $\MS\uppsi$. Note that in the non-semisimple case, the sewing $\MS \uppsi$ is more general than the sewing 
formulated in \cite{Segal-CFT1,Segal-CFT2}. But we will still call the sewing $\MS\uppsi$ a Segal's sewing.

Associate $\Wbb$ also to the ordered marked points $\sgm_\blt(\MB)$ of $\fx$. Let $\scr T^*_\fx(\Wbb)$ be the $\MB$-sheaf of conformal blocks associated to $\fx$ and $\Wbb$. Thus, on $\MD_{\eps_\blt}^\times$, a section of this sheaf is precisely a linear map $\Wbb\rightarrow\MO(\MD_{\eps_\blt}^\times)$ whose restriction to each $q_\blt\in \MD_{\eps_\blt}^\times$ is a conformal block associated to $\Wbb$ and the fiber $\fx_{q_\blt}$. (See Rem. \ref{lbb27} for details.) A major result of this paper (Thm. \ref{lbb49}) is the generalization of the following theorem to the case that $\wtd\fx$ is a family of pointed compact Riemann surfaces.

\begin{thm}\label{lbb64}
Let $\uppsi\in\scr T^*_{\wtd\fx}(\Wbb\otimes\Mbb\otimes\Mbb')$. Then $\MS\uppsi$ converges absolutely and locally uniformly (\textbf{a.l.u.}) on $\MD^\times_{\eps_\blt}$ to an element of $H^0(\wht\MD_{\eps_\blt}^\times,\scr T^*_\fx(\Wbb))$ (i.e., a section of $\scr T^*_\fx(\Wbb)$ on $\wht\MD_{\eps_\blt}^\times$), where $\wht\MD_{\eps_\blt}^\times$ is the universal cover of $\MD_{\eps_\blt}^\times$.
\end{thm}

Now, when does the factorization hold in this setting? In other words, in which situation can every conformal block associated to $\fx_{q_\blt}$ and $\Wbb$ be obtained by sewing conformal blocks associated to $\wtd\fx$ and $\Wbb\otimes\Mbb\otimes\Mbb'$? An answer has been provided in the Introduction of \cite{GZ1} and will be proved rigorously in the third part of the series:
\begin{itemize}
\item The factorization holds if the sewing of $\wtd\fx$ is a \textbf{disjoint sewing}, namely, if $\wtd C$ is a disjoint union of open subsets $\wtd C'\sqcup\wtd C''$ such that $x_1',\dots,x_R'$ belong to $\wtd C'$ and $x_1'',\dots,x_R''$ belong to $\wtd C''$.
\end{itemize}
For example, in Sec. \ref{lbb62}, the sewing in (a) is disjoint, any ``branched covering" of (a) is disjoint, but (1) is not a disjoint sewing. 

Thus, we can address the question posed in Sec. \ref{lbb60}: The reason that the sewing construction $v\mapsto \Tr_\Mbb Y_\Mbb(v,1)q^{L(0)}$ does not provide all vacuum torus conformal blocks is that this type of sewing is a ``self-sewing" rather than a disjoint sewing. 

There are two ways to achieve factorization for non-disjoint sewing. The first approach retains the sewing of conformal blocks but modifies the way of sewing compact Riemann surfaces. The second approach does the opposite. More precisely:
\begin{enumerate}
\item[($\alpha$)] Transform the non-disjoint sewing of \textit{Riemann surfaces} to a disjoint one, as shown e.g. in Fig. \ref{fig1}. Then the Segal sewing of \textit{conformal blocks} (in the sense of \eqref{eq2}) is sufficient.
\item[($\beta$)] Keep the sewing of \textit{Riemann surfaces}. However, replace Segal's sewing of \textit{conformal blocks} (in the sense of \eqref{eq2}) with pseudo-sewing (i.e. pseudo-$q$-trace).
\end{enumerate}

\begin{figure}[h]
	\centering
\scalebox{0.5}{

\tikzset{every picture/.style={line width=0.75pt}} %set default line width to 0.75pt        

\begin{tikzpicture}[x=0.75pt,y=0.75pt,yscale=-1,xscale=1]

%Shape: Regular Polygon [id:dp13847172478088754] 
\draw   (230.35,107.87) .. controls (233.06,126.69) and (205.12,113.86) .. (205.57,131.39) .. controls (206.02,148.93) and (235.83,134.45) .. (233.51,152.77) .. controls (231.19,171.1) and (192.1,153.19) .. (162.92,151.53) .. controls (133.73,149.87) and (103.29,182.28) .. (77.61,161.76) .. controls (51.93,141.23) and (67.25,96.32) .. (96.54,87.77) .. controls (125.82,79.22) and (138.23,113.9) .. (164.36,111.16) .. controls (190.49,108.42) and (227.65,89.05) .. (230.35,107.87) -- cycle ;
%Curve Lines [id:da9316500690377496] 
\draw    (87.77,130.02) .. controls (100.52,116.13) and (110.59,119.83) .. (119.32,131.87) ;
%Curve Lines [id:da8102235064386933] 
\draw    (91.8,127.24) .. controls (99.18,134.65) and (107.24,133.72) .. (114.62,127.24) ;
%Curve Lines [id:da36456753068971737] 
\draw    (140.77,133.56) .. controls (151.15,124.29) and (159.35,126.76) .. (166.45,134.79) ;
%Curve Lines [id:da23162036124938634] 
\draw    (144.05,131.7) .. controls (150.06,136.65) and (156.62,136.03) .. (162.63,131.7) ;
%Shape: Circle [id:dp45783593245851417] 
\draw  [fill={rgb, 255:red, 0; green, 0; blue, 0 }  ,fill opacity=1 ] (223.4,109.03) .. controls (223.4,108.13) and (224.13,107.39) .. (225.03,107.39) .. controls (225.94,107.39) and (226.67,108.13) .. (226.67,109.03) .. controls (226.67,109.93) and (225.94,110.67) .. (225.03,110.67) .. controls (224.13,110.67) and (223.4,109.93) .. (223.4,109.03) -- cycle ;
%Shape: Circle [id:dp9525268276367043] 
\draw  [fill={rgb, 255:red, 0; green, 0; blue, 0 }  ,fill opacity=1 ] (226.4,152.53) .. controls (226.4,151.63) and (227.13,150.89) .. (228.03,150.89) .. controls (228.94,150.89) and (229.67,151.63) .. (229.67,152.53) .. controls (229.67,153.43) and (228.94,154.17) .. (228.03,154.17) .. controls (227.13,154.17) and (226.4,153.43) .. (226.4,152.53) -- cycle ;
%Shape: Regular Polygon [id:dp3511170107820867] 
\draw   (499.35,108.87) .. controls (502.06,127.69) and (474.12,114.86) .. (474.57,132.39) .. controls (475.02,149.93) and (504.83,135.45) .. (502.51,153.77) .. controls (500.19,172.1) and (461.1,154.19) .. (431.92,152.53) .. controls (402.73,150.87) and (372.29,183.28) .. (346.61,162.76) .. controls (320.93,142.23) and (336.25,97.32) .. (365.54,88.77) .. controls (394.82,80.22) and (407.23,114.9) .. (433.36,112.16) .. controls (459.49,109.42) and (496.65,90.05) .. (499.35,108.87) -- cycle ;
%Curve Lines [id:da2158644950787414] 
\draw    (356.77,131.02) .. controls (369.52,117.13) and (379.59,120.83) .. (388.32,132.87) ;
%Curve Lines [id:da6899354300685749] 
\draw    (360.8,128.24) .. controls (368.18,135.65) and (376.24,134.72) .. (383.62,128.24) ;
%Curve Lines [id:da594841733390588] 
\draw    (409.77,134.56) .. controls (420.15,125.29) and (428.35,127.76) .. (435.45,135.79) ;
%Curve Lines [id:da5149262747913113] 
\draw    (413.05,132.7) .. controls (419.06,137.65) and (425.62,137.03) .. (431.63,132.7) ;
%Shape: Circle [id:dp6270802892356779] 
\draw  [fill={rgb, 255:red, 0; green, 0; blue, 0 }  ,fill opacity=1 ] (492.4,110.03) .. controls (492.4,109.13) and (493.13,108.39) .. (494.03,108.39) .. controls (494.94,108.39) and (495.67,109.13) .. (495.67,110.03) .. controls (495.67,110.93) and (494.94,111.67) .. (494.03,111.67) .. controls (493.13,111.67) and (492.4,110.93) .. (492.4,110.03) -- cycle ;
%Shape: Circle [id:dp7078135713985925] 
\draw  [fill={rgb, 255:red, 0; green, 0; blue, 0 }  ,fill opacity=1 ] (495.4,153.53) .. controls (495.4,152.63) and (496.13,151.89) .. (497.03,151.89) .. controls (497.94,151.89) and (498.67,152.63) .. (498.67,153.53) .. controls (498.67,154.43) and (497.94,155.17) .. (497.03,155.17) .. controls (496.13,155.17) and (495.4,154.43) .. (495.4,153.53) -- cycle ;
%Curve Lines [id:da536250492379833] 
\draw [color={rgb, 255:red, 74; green, 144; blue, 226 }  ,draw opacity=1 ] [dash pattern={on 1.5pt off 1.5pt}]  (236.25,111.65) .. controls (245.26,120.95) and (248.01,135.17) .. (238.63,151.19) ;
\draw [shift={(237.73,152.68)}, rotate = 302.15] [color={rgb, 255:red, 74; green, 144; blue, 226 }  ,draw opacity=1 ][line width=0.75]    (6.56,-1.97) .. controls (4.17,-0.84) and (1.99,-0.18) .. (0,0) .. controls (1.99,0.18) and (4.17,0.84) .. (6.56,1.97)   ;
\draw [shift={(234.73,110.18)}, rotate = 42.14] [color={rgb, 255:red, 74; green, 144; blue, 226 }  ,draw opacity=1 ][line width=0.75]    (6.56,-1.97) .. controls (4.17,-0.84) and (1.99,-0.18) .. (0,0) .. controls (1.99,0.18) and (4.17,0.84) .. (6.56,1.97)   ;
%Shape: Polygon Curved [id:ds3062623570111953] 
\draw   (535.38,112.21) .. controls (534.55,98.64) and (571.55,97.14) .. (578.05,108.64) .. controls (584.55,120.14) and (589.67,132.45) .. (581.61,147.29) .. controls (573.55,162.14) and (538.05,165.14) .. (537.55,150.64) .. controls (537.05,136.14) and (561.03,151.19) .. (561.03,131.38) .. controls (561.03,111.56) and (536.21,125.78) .. (535.38,112.21) -- cycle ;
%Shape: Circle [id:dp7630821989609264] 
\draw  [fill={rgb, 255:red, 0; green, 0; blue, 0 }  ,fill opacity=1 ] (540.4,111.03) .. controls (540.4,110.13) and (541.13,109.39) .. (542.03,109.39) .. controls (542.94,109.39) and (543.67,110.13) .. (543.67,111.03) .. controls (543.67,111.93) and (542.94,112.67) .. (542.03,112.67) .. controls (541.13,112.67) and (540.4,111.93) .. (540.4,111.03) -- cycle ;
%Shape: Circle [id:dp46289602527163476] 
\draw  [fill={rgb, 255:red, 0; green, 0; blue, 0 }  ,fill opacity=1 ] (540.9,151.03) .. controls (540.9,150.13) and (541.63,149.39) .. (542.53,149.39) .. controls (543.44,149.39) and (544.17,150.13) .. (544.17,151.03) .. controls (544.17,151.93) and (543.44,152.67) .. (542.53,152.67) .. controls (541.63,152.67) and (540.9,151.93) .. (540.9,151.03) -- cycle ;
%Straight Lines [id:da7687251466744287] 
\draw [color={rgb, 255:red, 74; green, 144; blue, 226 }  ,draw opacity=1 ] [dash pattern={on 1.5pt off 1.5pt}]  (505.15,110.71) -- (530.55,110.68) ;
\draw [shift={(532.55,110.68)}, rotate = 179.95] [color={rgb, 255:red, 74; green, 144; blue, 226 }  ,draw opacity=1 ][line width=0.75]    (6.56,-1.97) .. controls (4.17,-0.84) and (1.99,-0.18) .. (0,0) .. controls (1.99,0.18) and (4.17,0.84) .. (6.56,1.97)   ;
\draw [shift={(503.15,110.71)}, rotate = 359.95] [color={rgb, 255:red, 74; green, 144; blue, 226 }  ,draw opacity=1 ][line width=0.75]    (6.56,-1.97) .. controls (4.17,-0.84) and (1.99,-0.18) .. (0,0) .. controls (1.99,0.18) and (4.17,0.84) .. (6.56,1.97)   ;
%Straight Lines [id:da26514756400474426] 
\draw [color={rgb, 255:red, 74; green, 144; blue, 226 }  ,draw opacity=1 ] [dash pattern={on 1.5pt off 1.5pt}]  (507.65,152.71) -- (533.05,152.68) ;
\draw [shift={(535.05,152.68)}, rotate = 179.95] [color={rgb, 255:red, 74; green, 144; blue, 226 }  ,draw opacity=1 ][line width=0.75]    (6.56,-1.97) .. controls (4.17,-0.84) and (1.99,-0.18) .. (0,0) .. controls (1.99,0.18) and (4.17,0.84) .. (6.56,1.97)   ;
\draw [shift={(505.65,152.71)}, rotate = 359.95] [color={rgb, 255:red, 74; green, 144; blue, 226 }  ,draw opacity=1 ][line width=0.75]    (6.56,-1.97) .. controls (4.17,-0.84) and (1.99,-0.18) .. (0,0) .. controls (1.99,0.18) and (4.17,0.84) .. (6.56,1.97)   ;

% Text Node
\draw (277.95,127.54) node [anchor=north west][inner sep=0.75pt]  [font=\LARGE]  {$=$};

\end{tikzpicture}

}
	\caption{~~Transforming self-sewing to disjoint sewing}
	\label{fig1}
\end{figure}

In Sec. \ref{lbb63}, we will show that ($\beta$) arises naturally as a special case of ($\alpha$). In future work, we will show that in most typical cases, ($\alpha$) and ($\beta$) are roughly equivalent. For now, let me explain why the sewing in \cite{Gui-sewingconvergence} is insufficient to achieve the factorization: In that work, the $\Vbb^{\otimes R}$-module $\Mbb$ is a tensor product of $\Vbb$-modules. As we will see in the following section:
\begin{itemize}
\item The consideration of $\Vbb^{\otimes R}$-modules that are not tensor products of $\Vbb$-modules (or their direct sums) is closely related to the pseudo-trace construction.
\end{itemize}

\subsection{A pseudo-trace $\End^0_A(\Mbb)\rightarrow\Cbb$ is a two-pointed genus-0 conformal block associated to the $\Vbb^{\otimes 2}$-submodule $\End^0_A(\Mbb)$ of $\Mbb\otimes_\Cbb\Mbb'$}

To illustrate the main idea behind the embedding $(\beta)\hookrightarrow(\alpha)$ mentioned above, let's examine the simplest case: Let $\fk Q$ and $\fn$ be a $3$-pointed resp. $2$-pointed sphere with local coordinates:
\begin{align*}
\fk Q=(\Pbb^1|1, \infty,0;z-1,1/z,z)\qquad \fk N=(\Pbb^1|\infty,0;1/z,z)
\end{align*}   
Here, $z$ denotes the standard coordinate of $\Cbb$. Associate the vacuum module $\Vbb$ to the marked point $1\in\fk Q$. Let $\Xbb$ be a $\Vbb^{\otimes 2}$-module associated to the marked points $(\infty,0)$ of $\fn$. Its contragredient $\Xbb'$ is associated to the marked points $(\infty,0)$ of $\fk Q$. Choose conformal blocks
\begin{align*}
\upphi:\Vbb\otimes\Xbb'\rightarrow\Cbb\qquad \uptau:\Xbb\rightarrow \Cbb
\end{align*}
associated to $\fk Q$ and $\fn$ respectively.

Let $\fx$ be the sewing of $\fq$ and $\fn$ along two pairs of points: The first pair is $(\infty_\fq,\infty_\fn)$, and the second pair is $(0_\fq,0_\fn)$.\footnote{Technically, this sewing does not satisfy our assumption in Sec. \ref{lbb65} that each component of $\fq\sqcup\fn$ contains a marked point not intended for sewing. (Look at $\fn$.) However, this minor issue can be easily circumvented by propagating the conformal blocks. See Rem. \ref{lbb56}.} See Fig. \ref{fig2}. Then $\fx$ has base manifold $\MB=\MD_1\times\MD_1$ where $\MD_1=\{q\in\Cbb:|q|<1\}$. By Thm. \ref{lbb64}, for each $0<|q_1|,|q_2|<1$,  $\MS_{q_\blt}(\upphi\otimes\uptau)$ converges to a conformal block associated to $\Vbb$ and the fiber $\fx_{q_\blt}=\fx_{q_1,q_2}$.

\begin{figure}[h]
	\centering
\scalebox{0.8}{

\tikzset{every picture/.style={line width=0.75pt}} %set default line width to 0.75pt        

\begin{tikzpicture}[x=0.75pt,y=0.75pt,yscale=-1,xscale=1]

%Shape: Circle [id:dp7646882158430333] 
\draw   (96,156.36) .. controls (96,133.77) and (114.32,115.45) .. (136.91,115.45) .. controls (159.5,115.45) and (177.82,133.77) .. (177.82,156.36) .. controls (177.82,178.96) and (159.5,197.27) .. (136.91,197.27) .. controls (114.32,197.27) and (96,178.96) .. (96,156.36) -- cycle ;
%Shape: Circle [id:dp5616900944354493] 
\draw   (210,157.36) .. controls (210,134.77) and (228.32,116.45) .. (250.91,116.45) .. controls (273.5,116.45) and (291.82,134.77) .. (291.82,157.36) .. controls (291.82,179.96) and (273.5,198.27) .. (250.91,198.27) .. controls (228.32,198.27) and (210,179.96) .. (210,157.36) -- cycle ;
%Shape: Circle [id:dp5391874524382771] 
\draw  [color={rgb, 255:red, 65; green, 117; blue, 5 }  ,draw opacity=1 ][fill={rgb, 255:red, 0; green, 120; blue, 120 }  ,fill opacity=1 ] (103.86,156.39) .. controls (103.86,155.49) and (104.6,154.76) .. (105.5,154.76) .. controls (106.4,154.76) and (107.14,155.49) .. (107.14,156.39) .. controls (107.14,157.3) and (106.4,158.03) .. (105.5,158.03) .. controls (104.6,158.03) and (103.86,157.3) .. (103.86,156.39) -- cycle ;
%Shape: Circle [id:dp5128159318228409] 
\draw  [color={rgb, 255:red, 255; green, 0; blue, 0 }  ,draw opacity=1 ][fill={rgb, 255:red, 255; green, 0; blue, 0 }  ,fill opacity=1 ] (158.86,179.39) .. controls (158.86,178.49) and (159.6,177.76) .. (160.5,177.76) .. controls (161.4,177.76) and (162.14,178.49) .. (162.14,179.39) .. controls (162.14,180.3) and (161.4,181.03) .. (160.5,181.03) .. controls (159.6,181.03) and (158.86,180.3) .. (158.86,179.39) -- cycle ;
%Shape: Circle [id:dp1165149950080997] 
\draw  [color={rgb, 255:red, 255; green, 0; blue, 0 }  ,draw opacity=1 ][fill={rgb, 255:red, 255; green, 0; blue, 0 }  ,fill opacity=1 ] (156.36,130.89) .. controls (156.36,129.99) and (157.1,129.26) .. (158,129.26) .. controls (158.9,129.26) and (159.64,129.99) .. (159.64,130.89) .. controls (159.64,131.8) and (158.9,132.53) .. (158,132.53) .. controls (157.1,132.53) and (156.36,131.8) .. (156.36,130.89) -- cycle ;
%Shape: Circle [id:dp8446175992505727] 
\draw  [color={rgb, 255:red, 74; green, 144; blue, 226 }  ,draw opacity=1 ][fill={rgb, 255:red, 74; green, 144; blue, 226 }  ,fill opacity=1 ] (229.86,129.89) .. controls (229.86,128.99) and (230.6,128.26) .. (231.5,128.26) .. controls (232.4,128.26) and (233.14,128.99) .. (233.14,129.89) .. controls (233.14,130.8) and (232.4,131.53) .. (231.5,131.53) .. controls (230.6,131.53) and (229.86,130.8) .. (229.86,129.89) -- cycle ;
%Shape: Circle [id:dp01508848468761781] 
\draw  [color={rgb, 255:red, 74; green, 144; blue, 226 }  ,draw opacity=1 ][fill={rgb, 255:red, 74; green, 144; blue, 226 }  ,fill opacity=1 ] (230.36,179.89) .. controls (230.36,178.99) and (231.1,178.26) .. (232,178.26) .. controls (232.9,178.26) and (233.64,178.99) .. (233.64,179.89) .. controls (233.64,180.8) and (232.9,181.53) .. (232,181.53) .. controls (231.1,181.53) and (230.36,180.8) .. (230.36,179.89) -- cycle ;
%Straight Lines [id:da5638961111788239] 
\draw  [dash pattern={on 1.5pt off 1.5pt}]  (180.05,130.59) -- (206.05,130.59) ;
\draw [shift={(208.05,130.59)}, rotate = 180] [color={rgb, 255:red, 0; green, 0; blue, 0 }  ][line width=0.75]    (6.56,-1.97) .. controls (4.17,-0.84) and (1.99,-0.18) .. (0,0) .. controls (1.99,0.18) and (4.17,0.84) .. (6.56,1.97)   ;
\draw [shift={(178.05,130.59)}, rotate = 0] [color={rgb, 255:red, 0; green, 0; blue, 0 }  ][line width=0.75]    (6.56,-1.97) .. controls (4.17,-0.84) and (1.99,-0.18) .. (0,0) .. controls (1.99,0.18) and (4.17,0.84) .. (6.56,1.97)   ;
%Straight Lines [id:da7000826902399737] 
\draw  [dash pattern={on 1.5pt off 1.5pt}]  (182.05,180.59) -- (208.05,180.59) ;
\draw [shift={(210.05,180.59)}, rotate = 180] [color={rgb, 255:red, 0; green, 0; blue, 0 }  ][line width=0.75]    (6.56,-1.97) .. controls (4.17,-0.84) and (1.99,-0.18) .. (0,0) .. controls (1.99,0.18) and (4.17,0.84) .. (6.56,1.97)   ;
\draw [shift={(180.05,180.59)}, rotate = 0] [color={rgb, 255:red, 0; green, 0; blue, 0 }  ][line width=0.75]    (6.56,-1.97) .. controls (4.17,-0.84) and (1.99,-0.18) .. (0,0) .. controls (1.99,0.18) and (4.17,0.84) .. (6.56,1.97)   ;
%Shape: Ellipse [id:dp05124390189331729] 
\draw   (371.14,156.89) .. controls (371.14,137.14) and (395.37,121.14) .. (425.27,121.14) .. controls (455.17,121.14) and (479.41,137.14) .. (479.41,156.89) .. controls (479.41,176.63) and (455.17,192.64) .. (425.27,192.64) .. controls (395.37,192.64) and (371.14,176.63) .. (371.14,156.89) -- cycle ;
%Curve Lines [id:da07898906499947111] 
\draw    (396.91,159.96) .. controls (420.56,142.89) and (439.23,147.44) .. (455.41,162.24) ;
%Curve Lines [id:da6637716928196673] 
\draw    (404.38,156.55) .. controls (418.07,165.65) and (433,164.51) .. (446.7,156.55) ;
%Shape: Circle [id:dp16839202302025713] 
\draw  [color={rgb, 255:red, 65; green, 117; blue, 5 }  ,draw opacity=1 ][fill={rgb, 255:red, 0; green, 120; blue, 120 }  ,fill opacity=1 ] (377.86,158.39) .. controls (377.86,157.49) and (378.6,156.76) .. (379.5,156.76) .. controls (380.4,156.76) and (381.14,157.49) .. (381.14,158.39) .. controls (381.14,159.3) and (380.4,160.03) .. (379.5,160.03) .. controls (378.6,160.03) and (377.86,159.3) .. (377.86,158.39) -- cycle ;

% Text Node
\draw (110.5,141.45) node [anchor=north west][inner sep=0.75pt]    {$1$};
% Text Node
\draw (161.5,109.45) node [anchor=north west][inner sep=0.75pt]    {$\infty $};
% Text Node
\draw (211,108.45) node [anchor=north west][inner sep=0.75pt]    {$\infty $};
% Text Node
\draw (161,197.45) node [anchor=north west][inner sep=0.75pt]    {$0$};
% Text Node
\draw (218,198.95) node [anchor=north west][inner sep=0.75pt]    {$0$};
% Text Node
\draw (58.5,153.95) node [anchor=north west][inner sep=0.75pt]  [color={rgb, 255:red, 65; green, 117; blue, 5 }  ,opacity=1 ] [align=left] {$ $};
% Text Node
\draw (53,158) node [anchor=north west][inner sep=0.75pt]  [color={rgb, 255:red, 65; green, 117; blue, 5 }  ,opacity=1 ] [align=left] { \ \ };
% Text Node
\draw (81.5,148.9) node [anchor=north west][inner sep=0.75pt]  [color={rgb, 255:red, 65; green, 117; blue, 5 }  ,opacity=1 ]  {$\mathbb V$};
% Text Node
\draw (153.5,148.9) node [anchor=north west][inner sep=0.75pt]  [color={rgb, 255:red, 255; green, 0; blue, 0 }  ,opacity=1 ]  {$\mathbb X'$};
% Text Node
\draw (227,148.4) node [anchor=north west][inner sep=0.75pt]  [color={rgb, 255:red, 74; green, 144; blue, 226 }  ,opacity=1 ]  {$\mathbb X$};
% Text Node
\draw (385.5,144.95) node [anchor=north west][inner sep=0.75pt]    {};
% Text Node
\draw (356,150.9) node [anchor=north west][inner sep=0.75pt]  [color={rgb, 255:red, 65; green, 117; blue, 5 }  ,opacity=1 ]  {$\mathbb V$};
% Text Node
\draw (318.64,151.31) node [anchor=north west][inner sep=0.75pt]  [font=\Large]  {$=$};

\end{tikzpicture}

}
	\caption{~~Sewing $\fq\sqcup\fn$ to get the torus $\fx_{q_1,q_2}$.}
	\label{fig2}
\end{figure}

The factorization implies that any vacuum torus conformal block can be written as $\MS_{q_\blt}(\upphi\otimes\uptau)$. We now explain how the pseudo-trace construction fits into our setting.

Let $\Mbb$ be a $\Vbb$-module. If we let $\Xbb$ be $\Mbb\otimes\Mbb'$ and let $\uptau:\Xbb\otimes\Xbb'\rightarrow\Cbb$ be the standard pairing, then $\MS_{q_\blt}(\upphi\otimes\uptau)$ is just the usual $q$-trace (where $q=q_1q_2$). However, we can let $\Xbb$ be a $\Vbb^{\otimes 2}$-submodule of $\Mbb\otimes\Mbb'$. 

The following procedure provides a major way to get $\Vbb^{\otimes 2}$-submodules of $\Mbb\otimes\Mbb'$. Let
\begin{align*}
\End^0(\Mbb)=\bigoplus_{\lambda,\mu\in\Cbb}\Hom_\Cbb(\Mbb_{[\mu]},\Mbb_{[\lambda]})
\end{align*}
where $\Mbb_{[\lambda]}$ is the weight-$\lambda$ generalized eigenspace. The canonical linear equivalence
\begin{align*}
\End^0(\Mbb)\simeq\Mbb\otimes\Mbb'
\end{align*}
makes $\End^0(\Mbb)$ a $\Vbb^{\otimes 2}$-module. Choose a unital subalgebra $A$ of $\End_\Vbb(\Mbb)^{\mathrm{op}}$. Then $\End^0(\Mbb)$ has a $\Vbb^{\otimes 2}$-submodule
\begin{align*}
\End_A^0(\Mbb):=\{T\in\End^0(\Mbb):(Tm)a=T(m a)\text{ for all }m\in\Mbb,a\in A\}
\end{align*}

We now assume that $\Mbb$ is projective as a right $A$-module. Fix a symmetric linear functional $\omega:A\rightarrow\Cbb$. (Namely, $\omega$ is linear and $\omega(ab)=\omega(ba)$ for all $a,b\in A$.) Through the pseudo-trace construction in \cite{Ari10}, we obtain a symmetric linear functional
\begin{align*}
\Tr^\omega:\End^0_A(\Mbb)\rightarrow\Cbb
\end{align*}
The fact that $\uptau:=\Tr^\omega$ is symmetric implies that it is a conformal block associated to $\fn$ and the $\Vbb^{\otimes 2}$-module
\begin{align*}
\Xbb:=\End^0_A(\Mbb)
\end{align*}

Finally, the linear functional
\begin{align*}
\upphi:\Vbb\otimes \Mbb'\otimes\Mbb\rightarrow\Cbb\qquad v\otimes m'\otimes m\mapsto\bk{Y_\Mbb(v,1)m,m'}
\end{align*}
is a conformal block associated to $\fq$ and $\Vbb\otimes\Mbb'\otimes\Mbb$. Using the fact that the vertex operators of $\Mbb$ commute with $A$, we see that $\upphi$ descends to a linear map
\begin{align*}
\upphi:\Vbb\otimes\End^0_A(\Mbb)'\rightarrow\Cbb
\end{align*}
where $\End^0_A(\Mbb)'=\bigoplus_{\mu,\lambda\in\Cbb}\Hom_A(\Mbb_{[\mu]},\Mbb_{[\lambda]})^*$ is the contragredient $\Vbb^{\otimes2}$-module of $\End^0_A(\Mbb)$. %Therefore, at each $q_\blt\in\MD^\times_1\times\MD^\times_1$, $\MS(\uppsi\otimes\Tr^\omega)$ converges to a conformal block associated to $\Vbb$ and $\fx_{q_\blt}$. 
Thus one can define the sewing $\MS(\upphi\otimes\Tr^\omega):\Vbb\rightarrow\Cbb\{q_1,q_2\}[\log q_1,\log q_2]$, and by direct computation one sees that it has range in $\Cbb\{q_1q_2\}[\log(q_1q_2)]$. If we let $v\mapsto\Tr^\omega(Y_\Mbb(v,1)q^{L(0)})$ denote the pseudo-$q$-trace construction as in \cite{Miy-modular-invariance,AN-pseudo-trace,Fio-genus-1,Hua-modular-C2}, then as linear maps $\Vbb\mapsto\Cbb\{q\}[\log q]$ we have
\begin{align}
\tcboxmath{\Tr^\omega \big(Y_\Mbb(-,1)q^{L(0)}\big)=\MS(\upphi\otimes\Tr^\omega)\big|_{q_1q_2=q}}
\end{align}
Thus, the pseudo-$q$-trace on the LHS is realized by the sewing of conformal blocks on the RHS. In particular, the convergence of Segal's sewing (in the sense of \eqref{eq2}) implies the convergence of pseudo-$q$-traces. See Sec. \ref{lbb63} for the generalization of the above discussions to arbitrary genus.

\begin{itemize}
\item Conclusion: The pseudo-$q$-traces can be recovered from Segal's sewing of conformal blocks for $\Vbb^{\otimes2}$-modules.
\end{itemize}

\subsection{Connections, local freeness, and the convergence of Virasoro uniformization}

Besides Thm. \ref{lbb49}, in this paper, we also prove several results related to the connections on the sheaves. These connections were already used in \cite{Gui-sewingconvergence} to derive the differential equations ensuring the convergence of sewing conformal blocks. In this paper, we adopt the same method. The idea behind this method is to transform the sewing to the uniformization of Riemann surfaces. Huang has already used this idea in \cite{Hua97} to study the convergence of sewing conformal blocks associated to spheres with non-standard local coordinates.

Unlike \cite{Gui-sewingconvergence}, which takes a pragmatic view of these connections, this paper offers a more thorough exploration of the topic for several reasons.

\subsubsection*{The parallelism of sewing up to a projective term}

As suggested in Sec. \ref{lbb62}, in genus-1, the SF theorem (including the convergence of sewing) is insufficient to imply modular invariance. One also needs the flat connections on the locally free sheaf of conformal blocks, and one needs to show that the sewing is parallel under this connection up to a projective term. The general formula for the projective term in arbitrary genus is given in Thm. \ref{lbb52}. As we will show in a future work, this formula accounts for the extra factor $-\frac c{24}$ appearing in the modular invariance theorems.

\subsubsection*{Misunderstanding of the proof of analytic local freeness}

Let us from now on assume that $\fx=(\pi:\MC\rightarrow\MB|\sgm_1,\dots,\sgm_N)$ is a family of $N$-pointed compact Riemann surfaces with local coordinate $\eta_i$ at each $\sgm_i(\MB)$. Associate a $\Vbb^{\otimes N}$-module $\Mbb$ simultaneously to the marked points $\sgm_\blt(\MB)$. A standard application of the connections is to show that the sheaf of conformal blocks $\scr T^*_\fx(\Wbb)$, as an $\MO_\MB$-module, is locally free. (Namely, it is a (finite-rank) vector bundle on $\MB$.) As an application, the dimensions of spaces of conformal blocks are topological invariants. See \cite{TUY,Ueno97,Ueno08,NT-P1_conformal_blocks,DGT2} for instance.

We believe that the following aspect has not received sufficient attention in the literature: \textit{Almost all proofs of the local freeness of sheaves of conformal blocks only apply in the algebraic setting but not in the complex analytic setting}, the setting we take in our series of papers. The algebraic local freeness is proved in the following way: First, it is proved that the sheaf of coinvariants $\scr T_\fx(\Wbb)$ (whose dual sheaf is $\scr T_\fx^*(\Wbb)$) is coherent. Then, since any coherent sheaf with a connection is locally free, it follows that $\scr T_\fx(\Wbb)$ is locally free, and hence $\scr T_\fx^*(\Wbb)$ is locally free. 

However, the above proof that $\scr T_\fx(\Wbb)$ is coherent relies heavily on the fact that $\MB$ is (locally) Noetherian. (See e.g. the proof of \cite[Thm. 6.2.1]{NT-P1_conformal_blocks}.) Without the Noetherian property for $\MO_\MB$, one can only show that $\scr T_\fx(\Wbb)$ is of finite type---that is, locally it has an epimorphism from a free sheaf---but one cannot show that the sheaves of relations in $\scr T_\fx(\Wbb)$ are of finite type, which is essential in defining coherent sheaves. Unfortunately, complex manifolds rarely exhibit the Noetherian property, making it challenging to establish coherence as in the algebraic case. Consequently, the algebraic proof of local freeness fails in the analytic setting. 

Therefore, it is essential to give a correct proof of the local freeness for analytic families, which we achieve in Thm. \ref{lbb35}.

\subsubsection*{Convergence of Virasoro uniformization}

A key step in proving the (analytic) local freeness of $\scr T_\fx(\Wbb)$ (and $\scr T_\fx^*(\Wbb)$) is to prove the convergence of Virasoro uniformization. The algebraic Virasoro uniformization is explicitly stated in \cite[Ch. 17]{FB04}. Roughly speaking, suppose that we have an $N$-pointed compact Riemann surface $\fy=(C|x_1,\dots,x_N)$ with local coordinates $\eta_1,\dots,\eta_N$, where each $\eta_i$ is defined on a disk $U_i$ centered at $x_i$, and choose a meromorphic vector field $\xk_i=h_i\partial_{\eta_i}$ where $h_i$ is a holomorphic function on $U_i\setminus\{x_i\}$ with finite poles at $x_i$. In the algebraic setting, these vector fields $\xk_1,\dots,\xk_N$ generate the first order infinitesimal deformation of $\fy$. Translating $\xk_i$ to the Virasoro operators, we obtain the first order infinitesimal parallel transport of conformal block.

In our analytic setting, we can say more about the Virasoro uniformization. We can even assume that $h_i$ relies holomorphically on another complex variable $q$. So each $\xk_i$ is a non-autonomous vector field, which generates a non-autonomous holomorphic flow $q\mapsto\beta^i_q$ on any compact subset of $U_i-\{x_i\}$. Assume that $U_i$ is large enough and $r>0$ is small enough such that, whenever $|q|<r$, $U_i$ contains $\beta^i_q\circ\eta_i^{-1}(\ovl\MD_1)$. Then we can remove  $\beta^i_q\circ\eta_i^{-1}(\MD_1)$ from $C$, use the map $\eta_i\circ(\beta_q^i)^{-1}$ to glue $\beta_q^i\circ\eta_i^{-1}(\Sbb^1)$ with the boundary $\Sbb^1$ of the standard closed disk $\ovl\MD_1$, and replace $x_i$ and its local coordinates $\eta_i$ with the new marked point $0\in\ovl\MD_1$ and its standard coordinate $z$. Then this gives a new $N$-pointed compact Riemann surface $\fy_q$ with local coordinates.\footnote{In the language of compact Riemann surfaces with boundaries and boundary parametrizations, this process is as follows: At $q$, the Riemann surface with boundary is $C-\bigcup_i\beta^i_q\circ\eta_i^{-1}(\MD_1)$. The $i$-th boundary parametrization, i.e, the analytic diffeomorphism $\beta^i_q\circ\eta_i^{-1}(\Sbb^1)\rightarrow\Sbb^1$, is chosen to be $\eta_i\circ(\beta^i_q)^{-1}$.} See Sec. \ref{lbb67} for more details.

With respect to the analytic deformation $\fy\mapsto \fy_q$, a given conformal block $\uppsi$ (associated to $\fy$ and a $\Vbb^{\otimes N}$-module $\Wbb$) is transformed to a new one $\uppsi_q$. This transformation is described by a differential equation involving Virasoro operators. See \eqref{eqb53} (and also \eqref{eqb56}). Initially, $\uppsi_q$ is only a formal power series of $q$. Our theorem on the \textbf{convergence of Virasoro uniformization} says that for $|q|<r$, $\uppsi_q$ converges absolutely to a conformal block associated to $\fy_q$ and $\Wbb$.

In the special case that $\xk_i$ is autonomous (i.e. independent of $q$), $\beta^i$ is then an autonomous flow, and $\uppsi_q$ can be expressed as an exponential: Write
\begin{align}
\xk_i=\sum_{k=-\infty}^{+\infty}h_{i,k}\eta_i^k
\end{align}
where $h_{i,k}\in\Cbb$ is zero when $k\ll 0$. Then for each $w\in\Wbb$ we have
\begin{align}
\upphi(w)=\upphi_0\big(e^{-qA}w\big)\qquad\text{ where } A=\sum_{i=1}^N \sum_{k=-\infty}^{+\infty}h_{i,k}L_i(k-1)
\end{align}
(See Exp. \ref{lbb32}.) In this case, the convergence of Virasoro uniformization can also be deduced from \cite{Hua97}, so our result provides an alternative proof. Interestingly, our proof of the convergence of Virasoro uniformization shares many similarities with the proof of convergence for sewing conformal blocks.

\subsubsection*{Lie derivatives in sheaves of VOAs}

The most challenging part of constructing the connection  is to show that the connection $\nabla$, originally defined on $\scr W_\fx(\Wbb)=\Wbb\otimes\MO_\MB$, descends to $\scr T_\fx(\Wbb)$. (Note that $\scr T_\fx(\Wbb)$ is the quotient of $\scr W_\fx(\Wbb)$ by a suitable $\MO_\MB$-submodule $\SJ_\fx$.) The purely algebraic approach in \cite{FB04} does not suit our analytic framework optimally, so in this paper, we provide a differential-geometric proof.

Let $\yk$ be a vector field on $\MB$. To show that $\nabla_\yk$ preserves $\SJ_\fx$, we need to calculate the commutator of $\nabla_\yk$ and the action of the sheaf of VOA $\scr V_\fx$ (more precisely, the action of $\scr V_\fx\otimes_{\MO_\MC}\omega_{\MC/\MB}$) on $\scr W_\fx(\Wbb)$. As we will see in Subsec. \ref{lbb66} (especially in \eqref{eqb34}), this commutator is computed using the \textbf{Lie derivatives in \pmb{$\scr V_\fx\otimes\omega_{\MC/\MB}$}}, generalizing the usual Lie derivatives of tensor fields.

It is noteworthy that in the case of $\Vbb$ being a unitary affine VOA associated to a simple Lie algebra $\gk$, this Lie derivative structure is obscured. In fact, as shown in \cite{TUY}, to study the conformal blocks for such VOAs it suffices to use the subsheaf of $\scr V_\fx$ generated by weight one vectors. This subsheaf is canonically isomorphic to $\gk\otimes_\Cbb\Theta_{\MC/\MB}$, and hence its tensor with $\omega_{\MC/\MB}$ is canonically isomorphic to $\gk\otimes_\Cbb\MO_\MC$. Therefore, the usual (Lie) derivative on $\MO_\MC$ is adequate for studying connections. For this reason, extending the treatment of connections in \cite{TUY} to general $C_2$-cofinite VOAs is far from straightforward.

\subsection*{Acknowledgment}

We would like to thank David Ben-Zvi, Yi-Zhi Huang, and Baojun Wu for their valuable discussions. The authors were supported by BMSTC and ACZSP (Grant no. Z221100002722017). The first author would like to thank Zhipeng Yang for the warm hospitality during his visit to Yunnan Key Laboratory of Modern Analytical Mathematics and Applications (No. 202302AN360007).

%\hypertarget{currentwriting}{}

\section{Preliminaries}

\subsection{Notations}

If $X$ is a complex manifold, then we denote the holomorphic tangent (resp. cotangent) bundle as $\Theta_X$ (resp. $\omega_X$). We denote by $\mc O_X$ the sheaf of (germs of) holomorphic functions on $X$. Then $\mc O_X(X)=\mc O(X)$ is the space of holomorphic functions $X\rightarrow\Cbb$. Suppose that $\scr E$ is an $\mc O_X$-module. For each $x\in X$, the \textbf{stalk} of $\scr E$ at $x$ is denoted by $\scr E_x$. If $s\in\scr E(X)$, its \textbf{germ} at $x$ is denoted by $s_x$. The \textbf{fiber} of $\scr E$ at $x$ is the $\Cbb$-vector space
\begin{align}
\scr E|_x:=\scr E_x/\mk_{X,x}\scr E_x
\end{align}  
where $\mk_{X,x}=\{f\in\mc O_{X,x}:f(x)=0\}$. The equivalence class of $s_x$ in $\scr E|_x$ is denoted by $s|_x$.

Suppose that $\pi:X\rightarrow Y$ is a holomorphic map of complex manifolds, and $\scr E$ is a sheaf on $X$. Then $\pi_*\scr E$ is the direct image sheaf, i.e.,
\begin{align*}
(\pi_*\scr E)(V)=H^0(\pi^{-1}(V),\scr E)\equiv \scr E(\pi^{-1}(V))
\end{align*}
for any open $V\subset Y$. If $\scr E$ is an $\scr O_X$-module, then $\pi_*\scr E$ is an $\scr O_Y$-module.

In this article, if $X$ is a complex manifold, and if $\scr E$ and $\scr F$ are $\mc O_X$-modules, then $\scr E\otimes\scr F$ means $\scr E\otimes_{\mc O_X}\scr F$. In other words, $\otimes$ means $\otimes_{\mc O_X}$ but not (say) $\otimes_\Cbb$.

If $s_\blt=(s_i)_{i\in I}$ is a collection in $\scr E(X)$, we say that $s_\blt$ \textbf{generates} $\scr E$ if for each $x\in X$, every element of $\scr E_x$ is a (finite) $\mc O_{X,x}$-linear combination of elements of $s_\blt$. This is equivalent to saying that the $\mc O_X$-module morphism $\bigoplus_{i\in I}\mc O_X\rightarrow \scr E$ sending $\oplus_{i\in I}f_i$ to $\sum_i f_is_i$ is an epimorphism.

We say that $\scr E$ is of \textbf{finite type} if each $x\in X$ is contained in a neighborhood $U$ such that $\scr E|_U$ is generated by a finite subset of $\scr E(U)$. Clearly, if $\scr E$ is of finite type, then each $\scr E|_x$ is finite dimensional.

Suppose that $S$ is a closed submanifold of $X$ with codimension $1$. Then for each $k\in \Zbb$, $\mc O_X(kS)$ is the $\mc O_X$-submodule of $\MO\vert_{X- S}$ consisting of sections of $\MO\vert_{X- S}$ with poles of order $\leq k$ at $S$. Denote 
    \begin{align*}
        \MO_X(\blt S):=\varinjlim_{k\in \Nbb}\MO_X(k S).
    \end{align*}
If $\ME$ is an $\MO_X$-module, we set \index{ES@$\mc E(kS),\mc E(\blt S)$}
    \begin{gather*}
\mc E(kS)=\mc E\otimes\mc O_X(kS)\qquad \mc E(\blt S)=\varinjlim_{k\in \Nbb}\ME(k S)
    \end{gather*}
    If $\mc E$ is locally free (i.e. a vector bundle), then the sections of $\mc E(\blt S)$ (resp. $\mc E(k S)$) can be viewed as sections of $\mc E|_{X- S}$ with finite poles (resp. with poles of orders at most $k$) at $S$.

\begin{itemize}
\item We let $\Nbb=\{0,1,2,3,\dots\}$ and $\Zbb_+=\{1,2,3,\dots\}$.
\item For each $0<r\leq+\infty$ and $r_\blt=(r_1,\dots,r_N)$, let
\begin{gather*}
\mc D_r=\{z\in\Cbb:|z|<r\}\qquad \mc D_r^\times=\{z\in\Cbb:0<|z|<r\}\qquad\ovl{\mc D}_r=\{z\in\Cbb:|z|\leq r\}\\
\MD_{r_\blt}=\MD_{r_1}\times\cdots\times\MD_{r_N}\qquad \MD^\times_{r_\blt}=\MD^\times_{r_1}\times\cdots\times\MD^\times_{r_N}
\end{gather*}

    \item Suppose $\pi:\MC\rightarrow \MB$ is a holomorphic map of complex manifolds. For each $E\subset\MB$, 
    \begin{align*}
 \quad \MC_E:=\pi^{-1}(E)
    \end{align*}

\item We fix an $\Nbb$-graded VOA $\Vbb=\bigoplus_{n\in\Nbb}\Vbb(n)$ with conformal vector $\cbf\in\Vbb(2)$  and central charge $c$. We set $\Vbb^{\leq n}=\bigoplus_{k\leq n}\Vbb(k)$. 

\item On any weak $\Vbb$-module $\Wbb$, the vertex operator $Y_\Wbb(v,z)$ (often abbreviated to $Y$) is written as $Y(v,z)=\sum_{n\in\Zbb}Y(v)_nz^{-n-1}$. More generally, if $\Wbb$ is a weak $\Vbb^{\otimes N}$-module and $v\in\Vbb$, we write 
\begin{gather}
Y_i(v,z)=\sum_n Y_i(v)_nz^{-n-1}=Y(\idt\otimes\cdots\otimes 
\underset{
\begin{subarray}{c}
\uparrow\\
i\text{-th}
\end{subarray}
}{v}
\otimes\cdots\otimes\idt,z)\\
L_i(n)=Y_i(\cbf)_{n-1} \label{eqb93}
\end{gather}
Then $L(n)=\sum_i L_i(n)$ are the Virasoro operators of $\Vbb^{\otimes N}$ on $\Wbb$.
\item Consider the case that $\Wbb$ is a grading-restricted (generalized) $\Vbb^{\otimes N}$-module (cf. \cite{Hua-projectivecover}). Then each generalized eigenspace of $L(0)$ is finite-dimensional and $L_j(0)$-invariant for all $j$. So we can write
\begin{align}\label{eqb68}
\Wbb=\bigoplus_{\lambda_\blt\in\Cbb^N}\Wbb_{[\lambda_\blt]}
\end{align}
where $\Wbb_{[\lambda_\blt]}=\Wbb_{[\lambda_1,\dots,\lambda_N]}$ is the (finite-dimensional) subspace of all $w\in\Wbb$ that is a generalized $\lambda_i$-eigenvector of $L_i(0)$ for each $i$. Define the \textbf{algebraic completion}
\begin{align*}
\ovl\Wbb=\prod_{\lambda_\blt\in\Cbb^N}\Wbb_{[\lambda_\blt]}
\end{align*}
For each $\mu_\blt,\lambda_\blt\in\Cbb^N$, we say $\mu_\blt\leq\lambda_\blt$ whenever $\Re(\mu_i)\leq \Re(\lambda_i)$ for all $1\leq i\leq N$. Let
\begin{gather*}
\Wbb_{[\leq\lambda_\blt]}=\bigoplus_{\mu_\blt\in\Cbb^N,\mu_\blt\leq\lambda_\blt}\Wbb_{[\mu_\blt]}\\
P_{\lambda_\blt}=\text{ the projection }\ovl\Wbb\rightarrow\Wbb_{[\lambda_\blt]}\qquad P_{\leq\lambda_\blt}=\text{ the projection }\ovl\Wbb\rightarrow\Wbb_{[\leq\lambda_\blt]}\label{eqb66}
\end{gather*}
\item Let $W$ be a $\Cbb$-vector space. Let $z_\blt=(z_1,\dots,z_N)$ be (mutually commuting) variables. 
\begin{gather*}
W[[z_\blt]]=\Big\{\sum_{n_\blt\in\Nbb^N}a_{n_\blt}z_\blt^{n_\blt}\Big\}\\
W[z_\blt]=\Big\{\sum_{n_\blt\in\Nbb^N}a_{n_\blt}z_\blt^{n_\blt}:a_{n_\blt}=0\text{ for sufficiently large $n_1,\dots,n_N$}\}\\
W((z_\blt))=\Big\{\sum_{n_\blt\in\Zbb^N}a_{n_\blt}z_\blt^{n_\blt}:a_{n_\blt}=0\text{ for sufficiently negative $n_1,\dots,n_N$}\Big\}\\
W\{z_\blt\}=\Big\{\sum_{n_\blt\in\Cbb^N}a_{n_\blt}z_\blt^{n_\blt}:a_{n_\blt}=0\text{ for sufficiently negative $\Re(n_1),\dots,\Re(n_N)$}\Big\}
\end{gather*}
where all $a_{n_\blt}$ are in $\Wbb$. Note that $W\{z_\blt\}$ is a $\Cbb((z_\blt))$-module.
\end{itemize}

Let $X$ be a locally compact Hausdorff space. Let $(f_i)_{i\in I}$ be a countable collection of functions on $X$. We say that $\sum_i f_i$ \textbf{converges absolutely and locally uniformly (a.l.u.) on $X$} if each $x\in X$ is contained in a neighborhood $U$ such that $\sup_{y\in U}\sum_{i\in I}|f_i(y)|<+\infty$. This is equivalent to that for each compact $K\subset X$ we have $\sup_{y\in K}\sum_{i\in I}|f_i(y)|<+\infty$.

\subsection{Sewing a family of compact Riemann surfaces}\label{lbb41}

\subsubsection{The sewing construction}\label{lbb44}

We revisit the construction of sewing families of compact Riemann surfaces. The setting established in this subsection will be used throughout this paper. Our presentation follows and generalizes the one presented in \cite[Sec. 2]{Gui-sewingconvergence}. In \cite{Gui-sewingconvergence}, the sewing was restricted to a single pair of points, whereas here, we extend it to multiple pairs of points. Much of the geometric construction in this section is well known. We provide detailed explanations to establish the notations that will be used later.

Throughout this article, we fix a family of $(N+2R)$-pointed compact Riemann surfaces 
\begin{align}\label{eqb1}
    \wtd \fx=(\wtd \pi:\wtd \MC\rightarrow \wtd \MB\big|\sgm_\blt\big\Vert \sgm_\blt',\sgm_\blt'')
    =(\wtd \pi:\wtd \MC\rightarrow \wtd \MB\big|\sgm_1,\cdots,\sgm_N\big\Vert \sgm_1',\cdots,\sgm_R',\sgm_1'',\cdots,\sgm_R'').
\end{align}
where $N\geq1$ and $R\geq0$. So $\wtd\pi$ is a proper surjective holomorphic submersion, and each fiber $\wtd\MC_b=\wtd\pi^{-1}(b)$ (where $b\in\wtd\MB$) has complex dimension $1$.  $\sgm_\blt,\sgm_\blt',\sgm'_\blt$ are sections of $\wtd\pi:\wtd\MC\rightarrow\wtd\MB$ \footnote{Namely, they are holomorphic maps $\wtd\MB\rightarrow\wtd\MC$ whose compositions with $\wtd\pi$ are the identity map of $\wtd\MB$.} with mutually disjoint ranges. We always assume that
\begin{gather}\label{eqb52}
\text{For each $b\in\wtd\MB$, each connected component of $\wtd\MC_b$ contains one of $\sgm_1(b),\dots,\sgm_N(b)$}
\end{gather}
which is stronger than (2-b) of \cite[Def. 1.19]{GZ1}. (The reason for assuming this is due to \eqref{eqb51}.) We fix local coordinates $\xi_\blt,\varpi_\blt$ of $\wtd \fx$ at $\sgm_\blt',\sgm_\blt''$. \footnote{Suppose that $\sgm:\wtd\MB\rightarrow\wtd\MC$ is a section. Then a local coordinate at $\sgm(\wtd\MB)$ means a ($\Cbb$-valued) holomorphic function $\eta$ on a neighborhood $\wtd U$ of $\sgm(\wtd\MB)$ whose restriction to each fiber $\wtd U_b=\wtd\pi^{-1}(b)$ (where $b\in\wtd\MB$) is univalent (i.e. injective), and whose restriction to $\sgm(\wtd\MB)$ is zero.}

Choose $r_1,\cdots,r_R,\rho_1,\cdots,\rho_R>0$ and neighborhoods $V_1',\cdots,V_R',V_1'',\cdots,V_R''$ of $\sgm_1'(\wtd \MB),\cdots,\sgm_R'(\wtd \MB),\sgm_1''(\wtd \MB),\cdots,\sgm_R''(\wtd \MB)$, on which $\xi_1,\cdots,\xi_R,\varpi_1,\cdots,\varpi_R$ are defined, such that 
\begin{align}\label{eqb2}
    (\xi_i,\wtd \pi):V_i'\xrightarrow{\simeq} \MD_{r_i}\times \wtd \MB,\quad (\varpi_i,\wtd \pi ):V_i''\xrightarrow{\simeq} \MD_{\rho_i}\times \wtd \MB
\end{align}
are biholomorphic maps for each $1\leq i\leq R$. Moreover, we assume $\sgm_1(\wtd \MB),\cdots,\sgm_N(\wtd\MB),V_1',\cdots,V_R',V_1'',\cdots,V_R''$ are disjoint. Identify
\begin{align}\label{eqb4}
    V_i'=\MD_{r_i}\times \wtd \MB,\qquad V_i''=\MD_{\rho_i}\times \wtd \MB
\end{align}
via \eqref{eqb2}. Then $\xi_i$ (resp. $\varpi_i$) becomes the projection onto the $\MD_{r_i}$-component (resp. $\MD_{\rho_i}$-component), and $\wtd \pi$ restricts to the projection onto $\wtd \MB$. 
%Set $q_i=\xi_i \varpi_i:\MD_{r_i}\times \MD_{\rho_i}\rightarrow \MD_{r_i \rho_i}$. 
We call $r_i,\rho_i$ the \textbf{sewing radii} for $\sgm_i',\sgm_i''$.

We \textbf{sew \pmb{$\wtd \fx$} along pairs of points} $\sgm_\blt',\sgm_\blt''$ using $\xi_\blt,\varpi_\blt$ to get a family
\begin{align}\label{eqb3}
    \fx=(\pi:\MC\rightarrow \MB\big| \sgm_\blt)=(\pi:\MC\rightarrow \MB\big|\sgm_1,\cdots,\sgm_N)
\end{align}
$\fx$ is described as follows. 

We shall freely switch the orders of Cartesian products. Set 
\begin{align}
    \MB=\MD_{r_\blt \rho_\blt}\times \wtd \MB=\MD_{r_1 \rho_1}\times \cdots \times \MD_{r_R \rho_R}\times \wtd \MB.
\end{align}
Define also $W_i$ and its open subsets $W_i',W_i''$ by
\begin{subequations}\label{eqb5}
\begin{gather}
\label{eqWi1} W_i=\MD_{r_i}\times\MD_{\rho_i}\times \prod_{j\ne i} \MD_{r_j\rho_j}\times \wtd \MB\\ 
\label{eqWi2} W_i'=\MD_{r_i}^\times \times \MD_{\rho_i}\times \prod_{j\ne i} \MD_{r_j\rho_j}\times \wtd \MB\\  
\label{eqWi3} W_i''=\MD_{r_i} \times \MD_{\rho_i}^\times\times \prod_{j\ne i} \MD_{r_j\rho_j}\times \wtd \MB
\end{gather}
\end{subequations}
Then we can extend $\xi_i,\varpi_i$ and $q_i=\xi_i\varpi_i$ constantly to 
\begin{subequations}\label{eqb6}
\begin{gather}
    \xi_i:W_i\rightarrow \MD_{r_i}  \qquad (z,w,*)\mapsto z\\
    \varpi_i:W_i\rightarrow \MD_{\rho_i}\qquad (z,w,*)\mapsto w\\
q_i:W_i\rightarrow \MD_{r_i\rho_i} \qquad (z,w,*)\mapsto zw
\end{gather}
\end{subequations}
so that $q_i=\xi_i\varpi_i$ still holds. Then we have open holomorphic embeddings
\begin{subequations}\label{eqb94}
\begin{gather}
(\xi_i,\varpi_i,\pr):W_i\xrightarrow{=} \MD_{r_i}\times \MD_{\rho_i}\times \prod_{j\ne i} \MD_{r_j\rho_j}\times \wtd \MB  \label{eqb7}\\
(\xi_i,q_i,\pr):W_i'\rightarrow \MD_{r_i}\times \MD_{r_i\rho_i}\times \prod_{j\ne i} \MD_{r_j\rho_j}\times \wtd \MB \simeq \MD_{r_i}\times \wtd\MB \times \MD_{r_\blt \rho_\blt}   \label{eqb8}\\
(\varpi_i,q_i,\pr):W_i''\rightarrow \MD_{\rho_i}\times \MD_{r_i\rho_i}\times \prod_{j\ne i} \MD_{r_j\rho_j}\times \wtd \MB \simeq \MD_{\rho_i}\times \wtd\MB \times \MD_{r_\blt \rho_\blt}     \label{eqb9}
\end{gather}
\end{subequations}
where $\pr$ denotes the projection onto the $\prod_{j\ne i} \MD_{r_j\rho_j}\times \wtd \MB $-component. The image of \eqref{eqb8} resp. \eqref{eqb9} is precisely the subset of all $(z_i,b,p_1,\dots,p_R)\in \MD_{r_i}\times \wtd\MB\times \MD_{r_\blt\rho_\blt}$ resp. $(w_i,b,p_1,\dots,p_R)\in \MD_{\rho_i}\times \wtd\MB\times \MD_{r_\blt\rho_\blt}$ satisfying 
\begin{align}\label{eqb84}
\frac{\vert p_i\vert }{\rho_i}<\vert z_i\vert <r_i \qquad \text{resp.}\qquad \frac{\vert p_i\vert }{r_i}<\vert w_i\vert <\rho_i.
\end{align}
So closed subsets $F_i'\subset \MD_{r_i}\times \wtd\MB\times \MD_{r_\blt\rho_\blt}$ and $F_i''\subset \MD_{\rho_i}\times \wtd\MB\times \MD_{r_\blt\rho_\blt}$ can be chosen such that we have biholomorphisms
\begin{subequations}
\begin{gather}\label{geosewb1}
(\xi_i,q_i,\pr):W_i'\xrightarrow{\simeq} \MD_{r_i}\times \wtd\MB\times \MD_{r_\blt\rho_\blt}-F_i'\\
 (\varpi_i,q_i,\pr):W_i''\xrightarrow{\simeq} \MD_{\rho_i}\times \wtd\MB\times \MD_{r_\blt\rho_\blt}-F_i''
\end{gather}
\end{subequations}
By the identifications \eqref{eqb4}, we can write the above maps as
\begin{subequations}\label{eqb10}
\begin{gather}
(\xi_i,q_i,\pr):W_i'\xrightarrow{\simeq} V_i'\times \MD_{r_\blt\rho_\blt}-F_i'\qquad \subset\wtd \MC\times\mc D_{r_\blt\rho_\blt}\\
 (\varpi_i,q_i,\pr):W_i''\xrightarrow{\simeq} V_i''\times \MD_{r_\blt\rho_\blt}-F_i''\qquad \subset \wtd \MC\times\mc D_{r_\blt\rho_\blt}
\end{gather}
\end{subequations}
In particular, we view $F_i'$ and $F_i''$ as closed subsets of $\wtd \MC\times \mc D_{r_\blt\rho_\blt}$.

The complex manifold $\MC$ is defined by 
\begin{gather}
\MC=\big(W_1\sqcup\cdots\sqcup W_R\big)\bigsqcup \big(\wtd \MC\times \MD_{r_\blt \rho_\blt}-\bigcup_{i=1}^R( F_i'\cup F_i'')\big)\Big/\sim
\end{gather}
Here, the equivalence $\sim$ is defined by identifying each subsets $W_i',W_i''$ of $W_i$ with the corresponding open subsets of $\wtd \MC\times \MD_{r_\blt \rho_\blt}-\bigcup_{i=1}^R( F_i'\cup F_i'')$ via the biholomorphisms \eqref{eqb10}.

$\pi:\MC\rightarrow \MB$ is defined as follows. The map
\begin{align*}
    \wtd \pi\times \id_{\MD_{r_\blt\rho_\blt}}:\wtd \MC\times \MD_{r_\blt \rho_\blt}\rightarrow \wtd \MB\times \MD_{r_\blt \rho_\blt}=\MB
\end{align*}
agrees with
\begin{align*}
q_i:W_i=\MD_{r_i}\times \MD_{\rho_i}\times \prod_{j\ne i} \MD_{r_j\rho_j}\times \wtd \MB\rightarrow \MD_{r_\blt \rho_\blt}\times \wtd \MB=\MB
\end{align*}
when restricted to $W_i^\prime$ and $W_i''$. These two maps give a well-defined surjective holomorphic map $\pi:\MC\rightarrow \MB$.

Extend $\sgm_i$ constantly to $\MB=\wtd \MB\times \MD_{r_\blt \rho_\blt}\rightarrow \wtd \MC\times \MD_{r_\blt \rho_\blt}$, whose image is disjoint from $F_i^\prime$ and $F_i''$ for $1\leq i\leq R$. So $\sgm_i$ can be extended to sections of $\pi:\MC\rightarrow \MB$, also denoted by $\sgm_i$. This together with $\pi:\MC\rightarrow \MB$ gives \eqref{eqb3}. 

\begin{rem}\label{lbb9}
If we choose local coordinates $\eta_i$ defined on a neighborhood $\wtd U_i$ of $\sgm_i(\wtd \MB)$ disjoint from $V_1',\cdots,V_R',V_1'',\cdots,V_R''$, then $\eta_i$ can be extended constantly to neighborhoods $U_i:=\wtd U_i\times \MD_{r_\blt \rho_\blt}$ of $\sgm_i(\MB)$, also denoted by $\eta_i$. 
\end{rem}

\begin{comment}
Just as \cite[Remark 1.7.2]{GZ1}, the fiber of $\fx$ is denoted by
\begin{subequations}

where $\MC_b=\pi^{-1}(b)$ is not necessarily smooth. If we choose local coordinates $\eta_i$ at $\sgm_i(\MB)$ for each $1\leq i\leq N$, then we have a family of $N$-pointed nodal curves with local coordinates

\end{subequations}
\end{comment}

\begin{df}
The set
\begin{align*}
\Sigma=\{x\in \mc C:\pi\text{ is not a submersion at }x\}
\end{align*}
is called the \textbf{critical locus} of $\fk X$. Write
\begin{align}
W=\bigsqcup_{i=1}^R W_i \qquad   W'=\bigsqcup_{i=1}^R W_i'\qquad W''=\bigsqcup_{i=1}^R W_i''
\end{align}
It is not hard to see that $\pi$ is a submersion outside $W$, and for each $i$ we have
\begin{align}
W_i\cap\Sigma= \big(\{0\}\times\{0\}\big)\times\prod_{j\ne i} \MD_{r_j\rho_j}\times \wtd \MB\qquad \subset\MD_{r_i}\times \MD_{\rho_i}\times \prod_{j\ne i} \MD_{r_j\rho_j}\times \wtd\MB
\end{align}
Thus, we have
\begin{align}
\Sigma=W-(W'\cup W'')=\bigsqcup_{i=1}^R(W_i-(W_i'\cup W_i''))  \label{eqx}
\end{align}
It is clear that the \textbf{discriminant locus} $\Delta=\pi(\Sigma)$\index{00@Discriminant locus $\Delta=\pi(\Sigma)$} satisfies
\begin{align}\label{eqb45}
\Delta\xlongequal{\mathrm{def}}\pi(\Sigma)=\{(p_1,\dots,p_R,\wtd b)\in\mc D_{r_\blt\rho_\blt}\times \wtd\MB:p_1\cdots p_R=0\}=(\mc D_{r_\blt\rho_\blt}-\mc D_{r_\blt\rho_\blt}^\times)\times \wtd \MB
\end{align}
\end{df}

So $\Delta$ is the set of all $b\in \MB$ such that $\MC_b$ is not smooth. If $b\in\MB-\Delta$, we set
\begin{align*}
    \fx_b:=(\MC_b\big|\sgm_1(b),\cdots,\sgm_N(b))
\end{align*}
If the local coordinates $\eta_\blt$ are chosen as in Rem. \ref{lbb9}, we set
\begin{align*}
    \fx_b:=(\MC_b\big|\sgm_1(b),\cdots,\sgm_N(b);\eta_1\vert_{\MC_b},\cdots,\eta_N \vert_{\MC_b}).
\end{align*}
More generally, if $U$ is an open or a closed complex submanifold of $\MB$, we let
\begin{align*}
\fx_U=\text{the restriction of $\fx$ to $\MC_U\rightarrow U$}
\end{align*}

\begin{ass}\label{lbb1}
Throughout this article, unless otherwise stated, \textit{we always assume that $\fx=\eqref{eqb3}$ is the family obtained by sewing the smooth family $\wtd \fx=\eqref{eqb1}$ along pairs of points $\sgm_\blt',\sgm_\blt''$ using $\xi_\blt,\varpi_\blt$}. 
\begin{comment}
Moreover, we choose local coordinates $\eta_1,\cdots,\eta_N$ at $\sgm_1(\wtd \MB),\cdots,\sgm_N(\wtd \MB)$ of $\wtd \fx$. Extend $\eta_\blt$ constantly to local coordinates $\eta_\blt$ of $\fx$. Suppose $\eta_1,\cdots,\eta_N$ are defined on disjoint neighborhoods $U_1,\cdots,U_N$ of $\sgm_1(\MB),\cdots,\sgm_N(\MB)$. 
\end{comment}
Define $1$-codimensional closed submanifolds
\begin{gather*}
    S_\fx=\bigcup_{i=1}^N\sgm_i(\MB)\qquad
    S_{\wtd \fx}=\bigcup_{i=1}^N\sgm_i(\wtd\MB)\cup\bigcup_{j=1}^R\sgm_j'(\wtd\MB)\cup\bigcup_{j=1}^R\sgm_j''(\wtd \MB)
\end{gather*}
of $\MC$ and $\wtd\MC$ respectively. Note that $\fk X$ is \textbf{smooth} if and only if $R=0$, if and only if $\fk X=\wtd {\fk X}$. If $E\subset\MB$, we let $\fx_E$ denote the restriction of $\fx$ to $E$, i.e.
\begin{align*}
\fx_E=\big(\pi:\MC_E\rightarrow\MB_E\big|\sgm_1|_E,\dots,\sgm_N|_E\big)
\end{align*}
If local coordinates $\eta_1,\dots,\eta_N$ of $\fx$ at $\sgm_1(\MB),\dots,\sgm_N(\MB)$ are chosen, we set
\begin{align*}
\fx_E=\big(\pi:\MC_E\rightarrow\MB_E\big|\sgm_1|_E,\dots,\sgm_N|_E;\eta_1|_{\MC_E},\dots,\eta_N|_{\MC_E}\big)
\end{align*}
\end{ass}

Let $\Delta,\Sigma$ be the critical locus and discriminant locus of $\fx$. Then the union of non-smooth fibers is therefore
\begin{align*}
    \MC_\Delta=\pi^{-1}(\Delta)
\end{align*}

Recall that if $\scr E$ is a locally free $\mc O_\MC$-module (i.e., a holomorphic vector bundle on $\MC$), then $\scr E(\blt\SX)$ denotes the sheaf of meromorphic sections of $\scr E$ with only possible poles at $\SX$. If $k\in\Nbb$, then $\scr E(k\SX)$ denotes the sheaf of sections of $\scr E(\blt\SX)$ whose orders of poles at $\SX$ are at most $k$.

\subsubsection{The relative tangent and dualizing sheaf}\label{lbb46}
Recall that $\Theta_\MB$ and $\Theta_\MC$ are the sheaves of holomorphic tangent fields of $\MB$ resp. $\MC$. In this subsection, we recall the construction of $\Theta_\MB(-\log \Delta)$ and $\Theta_\MC(-\log \MC_\Delta)$. These are the sheaves of sections of $\Theta_\MB$ resp. $\Theta_\MC$ tangent to $\Delta$ resp. $\MC_\Delta$. Then the differential $d\pi:\Theta_\MC\rightarrow \pi^*\Theta_\MB$ restricts to an $\MO_\MC$-module epimorphism
\begin{align}\label{eqb13}
    d\pi:\Theta_\MC(-\log \MC_\Delta)\rightarrow \pi^* \Theta_\MB(-\log \Delta).
\end{align}
Moreover, we will use \eqref{eqb13} to construct a short exact sequence, which is crucial to the definition of connections.

Assume $\wtd \MB$ is a Stein manifold and is small enough to admit a coordinate $\tau_\blt=(\tau_1,\cdots,\tau_m):\wtd \MB\rightarrow \Cbb^m$. Let $q_i$ be the standard coordinate of $\MD_{r_i \rho_i}$ for $1\leq i\leq R$. Then $(q_\blt,\tau_\blt)$ becomes a coordinate of $\MB$.
\begin{df}\label{lbb11}
    $\Theta_\MB(-\log \Delta)$\index{zz@$\Theta_\MB(-\log \Delta)$} is defined as the $\MO_\MB$-submodule of $\Theta_\MB$ generated freely by 
    \begin{align}\label{eqb14}
        q_1\partial_{q_1},\cdots,q_R \partial_{q_R},\partial_{\tau_1},\cdots,\partial_{\tau_m}.
    \end{align}
\end{df}
For simplicity, the pullback of \eqref{eqb14} under $\pi$ in $\pi^* \Theta_\MB$ will also be denoted as $q_1\partial_{q_1},\cdots,q_R \partial_{q_R},\partial_{\tau_1},\cdots,\partial_{\tau_m}$.

We now describe $\Theta_\MC(-\log \MC_\Delta)$\index{zz@$\Theta_\MC(-\log \MC_\Delta)$} and the morphism \eqref{eqb13} locally. Choose $x\in \MC$.
\begin{enumerate}[align=left]
    \item[\textbf{Case 1.}] When $x\notin \Sigma$, $x$ can be regarded as a point $(\wtd x,p)$ of $\wtd \MC\times \MD_{r_\blt \rho_\blt}$ disjoint from $F_i',F_i''$ for all $1\leq i\leq R$. Choose a neighborhood $\wtd U\subset \wtd \MC$ of $\wtd x$ together with $\eta\in \MO(\wtd U)$ univalent on each fiber of $\wtd U$. Choose a neighborhood $V$ of $p\in \MD_{r_\blt \rho_\blt}$ such that $U:=\wtd U\times V$ is disjoint from $F_i',F_i''$ for all $1\leq i\leq R$. Write $\tau_\blt \circ \wtd \pi$ as $\tau_\blt$ for short. Extend $\eta$ constantly to $U$ so that $\eta\in\mc O(U)$ is univalent on each fiber. Then $(\eta,\tau_\blt,q_\blt)$ becomes a set of coordinates of $U$. $\Theta_\MC(-\log \MC_\Delta)\vert_{U}$ is then the $\MO_U$-submodule of $\Theta_\MC\vert_U$ generated (automatically freely) by 
    \begin{align}\label{eqbb15}
        \partial_\eta,\partial_{\tau_1},\cdots,\partial_{\tau_m},q_1\partial_{q_1},\cdots,q_R \partial_{q_R}.
    \end{align}
    The restriction of \eqref{eqb13} to $U$ (which is also a restriction of $d\pi:\Theta_\MC\rightarrow\pi^*\Theta_\MB$) sends $\partial_\eta$ to $0$ and keeps the other elements of \eqref{eqbb15}. 
    
    \item[\textbf{Case 2.}] When $x\in \Sigma$, by \eqref{eqx}, we can find $1\leq i\leq R$, such that $W_i$ is a neighborhood of $x$. By \eqref{eqWi1}, $(\xi_i,\varpi_i,q_1,\cdots,\wht {q_i},\cdots,q_R,\tau_\blt)$ is a set of coordinates of $W_i$. $\Theta_\MC(-\log \MC_\Delta)\vert_{W_i}$ is then the $\MO_{W_i}$-submodule of $\Theta_\MC \vert_{W_i}$ generated (automatically freely) by 
    \begin{align}\label{eqbb16}
        \xi_i \partial_{\xi_i},\varpi_i \partial_{\varpi_i},q_1\partial_{q_1},\cdots,\widehat{ q_i\partial_{q_i}},\cdots,q_R\partial_{q_R},\partial_{\tau_1},\cdots,\partial_{\tau_m}.
    \end{align}
    As one can check, the restriction of \eqref{eqb13} (which is also a restriction of $d\pi:\Theta_\MC\rightarrow\pi^*\Theta_\MB$) to $W_i$ satisfies
    \begin{align}\label{eq174}
        d\pi(\xi_i\partial_{\xi_i})=d\pi (\varpi_i \partial_{\varpi_i})=q_i\partial_{q_i}
    \end{align}
and keeps the other elements of \eqref{eqbb16}.
\end{enumerate}
\begin{rem}
To see \eqref{eq174}, note that $\pi$ pulls the cotangent vector $dq_i$ back  to $d(\xi_i\varpi_i)=\xi_id\varpi_i+\varpi_id\xi_i$. So the dual element $q_i^{-1}dq_i$ of $q_i\partial_{ q_i}$ is pulled back to $\varpi_i^{-1}d\varpi_i+\xi_i^{-1}d\xi_i$. The latter is the summation of the dual elements of $\varpi_i\partial_{\varpi_i},\xi_i\partial_{\xi_i}$.
\end{rem}

Combining Case 1 and Case 2 together, we complete the description of $\Theta_\MC(-\log \MC_\Delta)$ and the morphism \eqref{eqb13}.
\begin{df}\label{defb1}
    The kernel of \eqref{eqb13} is called the \textbf{relative tangent sheaf} and is denoted by $\Theta_{\MC/\MB}$.\index{zz@$\Theta_{\MC/\MB}$, the relative tangent sheaf} It is a line bundle, since in Case 1 resp. Case 2 above, $\Theta_{\MC/\MB}|_U$ resp. $\Theta_{\MC/\MB}|_{W_i}$ is generated freely by the section
\begin{align*}
\partial_\eta\qquad\text{resp.}\qquad \xi_i\partial_{\xi_i}-\varpi_i\partial_{\varpi_i}
\end{align*}
The dual $\MO_\MC$-module of $\Theta_{\MC/\MB}$ is called the \textbf{relative dualizing sheaf} and is denoted by $\omega_{\MC/\MB}$.\index{zz@$\omega_{\MC/\MB}$, relative dualizing sheaf} When $b\in \mc B-\Delta$, we have canonical isomorphisms
\begin{align*}
\Theta_{\MC/\MB}|_{\MC_b}\simeq \Theta_{\MC_b}\qquad \omega_{\MC/\MB}|_{\MC_b}\simeq \omega_{\MC_b}
\end{align*}
where $ \Theta_{\MC_b}$ (resp. $\omega_{\MC_b}$) is the holomorphic tangent (resp. cotangent) line bundle of the Riemann surface $\MC_b$.
\end{df}

\begin{rem}\label{lbb4}
Similar to $\Theta_{\MC/\MB}$, we have a local description of the line bundle $\omega_{\MC/\MB}$. If $U\subset\mc C-\Sigma$ is open and $\eta\in\mc O(U)$ is univalent on each fiber, then
\begin{align*}
\text{$\omega_{\MC/\MB}|_U$ is freely generated by $d\eta$}
\end{align*}
where $d\eta$ is the dual element of the tangent field $\partial_\eta$ (orthogonal to $d\pi$). If $\mu\in\mc O(\wtd U)$ is also univalent on each fiber, the transition function is
\begin{align}\label{eqb74}
d\eta=(\partial_\mu\eta)\cdot d\mu
\end{align} 
This describes $\omega_{\MC/\MB}|_{\MC-\Sigma}$. One the other hand, $\omega_{\MC/\MB}|_{W_i}$ is generated freely by the element of $H^0(W_i-\Sigma,\omega_{\MC/\MB}|_{\MC-\Sigma})$ whose restrictions to $W_i'$ and $W_i''$ are
\begin{align}\label{eqb75}
\xi_i^{-1}d\xi_i\qquad\text{resp.}\qquad-\varpi_i^{-1}d\varpi_i
\end{align}
This section is well-defined since the two of \eqref{eqb75} agree on $W_i'\cap W_i''$ due to \eqref{eqb74}.
\end{rem}

\begin{comment}
When restricted to $W_i$, $\omega_{\MC/\MB}\vert_{W_i}$ is generated freely by the section whose restriction to $W_i'$ (resp. $W_i''$) is
\begin{align}
    \xi_i^{-1}d\xi_i \quad \text{resp.}\quad -\varpi_i^{-1}d\varpi_i. 
\end{align}
This is because, according to the description of $d\pi$ in Case 2 above, $\Theta_{\MC/\MB}(-\log\mc C_\Delta)|_{W_i}$ is generated freely by the section $\xi_i\partial_{\xi_i}-\varpi_i\partial_{\varpi_i}$. See \cite{Gui-sewingconvergence} the paragraph around Eq. (4.22) for details.
\end{comment}

\begin{rem}\label{lbb19}
By Def. \ref{defb1} and the surjectivity of \eqref{eqb13}, we have a short exact sequence of $\MO_\MC$-module
\begin{align}\label{shortexact}
    0\rightarrow \Theta_{\MC/\MB}\rightarrow \Theta_\MC(-\log \MC_\Delta)\xrightarrow{d\pi}\pi^* \Theta_\MB(-\log \Delta)\rightarrow 0.
\end{align}
When $\fx$ is a smooth family, \eqref{shortexact} becomes
\begin{align}
    0\rightarrow \Theta_{\MC/\MB}\rightarrow \Theta_\MC\xrightarrow{d\pi}\pi^* \Theta_\MB\rightarrow 0. \label{eqb48}
\end{align}
\end{rem}
%In the remaining part of this section, we always assume $\fx$ is obtained by sewing the smooth family $\wtd\fx$ unless otherwise stated. Notations in this subsection will be frequently used later.

\subsection{Sheaves of conformal blocks and coinvariants for the sewn family}\label{lbb70}

Recall \cite[Sec. 1.2]{GZ1} for the definitions of admissible $\Vbb^{\times N}$-modules and finitely-admissible $\Vbb^{\times N}$-modules. Roughly speaking, an admissible $\Vbb^{\times N}$-module $\Wbb$ is one having mutually commuting actions of $N$ pieces of $\Vbb$, together with simultaneously diagonalizable operators $\wtd L_1(0),\dots,\wtd L_N(0)$ compatible with the actions of the $N$ pieces of $\Vbb$. Moreover, $\wtd L_\blt(0)$ gives an $\Nbb^N$-grading $\Wbb=\bigoplus_{n_\blt\in\Nbb^N}\Wbb(n_\blt)$. If each $\Wbb(n_\blt)$ is finite dimensional, then $\Wbb$ is called finitely-admissible.

\begin{rem}
It is important to keep in mind that if $\Vbb$ is $C_2$-cofinite, then a finitely-admissible $\Vbb^{\times N}$-module is equivalently a grading-restricted (generalized) $\Vbb^{\otimes N}$-module, cf. \cite[Thm. A.14]{GZ1}. We will talk about grading-restricted $\Vbb^{\otimes N}$-modules instead of finitely-admissible $\Vbb^{\times N}$ when the choice of $\wtd L_\blt(0)$ is irrelevant. See e.g. Rem. \ref{lbb20}.
\end{rem}

Fix an admissible $\Vbb^{\times N}$-module $\Wbb$. Associate $\Wbb$ to $\sgm_\blt$. In \cite[Sec. 2.2]{GZ1}, we have defined sheaves of partial conformal blocks for families of $(M,N)$-pointed compact Riemann surfaces. Specializing to $M=0$, we obtain sheaves of conformal blocks as in \cite[Sec. 1.7]{GZ1}. This notion can be generalized in a straightforward way to the family $\fk X$ obtained by sewing $\wtd\fx$. The key ingredient is the generalization of sheaves of VOAs for smooth families, as considered in \cite[Sec. 1.7]{GZ1}, to the family $\fx$. Such sheaves of VOAs have already been constructed in \cite{DGT2} (in the algebraic setting) and \cite[Sec. 5]{Gui-sewingconvergence}. Let us briefly recall its definition following \cite{Gui-sewingconvergence} and \cite{GZ1}.

\subsubsection{Sheaves of VOAs}\label{mc2}
Recall that $\Gbb$ is the automorphism group of the set of local coordinates of $\Cbb$ at $0$. Namely, the elements of $\Gbb$ are $f(z)=\sum_{n>0}a_nz^n$ where $a_n\in\Cbb$, $a_1\neq0$, and $f(z)$ converges absolutely for sufficiently small $z$. The group multiplication is defined by the composition of functions. If $X$ is a complex manifold, a map $\rho:X\rightarrow\Gbb$ is called a \textbf{holomorphic family of transformations} if $\rho$ can be viewed as a holomorphic function on a neighborhood of $X\times\{0\}$ in $X\times\Cbb$.  Define
\begin{align*}
\scr V_{\fk X}=\varinjlim_{n\in\Nbb}\scr V_{\fk X}^{\leq n}
\end{align*}
where each $\scr V_{\fk X}^{\leq n}$ is an $\mc O_{\mc C}$-module  defined as follows. Since $\pi:\mc C-\Sigma\rightarrow\mc B$ is a submersion, the sheaf $\scr V_{\fk X-\Sigma}^{\leq n}$ is defined as in \cite[Sec. 5]{Gui-sewingconvergence} and \cite[Sec. 1.6]{GZ1}. %Then $\scr V_{\fk X}^{\leq n}$ is an $\mc O_{\mc C}$-submodule of $\scr V_{\fk X-\Sigma}^{\leq n}$ which agrees with $\scr V_{\fk X-\Sigma}^{\leq n}$ outside $\Sigma$. 
More precisely, if $U\subset\MC-\Sigma$ is open and $\eta\in\mc O(U)$ is univalent on each fiber $U_b=\pi^{-1}(b)\cap U$ (where $b\in\MB$), we have a trivialization (cf. \cite[Eq. (2.5)]{Gui-sewingconvergence} or \cite[Sec. 1.6]{GZ1}):
\begin{gather}\label{eq175}
\mc U_\varrho(\eta):\scr V^{\leq n}_\fx|_U\xlongrightarrow{\simeq}\Vbb^{\leq n}\otimes_\Cbb\mc O_U
\end{gather}
If $\mu\in\mc O(U)$ is also univalent on each fiber, the transition function is
\begin{align}\label{eqb22}
\mc U_\varrho(\eta)\mc U_\varrho(\mu)^{-1}=\mc U(\varrho(\eta|\mu)):\Vbb^{\leq n}\otimes_\Cbb\mc O_U\xlongrightarrow{\simeq}\Vbb^{\leq n}\otimes_\Cbb\mc O_U
\end{align}
Here, $\varrho(\eta|\mu):U\rightarrow\Gbb$ is a holomorphic family such that for each $p\in U$,  
\begin{align}\label{eqb39}
\text{$\varrho(\eta|\mu)_p$ transforms the local coordinate $(\mu-\mu(p))|_{\mc C_{\pi(p)}}$ to $(\eta-\eta(p))|_{\mc C_{\pi(p)}}$}
\end{align}
In other words,
\begin{align}\label{eqb40}
\varrho(\eta|\mu)_p(z)=\eta\circ\big(\mu,\pi)^{-1}(z+\mu(p),\pi(p)\big)-\eta(p)
\end{align}
And $\mc U(\varrho(\eta|\mu))$ associates to each $p$ the change of coordinate operator $\mc U(\varrho(\eta|\mu)_p)\in\mathrm{Aut}(\Vbb)$ (restricting to $\mathrm{Aut}(\Vbb^{\leq n})$); more precisely, if we set $\alpha=\varrho(\eta|\mu)_p$, and let $c_1,c_2,\dots\in\Cbb$ be the (necessarily unique) scalars such that $\alpha(z)=\alpha'(0)\cdot\exp\Big(\sum_{n>0}c_nz^{n+1}\partial_z\Big)z$, then on $\Vbb$ resp. $\Vbb^{\leq n}$ we have
\begin{align}
\mc U(\alpha)=\alpha'(0)^{L(0)}\cdot\exp\Big(\sum_{n>0}c_nL(n)\Big)
\end{align}
The operator $\mc U(\alpha)$ is due to Huang \cite{Hua97}; $\mc U$ is a group homomorphism.

\begin{rem}
With abuse of notations, we also denote the tensor product of \eqref{eq175} and the identity map of $\omega_{\mc C/\mc B}$ by
\begin{align}\label{eqb43}
\mc U_\varrho(\eta):\scr V^{\leq n}_\fx\otimes \omega_{\mc C/\mc B}|_U\xlongrightarrow{\simeq}\Vbb^{\leq n}\otimes_\Cbb\mc O_U\otimes_\Cbb d\eta
\end{align}
Namely, it sends $v\otimes d\eta$ to $\mc U_\varrho(\eta)v\otimes_\Cbb d\eta$.
\end{rem}

To define $\scr V_{\fk X}^{\leq n}$ near $\Sigma$, it suffices to describe its restriction to each $W_i$. Recall $W_i-\Sigma=W_i'\cup W_i''$ by \eqref{eqx}. Recall from \cite[Sec. 1.3]{GZ1} the change of coordinates
\begin{gather}\label{eqb87}
\upgamma_z(t)=\frac 1{z+t}-\frac 1z\qquad \mc U(\upgamma_z)=e^{zL(1)}(-z^{-2})^{L(0)}
\end{gather}

\begin{df}\label{lbb45}
$\scr V_{\fk X}^{\leq n}|_{W_i}$ is the (automatically free) $\MO_{W_i}$-submodule of $\SV_{\fx-\Sigma}^{\leq n}\vert_{W_i-\Sigma}$ generated by the sections whose restrictions to $W_i^\prime$ and $W_i''$ are 
\begin{align}\label{geosewb7}
    \MU_\varrho(\xi_i)^{-1}(\xi_i^{L(0)}v) \qquad \text{resp.}\qquad \MU_\varrho(\varpi_i)^{-1}(\varpi_i^{L(0)}\MU(\upgamma_1)v)
\end{align}
where $\xi_i,\varpi_i$ are defined by \eqref{eqb6} and $v\in \Vbb^{\leq n}$. This is well-defined (i.e. the two expressions in \eqref{geosewb7} agrees on $W_i'\cap W_i''$). See \cite{Gui-sewingconvergence} Sec. 5, especially Lemma 5.2. 
\end{df}

\subsubsection{Sheaves of conformal blocks and coinvariants}\label{lbb10}

Similar to  \cite[Def. 2.11]{GZ1}, we have a locally free $\MO_\MB$-module $\SW_\fx(\Wbb)$ defined by the trivialization
\begin{align*}
\mathcal U(\eta_\blt):\scr W_\fx(\Wbb)|_V\xlongrightarrow{\simeq}\Wbb\otimes_\Cbb\mc O_V
\end{align*}
associated $\eta_\blt$ if $\eta_\blt$ are local coordinates of $\fx_V$ at $\sgm_\blt(V)$ where $V\subset\MB$ is open. Readers unfamiliar with this trivialization can assume throughout the paper that $\fx$ admits local coordinates $\eta_\blt$ at $\sgm_\blt(\MB)$, allowing $\scr W_\fx(\Wbb)$ to be identified with $\Wbb\otimes_\Cbb\mc O_\MB$. In fact, in this paper, except for a very few places (such as Sec. \ref{lbb69}), we assume that $\fx$ has local coordinates.

Similar to \cite[Def. 2.15]{GZ1}, we also have the residue action of $\pi_*\big(\SV_{\fx}\otimes \omega_{\MC/\MB}(\blt S_\fx)\big)$ on $\SW_\fx(\Wbb)$ which is independent of the choice of local coordinates. Let us recall its definition in the case that the local coordinates $\eta_\blt$ have been chose.

Choose any connected open $V\subset\MB$. Let $U_i\subset\MC_V$ be a neighborhood of $\sgm_i(V)$ on which $\eta_i$ is defined. For each $n\in\Nbb$ and each $\sigma\in H^0(U_i,\scr V_\fx^{\leq n}\otimes\omega_{\MC/\MB}(\blt\SX))$ and $w\in\Wbb\otimes_\Cbb\MO(V)$, we define the \textbf{$i$-th residue action}
\begin{align}
\sigma*_i w=\Res_{\eta_i=0}Y_i(\MU_\varrho(\eta_i)\sigma,\eta_i)w
\end{align}
which is in $\Wbb\otimes_\Cbb\MO(V)$. More precisely, note that
the element $\mc U(\eta_i)\sigma\in H^0(U_i,\Vbb^{\leq n}\otimes_\Cbb\omega_{\MC/\MB}(\blt\SX))$ is a finite sum $\sum_l v_l\otimes f_ld\eta_i$ where $v_l\in\Vbb^{\leq n}$ and $f_l\in H^0(U_i,\MO_{U_i}(\blt\SX))$. Then for each $b\in V$, noting that $f_l|_{U_i\cap\pi^{-1}(b)}$ can be viewed as an element of $\Cbb((\eta_i))$, we have
\begin{align*}
\sigma*_iw\big|_b=\sum_l\Res_{\eta_i=0} Y_i(v_l,\eta_i)w(b)\cdot f_l|_{U_i\cap\pi^{-1}(b)}\cdot d\eta_i
\end{align*}
where the RHS is the residue of an element of $\Wbb((\eta_i))d\eta_i$.

Now, we define the \textbf{residue action} of each $\sigma\in H^0\big(V,\pi_*\big(\SV_{\fx}\otimes \omega_{\MC/\MB}(\blt S_\fx)\big)\big)=H^0(\MC_V,\SV_{\fx}\otimes \omega_{\MC/\MB}(\blt S_\fx))$ on each $H^0(V,\scr W_\fx(\Wbb))\simeq \Wbb\otimes\MO(V)$ to be
\begin{align}
\sigma\cdot w=\sum_{i=1}^N\sigma*_i w
\end{align}

\begin{comment}
\begin{df}
Define an $\mc O_{\mc B}$-linear action of $\pi_*\big(\SV_{\fx}\otimes \omega_{\MC/\MB}(\blt S_\fx)\big)$ on $\SW_\fx(\Wbb)$ (called \textbf{residue action}) \index{00@Residue action} as follows. Choose any $V\subset \MB$ small enough such that there are local coordinates $\eta_1,\cdots,\eta_N$ at $\varsigma_1(V),\cdots,\varsigma_N(V)$. For each section $\sigma\in \pi_*\big(\SV_{\fx}\otimes \omega_{\MC/\MB}(\blt S_\fx)\big)(V)=H^0\big(\MC_V,\SV_{\fx}\otimes \omega_{\MC/\MB}(\blt S_\fx)\big)$ and $\wbf\in H^0\big(V,\SW_\fx(\Wbb)\big)$, we have $\MU(\eta_\blt)\wbf\in \Wbb\otimes \MO(V)$. Recall that the residue action of $\sigma$ on $\MU(\eta_\blt)\wbf$ is defined in \cite[(1.7.6)]{GZ1}. The residue actions of $\sigma$ on $\wbf$ are defined by
\begin{align}
\sigma \cdot \wbf:=\MU(\eta_\blt)^{-1}\big(\sigma \cdot \MU(\eta_\blt)\wbf\big)\qquad\in H^0(V,\SW_\fx(\Wbb))
\end{align}
\end{df}
Just as \cite[Remark 2.2.9]{GZ1}, we can prove that the residue action is independent of the choice of $\eta_\blt$.
\end{comment}

\begin{df}\label{lbb14}
    Define the \textbf{sheaf of coinvariants}\index{TT@$\ST_{\fx}(\Wbb)$, sheaves of coinvariants} associated to $\fx$ and $\Wbb$:
    \begin{align*}
        \ST_\fx(\Wbb)=\frac{\SW_\fx(\Wbb)}{\pi_*\big(\SV_{\fx}\otimes \omega_{\MC/\MB}(\blt S_\fx)\big)\cdot \SW_\fx(\Wbb)}.
    \end{align*}
The denominator on the RHS is the sheaf associated to the presheaf $\SJ_\fx^\pre$ on $\MB$ defined by
\begin{align}\label{eqb58}
\SJ_\fx^\pre(V)=H^0\big(\MC_V,\scr V_\fx\otimes\omega_{\MC/\MB}(\blt S_\fx)\big)\cdot H^0(V,\scr W_\fx(\Wbb))
\end{align}
for all open $V\subset\MB$, where $\Span_\Cbb$ has been suppressed on the RHS. The dual sheaf of $\ST_\fx(\Wbb)$ is denoted by 
    \begin{align*}
        \ST_\fx^*(\Wbb)=\big(\ST_\fx(\Wbb)\big)^*
    \end{align*}
    and called the \textbf{sheaf of conformal blocks}\index{TT@$\ST_{\fx}^*(\Wbb)$, sheaves of conformal blocks} associated to $\fx$ and $\Wbb$. An element of $H^0(\MB,\scr T_\fx^*(\Wbb))$ is called a \textbf{conformal block} associated to $\fk X$ and $\Wbb$. 
\end{df}

\begin{rem}\label{lbb20}
Both $\scr W_\fx(\Wbb)$ and $\scr T_\fx(\Wbb)$ depend on the choice of $\wtd L_\blt(0)$ on $\Wbb$. However, if $\fx$ admits local coordinates $\eta_\blt$ at $\sgm_\blt(\MB)$ and if the identification $\scr W_\fx(\Wbb)=\Wbb\otimes_\Cbb\mc O_\MB$ via $\mc U(\eta_\blt)$ is assumed, then $\scr T_\fx(\Wbb)$ does not rely on the choice of $\wtd L_\blt(0)$. In that case, when $\Vbb$ is $C_2$-cofinite and $\Wbb$ is a grading-restricted generalized $\Vbb^{\otimes N}$-module, there is no ambiguity in talking about $\scr W_\fx(\Wbb),\scr T_\fx(\Wbb),\scr T^*_\fx(\Wbb)$.
\end{rem}

\begin{thm}\label{lbb28}
Assume that $\fx$ admits local coordinates $\eta_\blt$ at $\sgm_\blt(\MB)$. Let $V\subset\MB$ be open, and let $\upphi:\scr W_{\fx_V}(\Wbb)\rightarrow\mc O_V$ be an $\mc O_V$-module morphism. Assume that $\MB$ is Stein. Then $\upphi$ belongs to $H^0(V,\scr T_{\fx_V}^*(\Wbb))$ if and only if $\upphi$ vanishes on the elements of $\SJ^\pre_\fx(\MB)$ (restricted to $V$).
\end{thm}

\begin{proof}
The direction ``$\Rightarrow$" is obvious. Let us prove ``$\Leftarrow$". Note that $\upphi$ is a conformal block iff for each $b\in V$, $\upphi$ vanishes on the stalk
\begin{align*}
(\SJ^\pre_\fx)_b= \pi_*\big(\SV_{\fx}\otimes \omega_{\MC/\MB}(\blt S_\fx)\big)_b\cdot \SW_\fx(\Wbb)_b=\bigcup_{k,l\in\Nbb}\pi_*\big(\SV_{\fx}^{\leq k}\otimes \omega_{\MC/\MB}(l S_\fx)\big)_b\cdot \SW_\fx(\Wbb)_b
\end{align*}
where $\Span_\Cbb$ has been suppressed as usual. Since $\fx$ admits local coordinates, we have $\scr W_\fx(\Wbb)\simeq\Wbb\otimes_\Cbb\mc O_\MB$. Therefore $\scr W_\fx(\Wbb)_b$ is generated by elements of $H^0(\MB,\scr W_\fx(\Wbb))$. Since $\SV_{\fx}^{\leq k}\otimes \omega_{\MC/\MB}(l S_\fx)$ is a locally free $\mc O_\MC$-module (of finite rank), by Grauert's direct image theorem, $\pi_*\big(\SV_{\fx}^{\leq k}\otimes \omega_{\MC/\MB}(l S_\fx)\big)$ is a coherent $\mc O_\MB$-module. Thus, by Cartan's theorem A, elements of $H^0\big(\MB,\pi_*\big(\SV_{\fx}^{\leq k}\otimes \omega_{\MC/\MB}(l S_\fx)\big)\big)$ generate $\big(\SV_{\fx}^{\leq k}\otimes \omega_{\MC/\MB}(l S_\fx)\big)_b$. Therefore, if $\upphi$ vanishes on $\SJ_\fx^\pre(\MB)$, it must vanish on $(\SJ_\fx^\pre)_b$. This proves ``$\Leftarrow$".
\end{proof}

%Recall that if $T\subset\MB$ is open and $\fx_T$ has local coordinates $\eta_\blt$ at $\sgm_\blt(T)$, then we have a trivialization $\mc U(\eta_\blt):\scr W_\fx(\Wbb)\xrightarrow{\simeq}\Wbb\otimes\mc O_T$. Thus, for each $\nu_\blt=(\nu_1,\dots,\nu_N)\in\Nbb^N$, $\scr W_\fx(\Wbb)$ has an $\mc O_\MB$-submodule $\scr W^{\leq\nu_\blt}_\fx(\Wbb)$ such that for each such $T,\eta_\blt$, the trivialization $\mc U(\eta_\blt)$ sends $\scr W_\fx^{\leq\nu_\blt}(\Wbb)|_V$ onto  $\Wbb^{\leq\nu_\blt}\otimes\mc O_T$. Here, $\Wbb^{\leq\nu_\blt}$ is the (finite-dimensional) subspace of $\Wbb$ spanned by joint-eigenvectors of $\wtd L_1(0),\dots,\wtd L_N(0)$ whose joint-eigenvalues are less than $(\nu_1,\dots,\nu_N)$. 

\begin{rem}
When $\Wbb$ is a tensor product of grading-restricted $\Vbb$-modules, Thm. \ref{lbb28} is simply \cite[Thm. 6.3]{Gui-sewingconvergence}. Note that in Thm. \ref{lbb28} we assumed the existence of $\eta_\blt$. This condition is missing from \cite[Thm. 6.3]{Gui-sewingconvergence} and should be added. Therefore, we have presented the proof of Thm. \ref{lbb28} above to illustrate how to use the existence of $\eta_\blt$.
\end{rem}

The following theorem generalizes \cite[Thm. 7.4]{Gui-sewingconvergence}. Note that it was always assumed in \cite{Gui-sewingconvergence} that the base manifold has finitely many components. In particular, \cite[Thm. 7.4]{Gui-sewingconvergence} assumes this condition.

\begin{thm}\label{lbb21}
Assume that $\fx$ is equipped with local coordinates $\eta_\blt$ at $\sgm_\blt(\MB)$. \footnote{Here, we do not assume that $\eta_\blt$ is obtained by constantly extending local coordinates of $\fx$ as in Rem. \ref{lbb9}.} Assume that $\wtd\MB$ (and hence $\MB$) has finitely many connected components. Assume that $\Vbb$ is $C_2$-cofinite. Associate a grading restricted $\Vbb^{\otimes N}$-module $\Wbb$ to $\sgm_\blt(\MB)$. Then for each Stein open subset $V\subset \MB$, 
    \begin{align}
        \frac{\SW_\fx(\Wbb)(V)}{\SJ_\fx^\pre(V)}
    \end{align}
is a finitely generated $\mc O(V)$-module.  
\end{thm}

\begin{proof}
This was proved in \cite[Thm. 7.4]{Gui-sewingconvergence} in the special case that $R=1$, that $V=\MB$ is Stein, and that $\Mbb$ is a tensor product of grading-restricted $\Vbb$-modules. However, it is not hard to generalize the proof of \cite[Thm. 7.4]{Gui-sewingconvergence} to the present case. Instead of reproducing the proof, we highlight some key points of the proof and the adaptions that should be made.

Similar to (7.2) of \cite{Gui-sewingconvergence}, we must show that for each $E\in\Nbb$ there exists $k_0\in\Nbb$ such that for each $k\geq k_0$ and each $b\in V$, we have
\begin{align}\label{eqb51}
H^1(\MC_b,\scr V_{\MC_b}^{\leq E}\otimes\omega_{\MC_b}(k\SX))=0
\end{align}
This follows from \cite[Thm. 2.3]{Gui-sewingconvergence}, which shows that a $b$-dependent $k_0$ can be chosen to rely only on the topology of the fiber $\fx_b$, namely, the topology of $\MC_b$ and the number of elements of $\sgm_1(b),\dots,\sgm_N(b)$ in each irreducible component of $\MC_b$. Although \cite[Thm. 2.3]{Gui-sewingconvergence} only deals with the case that the nodal curve $\MC_b$ has at most one node (i.e. $R=1$), the proof clearly also applies to nodal curves with arbitrary numbers of nodes (i.e. arbitrary $R$). Moreover, \cite[Thm. 2.3]{Gui-sewingconvergence} assumes that each irreducible component of $\MC_b$ contains at least one of  $\sgm_1(b),\dots,\sgm_N(b)$. This is ensured by our assumption \eqref{eqb52}. Now, the assumption that $\wtd\MB$ has finitely many components ensures that the fibers of $\fx$ have only finitely many different topologies. Thus, a uniform $k_0$ for all $b\in V$ can be found.

The generalization of \cite[Thm. 7.4]{Gui-sewingconvergence} to arbitrary $R$ and Stein open $V\subset\MB$ has been addressed. The generalization of the proof of \cite[Thm. 7.4]{Gui-sewingconvergence} to the case that $\Wbb$ is not necessarily a tensor product $\Mbb_1\otimes\cdots\otimes\Mbb_N$ is straightforward, except that the third paragraph of that proof (using a finiteness property of $\Mbb_i$ due to \cite[Cor. 7.3]{Gui-sewingconvergence}) should be modified. The correct way to modify can be found in the proof of \cite[Thm. 3.25]{GZ1}, a variant of Thm. \ref{lbb21} which deals with partial conformal blocks but concerns a fiber $\fx_b$ instead of a family $\fx$.
\end{proof}

\begin{co}\label{locallyfree2}
Assume that $\Vbb$ is $C_2$-cofinite. Associate a grading restricted $\Vbb^{\otimes N}$-module $\Wbb$ to $\sgm_\blt(\MB)$.  Then  $\ST_\fx(\Wbb)$ is a finite type $\MO_\MB$-module. In particular, for each $b\in\MB$ we have $\dim \scr T_\fx(\Wbb)|_b<+\infty$, and the function $b\in\MB\mapsto\dim\scr T_\fx(\Wbb)|_b$ is lower semicontinuous.
\end{co}

\begin{proof}
Since the finite type condition can be checked locally, we may assume that $\fx$ admits local coordinates $\eta_\blt$ at $\sgm_\blt(\MB)$, and $\wtd\MB$ is connected and Stein. By Thm. \ref{lbb21}, we can find $s_1,\dots,s_n\in\scr W_\fx(\Wbb)(\MB)$ generating the $\mc O(\MB)$-module $\scr W_\fx(\Wbb)(\MB)/\SJ_\fx^\pre(\MB)$. Since $\scr W_\fx(\Wbb)\simeq\Wbb\otimes_\Cbb\mc O_\MB$, the elements of  $\scr W_\fx(\Wbb)(\MB)$ generate the $\mc O_\MB$-module $\scr T_\fx(\Wbb)$, i.e., for each $b\in\MB$, every element of the stalk $\scr T_\fx(\Wbb)_b$ is an $\mc O_{\MB,b}$-linear combination of elements of $\scr W_\fx(\Wbb)(\MB)$. Thus $s_1,\dots,s_n$ generate the $\mc O_\MB$-module $\scr T_\fx(\Wbb)$. This proves that $\scr T_\fx(\Wbb)$ is of finite type. The rest of the corollary follows from the following Prop. \ref{lbb22}.
\end{proof}

\begin{pp}\label{lbb22}
Let $X$ be a complex manifold. Let $\scr E$ be a finite-type $\mc O_X$-module. Then for each $x\in X$ we have $\dim\scr E|_x<+\infty$, and the \textbf{rank function}
\begin{align*}
\Rbf:X\rightarrow\Nbb\qquad x\mapsto\dim\scr E|_x
\end{align*}
is upper semicontinuous. Moreover, if $\Rbf$ is also lower semicontinuous (and hence locally constant), then $\scr E$ is locally free, i.e., it is a holomorphic vector bundle.
\end{pp}

We warn the readers that in algebraic geometry, there is a similar proposition which does not distinguish between coherent sheaves and finite type sheaves. This is because the sheaves involved are automatically quasi-coherent and are over locally Noetherian schemes, in which case coherence is equivalent to finite type. In our analytic setting, however, one must distinguish between coherence and finite type. Due to this subtlety, we provide a self-contained proof of Prop. \ref{lbb22} for the reader's convenience, even though this proposition is elementary and well-known.

%It is not hard to see that the proposition, except the last sentence, holds true if $X$ is more generally a complex space. The last sentence holds true if $X$ is a reduced complex space.

\begin{proof}
Choose any $x\in X$. Then $x$ has a neighborhood $U$ such that there exist $s_1,\dots,s_n\in\scr E(U)$ generating $\scr E|_U$ (i.e. generating $\scr E_y$ for each $y\in U$). In particular, $s_\blt$ generate $\scr E_x$, and hence spans the fiber $\scr E|_x$. Therefore, by removing some members of $s_\blt$ we may assume that $s_1|_x,\dots,s_n|_x$ form a basis of $\scr E|_x$. So $n=\dim\scr E|_x$. Nakayama's lemma implies that the germs $s_{\blt,_x}$ (of $s_\blt$ at $x$) generate the stalk $\scr E_x$. 

Since $\scr E$ is of finite type, by shrinking $U$ to a smaller neighborhood of $x$ we can find $t_1,\dots,t_m\in\scr E(U)$ generating $\scr E|_U$. Since each germ $t_{j,x}$ (of $t_j$ at $x$) is an $\mc O_{X,x}$-linear combination of $s_{\blt,_x}$, by further shrinking $U$, each $t_j$ is an $\mc O(U)$-linear combination of $s_\blt$. Thus $s_\blt$ generate $\scr E|_U$. In particular, for each $y\in U$, $s_1|_y,\dots,s_n|_y$ span the fiber $\scr E|_y$. Thus $\dim\scr E|_y\leq n$. This proves that $\Rbf$ is upper semicontinuous.

The above paragraph shows that we have an $\mc O_U$-module epimorphism $\varphi:\mc O_U^n\rightarrow\scr E|_U$ defined by $(f_1,\dots,f_n)\mapsto f_1s_1+\cdots+f_ns_n$.  Now, suppose that $\Rbf$ is constantly $n$ on $U$. Then for each $y\in U$, since $s_\blt|_y$ span $\scr E|_y$, they form a basis of $\scr E|_y$. Therefore, if $V\subset U$ is open and $f_1,\dots,f_n\in\mc O(V)$ satisfy $f_1s_1+\cdots+f_ns_n=0$ in $\scr E(V)$, then $f_1(y)s_1|_y+\cdots+f_n(y)s_n|_y=0$ in $\scr E|_y$. Therefore $f_1(y)=\cdots=f_n(y)=0$ for all $y\in V$. This proves that $s_\blt$ is a free generator of $\scr E|_U$. Thus $\scr E|_U$ is free. We have thus proved that if $\Rbf$ is locally constant then $\scr E$ is locally free.
\end{proof}

\begin{rem}\label{lbb33}
Assume that $R=0$ and hence $\fx=\wtd\fx$ is smooth. Recall that for each $b\in\MB$, $\scr T_{\fx_b}(\Wbb)$ is the space of coinvariants associated to the fiber $\fx_b$. By \cite[Prop. 2.19]{GZ1} (see also \cite[Rem. 6.5]{Gui-sewingconvergence}), the restriction map $\scr T_\fx(\Wbb)_b\mapsto\scr T_{\fx_b}(\Wbb)$ descends to a linear isomorphism
\begin{align*}
\scr T_\fx(\Wbb)|_b\xlongrightarrow{\simeq}\scr T_{\fx_b}(\Wbb)
\end{align*}
Thus, if $\Vbb$ is $C_2$-cofinite and $\Wbb$ is a grading restricted $\Vbb^{\otimes N}$-module, then by Cor. \ref{locallyfree2}, the following three numbers are finite and equal:
\begin{align*}
\dim\scr T_\fx(\Wbb)|_b=\dim\scr T_{\fx_b}(\Wbb)=\dim\scr T^*_{\fx_b}(\Wbb)<+\infty
\end{align*}
The second half of Prop. \ref{lbb22} immediately implies:
\end{rem}

\begin{lm}\label{lbb34}
Assume that $R=0$ and hence $\fx=\wtd\fx$ is smooth. Assume that $\Vbb$ is $C_2$-cofinite. Associate a grading restricted $\Vbb^{\otimes N}$-module $\Wbb$ to $\sgm_\blt(\MB)$. Assume that the rank function
\begin{gather*}
\Rbf:\MB\rightarrow\Nbb\qquad b\mapsto\dim\scr T_\fx(\Wbb)|_b=\dim\scr T_{\fx_b}(\Wbb)=\dim\scr T^*_{\fx_b}(\Wbb)
\end{gather*}
is lower semicontinuous. Then $\scr T_\fx(\Wbb)$ is locally free. Hence $\scr T_\fx^*(\Wbb)$ is also locally free.
\end{lm}

\section{Connections}

We fix an admissible $\Vbb^{\times N}$-module $\Wbb$ (cf. the beginning of Sec. \ref{lbb70}) associated to the ordered marked points $\sgm_\blt(\mc B)$ of $\fx$.

\subsection{Lie derivatives in sheaves of VOAs}
In this section, we assume that $R=0$, and hence $\fx=\wtd{\fk X}$ is smooth (cf. Asmp. \ref{lbb1}). Choose a section $\xk$ of $\Theta_{\MC}(\blt S_\fx)$ on $U$, where $U$ is an open subset of $\MC$. 

\begin{df}\label{lbb3}
We say that $\xk$ is \textbf{fiber-preserving} if for each $b\in\mc B$, there exists a (necessarily unique) $\yk(b)\in\Theta_{\mc B}|_b$ such that $d\pi(\xk(p))=\yk(b)$ for every $p$ in the fiber $U_b\equiv \pi^{-1}(b)\cap U$. Equivalently, in view of the map
\begin{align*}
d\pi:H^0\big(U,\Theta_\MC)\rightarrow H^0(U,\pi^*\Theta_\MB)
\end{align*}
induced by the $\mc O_\MC$-module morphism $d\pi:\Theta_\MC\rightarrow\pi^*\Theta_\MB$, there exists a (necessarily unique) $\yk\in H^0(\pi(U),\Theta_\MB)$ such that $d\pi(\xk)=\pi^*\yk$.
\end{df}

The goal of this subsection is to define, for each fiber-preserving $\xk$, the \textbf{Lie derivative operator $\ML_\xk$}
\begin{align}\label{eqb15}
    \ML_\xk:\SV_\fx^{\leq n}(U)\rightarrow \SV_\fx^{\leq n}(U)
\end{align}
Lie derivatives play an important role in defining the connection on the sheaf of coinvariants and conformal blocks. 

\begin{df}
    Let $\varphi:U_1\rightarrow U_2$ be a biholomorphic map, where $U_1,U_2$ are open subsets of $\MC$ and $\varphi(U_1)=U_2$. We say $\varphi$ is \textbf{fiber-preserving}\index{00@Fiber-preserving biholomorphic maps} if $\varphi$ maps each fiber of $U_1$ onto a fiber of $U_2$; equivalently,  $\varphi((U_1)_{\pi(x)})=(U_2)_{\pi\circ \varphi(x)}$ for each $x\in U_1$.
\end{df}

Let $\varphi:U_1\rightarrow U_2$ be a fiber-preserving biholomorphic map. With abuse of notations, we let $\varphi_*$ denote the two pushforward maps
\begin{subequations}
\begin{gather}
    \varphi_*:\MO_{U_1}\xrightarrow{\simeq} \MO_{U_2},\qquad f\mapsto f\circ \varphi^{-1},\label{equiv1}\\
    \varphi_*:\Vbb^{\leq n}\otimes_\Cbb \MO_{U_1}\xrightarrow{\simeq} \Vbb^{\leq n}\otimes_\Cbb \MO_{U_2},\qquad v\mapsto v\circ \varphi^{-1}.\label{equiv2}
\end{gather}
\end{subequations}

\begin{comment}
Suppose that $\eta\in \MO(U_2)$ is univalent on each fiber. Then we have an equivalence 
\begin{align}
    (\eta,\pi)_*:\MO_{U_2}\xrightarrow{\simeq} \MO_{(\eta,\pi)(U_2)},\qquad f\mapsto f\circ (\eta,\pi)^{-1}.
\end{align}
Recall (cf. \cite[Sec. 1.7]{GZ1}) that the trivialization $\mc U_\varrho(\eta)=\eqref{eq175}$ gives a sheaf-isomorphism
\begin{equation}\label{equiv3}
    \begin{gathered}
     \MV_\varrho(\eta):\SV_\fx^{\leq n}\vert_{U_2} \xrightarrow{\simeq} \Vbb^{\leq n}\otimes_\Cbb \MO_{(\eta,\pi)(U_2)},\\
     \MV_\varrho(\eta)=(\eta,\pi)_* \MU_\varrho(\eta).
    \end{gathered}
\end{equation}
\end{comment}

%Similarly, since $\eta\circ\varphi\in\mc O(U_1)$ is univalent on each fiber, and since the fiber-preserving assumption implies $(\eta\circ\varphi,\pi)(U_1)=(\eta,\pi)(U_2)$, we have
%\begin{align*}
%\mc V_\varrho(\eta\circ\varphi):\SV_\fx^{\leq n}\vert_{U_1} \xrightarrow{\simeq} \Vbb^{\leq n}\otimes_\Cbb \MO_{(\eta,\pi)(U_2)}.
%\end{align*}

\begin{df}
    The \textbf{pushforward} $\MV_\varrho(\varphi)$\index{V@$\MV_\varrho(\varphi)$, pushforward of $\varphi$} of $\varphi$ is a sheaf equivalence determined by
    \begin{equation}\label{eqb16}
        \begin{gathered}
            \MV_\varrho(\varphi):\SV_\fx^{\leq n}\vert_{U_1} \xrightarrow{\simeq} \SV_\fx^{\leq n}\vert_{U_2},\\
    \MU_\varrho(\eta)\MV_\varrho(\varphi)=\varphi_* \cdot \MU_\varrho(\eta\circ \varphi).
        \end{gathered}
    \end{equation}
In other words, the following diagram commutes:
\begin{equation}
\begin{tikzcd}
\scr V_{\fk X}^{\leq n}|_{U_1} \arrow[r,"\mc V_\varrho(\varphi)"] \arrow[d,"\mc U_\varrho(\eta\circ\varphi)"'] & \scr V_\fx^{\leq n}|_{U_2} \arrow[d,"\mc U_\varrho(\eta)"] \\
\Vbb^{\leq n}\otimes_\Cbb\mc O_{U_1} \arrow[r,"{\varphi_*}"]           & \Vbb^{\leq n}\otimes_\Cbb\mc O_{U_2}         
\end{tikzcd}
\end{equation}
\end{df}

\begin{rem}
The pushforward $\MV_\varrho(\varphi)$ is independent of the choice of $\eta$. To see this, choose another $\mu\in \MO(U_2)$ univalent on each fiber. From the description of $\mc U(\varrho(\eta|\mu))$ after \eqref{eq175}, it is easy to see that, as automorphisms on $\Vbb^{\leq n}\otimes_\Cbb\mc O_{U_1}$, we have
\begin{align*}
    \MU(\varrho(\eta\circ \varphi\vert \mu\circ \varphi))=\varphi_*^{-1}\cdot \MU(\varrho(\eta\vert \mu))\cdot \varphi_*
\end{align*}
(Quick proof: if we identify $U_1$ with $U_2$ via $\varphi$ so that $\varphi$ is the identity map, then both sides equal $\MU(\varrho(\eta\vert \mu))$.) Equivalently,
\begin{align}\label{eqb17}
    \MU_\varrho(\eta\circ \varphi)\MU_\varrho(\mu \circ \varphi)^{-1}=\varphi_*^{-1}\cdot \MU_\varrho(\eta)\MU_\varrho(\mu)^{-1}\cdot \varphi_*.
\end{align}
Then the independence of $\eta$ follows from \eqref{eqb16} and \eqref{eqb17}. 
\end{rem}

\begin{rem}\label{mc1}
$\mc V_\varrho$ is a groupoid homomorphism. More precisely: It is clear that $\mc V_\varrho(\id)=\idt$.  Suppose that $\psi:U_0\rightarrow U_1$ is another fiber-preserving biholomorphic map. From the commutative diagrams
\begin{equation*}
\begin{tikzcd}
\scr V_{\fk X}^{\leq n}|_{U_0}\arrow[r,"\mc V_\varrho(\psi)"] \arrow[d,"\mc U_\varrho(\eta\circ\varphi\circ\psi)"']&\scr V_{\fk X}^{\leq n}|_{U_1} \arrow[r,"\mc V_\varrho(\varphi)"] \arrow[d,"\mc U_\varrho(\eta\circ\varphi)"'] & \scr V_\fx^{\leq n}|_{U_2} \arrow[d,"\mc U_\varrho(\eta)"] \\
\Vbb^{\leq n}\otimes_\Cbb\mc O_{U_0}\arrow[r,"\psi_*"]&\Vbb^{\leq n}\otimes_\Cbb\mc O_{U_1} \arrow[r,"{\varphi_*}"]           & \Vbb^{\leq n}\otimes_\Cbb\mc O_{U_2}         
\end{tikzcd}
\end{equation*}
we have $\mc V_\varrho(\varphi\circ\psi)=\MV_\varrho(\varphi)\circ\MV_\varrho(\psi)$. Therefore, we also have $\MV_\varrho(\varphi^{-1})=\MV_\varrho(\varphi)^{-1}$.
\end{rem}

Choose a section $\xk$ of $\Theta_{\MC}$ on $U$, where $U$ is an open subset of $\MC$. We let $\varphi^\xk$ be the flow generated by the vector field $\xk$. More precisely, if we choose any open precompact $V\subset U$, we can define the $\varphi^\xk\in \MO_{T\times V}(T\times V),(\zeta,p)\mapsto \varphi_\zeta^\xk(p)$ satisfying
\begin{subequations}
    \begin{gather}
    \varphi_0^\xk(p)=p, \label{ode1} \\
    \partial_\zeta \varphi_\zeta^\xk(p)\vert_{\zeta=0}\cdot\partial_\zeta=\xk(p),\label{ode2}
\end{gather}
\end{subequations}
where $T$ is an open subset of $\Cbb$ containing 0. Note that \eqref{ode1} implies that the pushforward $\MV_\varrho(\varphi_0^\xk)$ is the identity map on $\SV_\fx^{\leq n}\vert_V$. By the chain rule, \eqref{ode2} is equivalent to that for any $f\in\MO_V$,
\begin{align}\label{eqb18}
    \partial_{\zeta}(f\circ \varphi_\zeta^\xk)\vert_{\zeta=0}=\xk f.
\end{align}
It is clear that if $\xk$ is fiber-perserving, then so is its flow $\varphi_\zeta^\xk$. Therefore, for each $\zeta$, the operator $\mc V_\varrho(\varphi_\zeta^\xk)$ can be defined.

\begin{df}\label{lbb2}
Let $\xk\in H^0(U,\Theta_\MC)$ be fiber preserving.  For any $v\in \SV_\fx^{\leq n}(U)$, we define $\ML_\xk v\in \SV_\fx^{\leq n}(U)$ as follows. For any precompact open subset $V$ of $U$,
    \begin{align}\label{lie1}
        \ML_\xk v\vert_V:=\lim_{\zeta\rightarrow 0} \frac{\MV_\varrho(\varphi_\zeta^\xk)^{-1}\big(v\vert_{\varphi_\zeta^\xk(V)}\big)-v\vert_V}{\zeta}
    \end{align}
\end{df}

\begin{comment}
\begin{rem}\label{rem1}
\eqref{shortexact} implies a long exact sequence of cohomology group
\begin{align*}
    0\rightarrow H^0\big(U,\Theta_{\MC/\MB}(\blt S_\fx)\big)\rightarrow H^0\big(U,\Theta_\MC(\blt S_\fx)\big)
    \xrightarrow{d\pi}H^0\big(U,\pi^* \Theta_\MB(\blt S_\fx)\big)\rightarrow \cdots.
\end{align*}
We can find $\yk\in H^0\big(\pi(U),\Theta_\MB(\blt S_\fx)\big)$ such that $d\pi (\xk)=\pi^* \yk$. This means $\xk$ is a vector field on $U$ whose projection onto $\MB$ depends only on the points of $\MB$.
\end{rem}
\end{comment}

We would like to give an explicit formula of $\ML_\xk v$, which will imply the convergence of the limit \eqref{lie1}. For that purpose, we need the following lemma.
\begin{lm}\label{lemma1}
Let $T$ be an open subset of $\Cbb$ containing $0$. Let $\rho:T\rightarrow\Gbb,\zeta\mapsto\rho_\zeta$ be a holomorphic family of transformations satisfying $\rho_0(z)=z$. Then for any $v\in\Vbb^{\leq n}$,
\begin{align}\label{eqb19}
    \partial_\zeta \MU(\rho_\zeta)v\big\vert_{\zeta=0}=-\partial_\zeta \MU(\rho_\zeta^{-1})v\big\vert_{\zeta=0}=\sum_{k\geq 1}\frac{1}{k!}\big(\partial_\zeta\rho_\zeta^{(k)}(0)\big|_{\zeta=0}\big)L(k-1)v
\end{align}
where $\partial_\zeta\rho_\zeta^{(k)}(z)=\partial_z^k\partial_\zeta\rho_\zeta(z)=\partial_\zeta\partial_z^k\rho_\zeta(z)$.
\end{lm}

%Note that $\zeta\mapsto\varrho(\eta_\zeta\vert \eta)$ denotes the holomorphic family of transformations $(\zeta,p)\in T\times U\mapsto \varrho(\eta_\zeta|\eta)_p\in\Gbb$ (as described in \eqref{eqb39}). Thus $(\zeta,p)\mapsto \mc U(\varrho(\eta_\zeta|\eta)_p)v \in\Vbb^{\leq n}$ is holomorphic, and hence the LHS of \eqref{eqb19} is in $\Vbb^{\leq n}\otimes_\Cbb\mc O(U)$. Since $(\zeta,p)\mapsto\eta_\zeta(p)$ is holomorphic, the RHS of \eqref{eqb19} is also in $\Vbb^{\leq n}\otimes_\Cbb\mc O(U)$.

\begin{proof}
The first equality of \eqref{eqb19} follows from $\mc U(\rho_\zeta)\cdot\mc U(\rho_\zeta)^{-1}=\id$ and $\mc U(\rho_0)=\id$. Therefore, it suffices to show that the first term equals the third term of \eqref{eqb19}. Let $c_1,c_2,\dots\in\scr O_{\Cbb}(T)$ such that 
\begin{gather*}
\rho_\zeta(z)=\rho_\zeta'(0)\exp\Big(\sum_{k\geq 1}c_k(\zeta)z^{k+1}\partial_z\Big)(z).
\end{gather*}
Since $\rho_0(z)=z$, and since $c_1(0),c_2(0),\dots$ are uniquely determined by $\rho_0$, we have $\rho_0'(0)=1$ and $c_1(0)=c_2(0)=\cdots=0$. Clearly $\rho_\zeta(z)$ equals
\begin{align*}
\rho_\zeta'(0)\Big(z+\sum_{k\geq 1}c_k(\zeta)z^{k+1}\Big)
\end{align*}
plus some polynomials of $z$ multiplied by at least two terms among $c_1(\zeta),c_2(\zeta),\dots$. This observation, together with $\rho_0'(0)=1,c_i(0)=0$, implies
\begin{align*}
\partial_\zeta\rho_\zeta(z)\Big|_{\zeta=0}=\partial_\zeta\rho_\zeta'(0)z\Big|_{\zeta=0}+\Big(\sum_{k\geq 1}\partial_\zeta c_k(0)z^{k+1}\Big)
\end{align*}
Therefore, when $k\geq 2$,
\begin{align*}
\frac 1{k!}\partial_\zeta\rho_\zeta^{(k)}(0)\Big|_{\zeta=0}=\partial_\zeta c_{k-1}(0). \tag{$\star$} \label{eqb20}
\end{align*}
Similarly, $\mc U(\rho_\zeta)v=\rho_\zeta'(0)^{L(0)}\exp\big(\sum_{k\geq1}c_k(\zeta)L(k)\big)v$ equals
\begin{align*}
\rho_\zeta'(0)^{L(0)}\Big( v+\sum_{k\geq1}c_k(\zeta)L(k)v\Big)
\end{align*}
plus some vectors of $\Vbb^{\leq n}$ multiplied by at least two terms among $c_1(\zeta),c_2(\zeta),\dots$, and this sum is a finite sum. Using again $\rho_0'(0)=1,c_i(0)=0$, we obtain
\begin{align*}
\partial_\zeta\mc U(\rho_\zeta)v|_{\zeta=0}=\partial_\zeta\rho_\zeta'(0) L(0)v\Big|_{\zeta=0}+\Big(\sum_{k\geq 1}\partial_\zeta c_k(0) L(k)v\Big)
\end{align*}
Plugging \eqref{eqb20} into this formula, we obtain \eqref{eqb19}.
\end{proof}

\begin{rem}\label{lbb68}
Assume the setting of Def. \ref{lbb2}. Assume $U$ is small enough such that there is $\eta\in \MO(U)$ univalent on each fiber. Recall that $R=0$ and so there are no $q_\blt$-variables. Assume that $\MB$ is small enough to admit a set of coordinates $\tau_\blt=(\tau_1,\dots,\tau_m)$. Denote $\tau_\blt\circ \pi$ also by $\tau_\blt$ for simplicity. Then $(\eta,\tau_\blt)$ is a set of coordinates of $U$. By Def. \ref{lbb3}, we can find $\yk\in H^0(\pi(U),\Theta_\MB)$ such that $(d\pi)(\xk)=\pi^*\yk$. Write
\begin{align}
\yk=\sum_{j=1}^m g_j(\tau_\blt)\partial_{\tau_j}
\end{align}
where $g_j\in\mc O(\tau_\blt(U))$. Then we can $h\in \mc O((\eta,\tau_\blt)(U))$ such that
\begin{align}\label{eqb21}
    \xk=h(\eta,\tau_\blt)\partial_\eta+\sum_{j=1}^m g_j(\tau_\blt)\partial_{\tau_j}
\end{align}
\end{rem}

\begin{pp}\label{lbb5}
Assume the setting of Def. \ref{lbb2}. Then the limit in \eqref{lie1} is convergent. Moreover, assume the setting of Rem. \ref{lbb68}. Set $u=\mc U_\varrho(\eta)v$, which is an element of $\Vbb^{\leq n}\otimes_\Cbb\mc O(U)$. Then, in $\Vbb^{\leq n}\otimes_\Cbb\mc O(U)$ we have
\begin{align}\label{eqb24}
\mc U_\varrho(\eta)\mc L_\xk v=
    h(\eta,\tau_\blt)\partial_\eta u+\sum_{j=1}^m g_j(\tau_\blt)\partial_{\tau_j}u-\sum_{k\geq 1}\frac{1}{k!}\partial_\eta^k h(\eta,\tau_\blt)L(k-1)u
\end{align}
\end{pp}

\begin{proof}
Choose any precompact open subset $V$ of $U$. Let us show that the limit in \eqref{lie1} converges to $\mc U_\varrho(\eta)$ times the RHS of \eqref{eqb24}. Choose $V,\varphi^{\fk x}$ as above. We have $v=\mc U_\varrho(\eta)^{-1}u$. Then
\begin{align*}
&\mc U_\varrho(\eta)\mc V_\varrho(\varphi^{\fk x}_\zeta)^{-1}\big(v\big|_{\varphi^{\fk x}_\zeta(V)}\big)=\mc U_\varrho(\eta)\mc V_\varrho(\varphi^{\fk x}_\zeta)^{-1}\mc U_\varrho(\eta)^{-1}\big(u\big|_{\varphi^{\fk x}_\zeta(V)}\big)\\
\xlongequal{\eqref{eqb16}}&\mc U_\varrho(\eta)\mc U_\varrho(\eta\circ\varphi^{\fk x}_\zeta)^{-1}(\varphi^{\fk x}_\zeta)_*^{-1}\big(u\big|_{\varphi^{\fk x}_\zeta(V)}\big)\xlongequal{\eqref{eqb22}}\mc U(\varrho(\eta|\eta\circ\varphi^\xk_\zeta))\big(u\circ\varphi^{\fk x}_\zeta\big)\big|_V
\end{align*}
It is easy to see that the derivative over  $\zeta$ of the above expression at $\zeta=0$ equals \eqref{eqb24}. Indeed, the first two terms of \eqref{eqb24} come from the derivative of $u\circ\varphi^{\fk x}_\zeta$. We shall show that last term comes from the derivative of $\mc U(\varrho(\eta|\eta\circ\varphi_\zeta^{\fk x}))$. We claim that for each $p\in U$,
\begin{align*}\label{eqb41}
\partial_\zeta\varrho(\eta\circ\varphi^\xk_\zeta|\eta)^{(k)}_p(0)\big|_{\zeta=0}=\partial_\eta^k h(\eta,\tau_\blt)\big|_p  \tag{$\star$}
\end{align*}
Then, applying Lem. \ref{lemma1} to the family $\zeta\mapsto \varrho(\eta\circ\varphi_\zeta^\xk|\eta)_p\in\Gbb$ and noting $\varrho(\eta|\eta\circ\varphi_\zeta^{\fk x})=\varrho(\eta\circ\varphi_\zeta^{\fk x}|\eta)^{-1}$, we get
\begin{align*}
\partial_\zeta\mc U(\varrho(\eta|\eta\circ\varphi_\zeta^{\fk x}))\big|_{\zeta=0}=-\partial_\zeta\mc U(\varrho(\eta\circ\varphi_\zeta^{\fk x}|\eta))\big|_{\zeta=0}=-\sum_{k\geq 1}\frac 1{k!}\partial_\eta^k h(\eta,\tau_\blt)L(k-1).
\end{align*}
which will complete the proof.

By \eqref{eqb40},  
\begin{align*}
\varrho(\eta\circ\varphi^\xk_\zeta|\eta)_p(z)=\eta\circ\varphi_{\zeta}^\xk\circ(\eta,\pi)^{-1}(z+\eta(p),\tau_\blt(p))-\eta\circ\varphi_\zeta^\xk(p).
\end{align*}
By \eqref{eqb18}, we have $\partial_\zeta (\eta\circ\varphi_{\zeta}^\xk)\big|_{\zeta=0}=\xk\eta=h(\eta,\tau_\blt)$, and hence
\begin{align*}
\partial_\zeta\varrho(\eta\circ\varphi^\xk_\zeta|\eta)_p(z)\big|_{\zeta=0}=h(z+\eta(p),\tau_\blt(p))-h(\eta(p),\tau_\blt(p))
\end{align*}
Eq. \eqref{eqb41} follows immediately from the above equation.
\end{proof}

In the following, we always assume the setting of Rem. \ref{lbb68}.

\begin{rem}\label{lbb6}
Recall that $\Theta_{\MC/\MB}$ is the relative tangent sheaf, cf. Def. \ref{defb1}. Choose any $k\in\Zbb$. If $\varphi:U_1\rightarrow U_2$ is a fiber-preserving biholomorphism, we have the pushforward map
\begin{align*}
\varphi_*:\Theta_{\mc C/\mc B}^{\otimes k}|_{U_1}\rightarrow\Theta_{\mc C/\mc B}^{\otimes k}|_{U_2}
\end{align*}
(For each $b,b'\in\mc B$ such that $\varphi$ sends the fiber $(U_1)_b=U_1\cap \pi^{-1}(b)$ to $(U_2)_{b'}$, $\varphi_*$ restricts to the $k$-th tensor product of the differential map of tangent vectors  $\Theta_{(U_1)_b}^{\otimes k}\rightarrow\Theta^{\otimes k}_{(U_2)_{b'}}$) This allows us to define the Lie derivative $\mc L_\xk$ on $\Theta^{\otimes k}_{\mc C/\mc B}|_U$ if $\xk\in H^0(U,\Theta_\MC)$ is fiber-preserving: If $\nu\in \Theta^{\otimes k}_{\mc C/\mc B}(U)$, then for any precompact open subset $V$ of $U$ we have
\begin{align*}
\mc L_\xk\nu|_V:=\lim_{\zeta\rightarrow 0} \frac{(\varphi^\xk_\zeta)_*^{-1}\big(\nu\vert_{\varphi_\zeta^\xk(V)}\big)-\nu\vert_V}{\zeta}
\end{align*}
If $\xk$ is written as \eqref{eqb21}, and if $\nu=f \partial_\eta^{\otimes k}$ where $f\in\mc O(U)$, then similar to the proof of Prop. \ref{lbb5}, we have
\begin{align}\label{eqb42}
\mc L_\xk\nu=\Big(h(\eta,\tau_\blt)\partial_\eta f+\sum_{j=1}^m g_j(\tau_\blt)\partial_{\tau_j}f-k\cdot \partial_\eta h(\eta,\tau_\blt)\cdot f\Big)\cdot \partial_\eta^{\otimes k}
\end{align}
(This is because the change of coordinate formula for $\Theta_{\mc C/\mc B}^{\otimes k}$ is the same as that of the line sub-bundle of $\scr V_\fx$ generated by $\mc U_\varrho(\eta)v\big|_U$ (for all possible open subset $U\subset\mc C$ and all $\eta\in\mc O(U)$ univalent on fibers) where $v\in\Vbb$ is a given vector satisfying $L(0)v=kv$ and $L(1)v=L(2)v=\cdots=0$.)
%(Thus, the trick is that we set $L(0)\partial_\eta^{\otimes k}=k\partial_\eta^{\otimes k},L(1)\partial_\eta^{\otimes k}=L(2)\partial_\eta^{\otimes k}=\cdots=0$ and apply \eqref{eqb24}. To be more rigorous, \eqref{eqb42} follows from \eqref{eqb24} and the canonical equivalence $\Vbb(k)\otimes_\Cbb\Theta_{\mc C/\mc B}^{\otimes k}\simeq \scr V^{\leq k}_{\fk X}/\scr V^{\leq k-1}_{\fk X}$ in which $u\otimes \partial_\eta^{\otimes k}$ (where $u\in\Vbb(k)$) is equivalent to the equivalence class of $\mc U_\varrho(\eta)^{-1}u$, cf. \cite[Prop. 5.7]{Gui-sewingconvergence}.) 
Specializing to $k=-1$ and $d\eta=\partial_\eta^{\otimes(-1)}$, we obtain the formula for the Lie derivative on $\omega_{\mc C/\mc B}$.
\end{rem}

\begin{rem}\label{lbb7}
If $\varphi:U_1\rightarrow U_2$ is a fiber-preserving biholomorphism, we have pushforward maps $\mc V_\varrho(\varphi):\SV_\fx^{\leq n}\vert_{U_1} \xrightarrow{\simeq} \SV_\fx^{\leq n}\vert_{U_2}$ and (setting $k=-1$ in Rem. \ref{lbb6}) $\varphi_*:\omega_{\mc C/\mc B}|_{U_1}\xrightarrow{\simeq}\omega_{\mc C/\mc B}|_{U_2}$. Their tensor product is denoted (with abuse of notation) by
\begin{align}
\mc V_\varrho(\varphi):\SV_\fx^{\leq n}\otimes\omega_{\mc C/\mc B}\vert_{U_1} \xrightarrow{\simeq} \SV_\fx^{\leq n}\otimes\omega_{\mc C/\mc B}\vert_{U_2}
\end{align}
and also called the \textbf{pushforward map}. Therefore, if $U\subset \mc C$ is open and $\xk\in H^0(U,\Theta_\MC)$ is fiber-preserving, using the pushforward of the flow $\varphi^\xk_\zeta$ we can define the \textbf{Lie derivative} 
\begin{align}
\mc L_\xk:H^0(U,\SV_{\fk X}^{\leq n}\otimes\omega_{\mc C/\mc B})\rightarrow H^0(U,\SV_{\fk X}^{\leq n}\otimes\omega_{\mc C/\mc B})
\end{align}
using \eqref{lie1}. Its local expression is ``the sum of the formulas \eqref{eqb24} and \eqref{eqb42}": Let $v\in\SV_{\fk X}^{\leq n}\otimes\omega_{\mc C/\mc B}(U)$. If $\xk$ is written as \eqref{eqb21}, and if the trivialization \eqref{eqb43} sends $v$ to
\begin{align*}
\mc U_\varrho(v)=ud\eta
\end{align*}
where $u\in \Vbb^{\leq n}\otimes_\Cbb\mc O(U)$, then in $\Vbb^{\leq n}\otimes_\Cbb\mc O(U)\otimes_\Cbb d\eta$ we have
\begin{align}\label{eqb44}
\begin{aligned}
&\mc U_\varrho(\eta)\mc L_\xk v\\
=&\Big(h(\eta,\tau_\blt)\partial_\eta u+\sum_{j=1}^m g_j(\tau_\blt)\partial_{\tau_j}u-\sum_{k\geq 1}\frac{1}{k!}\partial_\eta^k h(\eta,\tau_\blt)L(k-1)u+\partial_\eta h(\eta,\tau_\blt)\cdot u\Big)\cdot d\eta
\end{aligned}
\end{align}
\end{rem}

\begin{rem}\label{lieder1}
Assume that $U\subset \mc C$ is open and $\xk\in H^0(U,\Theta_\MC(\blt\SX))$ is fiber-preserving. (So $\xk$ can be viewed as a fiber-preserving element of $\Theta_\MC(U-\SX)$ with finite poles at $\SX$.) Then the Lie derivative in Rem. \ref{lbb7} gives rise to
\begin{align}\label{eqb26}
    \ML_\xk:H^0\big(U,\SV_\fx^{\leq n}\otimes \omega_{\MC/\MB}(\blt S_\fx)\big)\rightarrow H^0\big(U,\SV_\fx^{\leq n}\otimes \omega_{\MC/\MB}(\blt S_\fx)\big)
\end{align}
In fact, if we choose $v\in H^0\big(U,\SV_\fx^{\leq n}\otimes \omega_{\MC/\MB}(\blt S_\fx)\big)$, then $v$ can equivalently be viewed as an element of $H^0(U-\SX,\SV_\fx^{\leq n}\otimes \omega_{\MC/\MB})$ with finite poles at $\SX$. Then $\mc L_\xk v$ belongs to  $H^0(U-\SX,\SV_\fx^{\leq n}\otimes \omega_{\MC/\MB})$. In the local expression \eqref{eqb44}, $u$ and $h(\eta,\tau_\blt)$ have finite poles at $\SX$. Thus \eqref{eqb44} implies that $\mc L_\xk v$ has finite poles at $\SX$. Therefore $\mc L_\xk v\in H^0\big(U,\SV_\fx^{\leq n}\otimes \omega_{\MC/\MB}(\blt S_\fx)\big)$.
\end{rem}

\subsection{A commutator formula}
In this section, we give a commutator formula, which will be used in defining the connection $\nabla$. As in the previous section, we assume that $\wtd{\fk X}$ is smooth, i.e., $R=0$ (cf. Asmp. \ref{lbb1}). So  $\fx=\wtd{\fx}=(\pi:\MC\rightarrow\MB\big| \sgm_1,\cdots,\sgm_N)$. 

Choose an open subset $U\subset \MC$ such that there exists $\eta\in\mc O(U)$ univalent on each fiber, and that $\pi(U)$ admits coordinates $\tau_\blt=(\tau_1,\tau_2,\cdots)$. Write $\tau_\blt\circ \pi$ as $\tau_\blt$ for simplicity.   We shall define a morphism of $\MO_\MC(U)$-modules
\begin{align}\label{commutator1}
    \Lbf:\SV_\fx\otimes \omega_{\MC/\MB}(\blt S_\fx)(U)\rightarrow\send_{\MO_U}(\SV_\fx(\bullet S_{\mathfrak{X}})\vert_U)
\end{align}

Fix the following identification 
\begin{equation}\label{commutator6}
\begin{gathered}
    U\simeq(\eta,\tau_\blt)(U)\subset \Cbb\times \MB,\, \eta=z\text{ as the standard coordinate of }\Cbb\\
    \SV_\fx\vert_U \simeq \Vbb\otimes_\Cbb \MO_U,\, \SV_\fx\otimes \omega_{\MC/\MB}\vert_U \simeq \Vbb \otimes_\Cbb \omega_{\MC/\MB}\vert_U,\quad \text{via }\MU_\varrho(\eta)=\MV_\varrho(\eta).
\end{gathered}
\end{equation}
Choose any $udz=u(z,\tau_\blt)dz\in \Vbb\otimes_\Cbb\omega_{\MC/\MB}(\blt S_\fx)(U)$, open subset $V\subset U$, and $v=v(z,\tau_\blt)\in \Vbb\otimes_\Cbb \MO_U(\blt S_\fx)(V)$. Define the action of $udz$ on $v$ to be 
\begin{subequations}\label{commutator23}
\begin{align}\label{commutator2}
    \Lbf_{udz}v=(\Lbf_{udz}v)(z,\tau_\blt)=\Res_{\zeta-z=0}Y\big(u(\zeta,\tau_\blt),\zeta-z\big)v(z,\tau_\blt)d\zeta
\end{align}
where $\zeta$ is another distinct standard coordinate of $\Cbb$. By \eqref{commutator2}, it is easy to see 
\begin{align}\label{commutator3}
    \Lbf_{udz}v=\sum_{n\geq 0}\frac{1}{n!}Y\big(\partial_z^n u(z,\tau_\blt)\big)_n v(z,\tau_\blt)
\end{align}
where the sum is finite by lower truncation property.
\end{subequations}
By tensoring \eqref{commutator23} with the identity map of $\omega_{\MC/\MB}$, we get 
\begin{align}\label{commutator4}
    \Lbf:\SV_\fx\otimes \omega_{\MC/\MB}(\blt S_\fx)(U)\rightarrow\send_{\MO_U}\left(\SV_\fx\otimes \omega_{\MC/\MB}\left(\bullet S_{\mathfrak{X}}\right)\vert_U\right)
\end{align}
whose local expression under $\eta$ is 
\begin{equation}\label{commutator5}
    \begin{aligned}
       \Lbf_{udz}vdz&=\Big(\Res_{\zeta-z=0}Y\big(u(\zeta,\tau_\blt),\zeta-z\big)v(z,\tau_\blt)d\zeta \Big)dz\\
       &=\sum_{n\geq 0}\frac{1}{n!}Y\big(\partial_z^n u(z,\tau_\blt)\big)_n v(z,\tau_\blt)dz.
    \end{aligned}
\end{equation}

Now assume that $\MB$ is small enough such that we have local coordinates $\eta_1,\cdots,\eta_N$ at $\sgm_1(\MB),\cdots,\sgm_N(\MB)$. %Associate a finitely admissible $\Vbb^{\times N}$-module $\Wbb$ to $\sgm_\blt$. 
Fix $1\leq i\leq N$ and set $\eta=\eta_i$. Let $U=U_i$ be a neighborhood of $\sgm_i(\MB)$ on which $\eta_i$ is defined. Fix the identifications given in \eqref{commutator6} and 
\begin{align}
    \SW_\fx(\Wbb)\simeq \Wbb\otimes_\Cbb \MO_\MB,\quad \text{via }\MU(\eta_\blt).
\end{align}

\begin{pp}\label{commutator12345}
    For any $udz,vdz\in \Vbb\otimes_\Cbb \omega_{\MC/\MB}(\blt S_\fx)(U)$ and $w\in \Wbb\otimes_\Cbb \MO(\MB)$, we have 
    \begin{align}\label{commutator12}
        udz*_i (vdz*_i w)-vdz*_i(udz*_i w)=(\Lbf_{udz}vdz)*_i w.
    \end{align}
\end{pp}

\begin{proof}
    Since residue actions can be defined fiberwise, it suffices to assume that $\MB$ is a single point. So we can suppress the symbol $\tau_\blt$ and assume that $U=\eta(U)$ is an open disk centered at $\eta_i(x_i)=0$. Now $w$ is a vector of $\Wbb$.

%In this proof, $\Wbb=(\Wbb,Y_i)$ is understood as a finitely admissible $\Vbb$-module, whose contragredient module is denoted by $\Wbb'$ by abuse of notations. Moreover, assume $w$ is a constant section of $\Wbb\otimes \MO_\MB$.

We view $\Wbb$ as an admissible (i.e. $\Nbb$-gradable) $\Vbb$-module with vertex operator $\mbb Y_i$ and grading $\bigoplus_{n\in\Nbb}\Wbb(n)$. Choose any $w'\in\Wbb'=\bigoplus_{n\in\Nbb}\Wbb(n)^*$. By the duality property of admissible $\Vbb$-modules, there exists
    \begin{align}
    f=f(z,\zeta)\in \Conf^2(U-\{0\}).
    \end{align}
such that for each fixed $z$, the series expansions of $f$ (with respect to $\zeta$) near $0$ and $z$ are
    \begin{subequations}\label{commutator789}
    \begin{gather}
        \alpha_z(\zeta)=\<w',Y_i(v,z)Y_i(u,\zeta)w\>\in \Cbb((\zeta)),\label{commutator7}\\
        \gamma_z(\zeta-z)=\<w',Y_i(Y(u,\zeta-z)v,z)w\>\in \Cbb((\zeta-z)),\label{commutator8}
    \end{gather}
    and that for each fixed $\zeta$, the series expansion of $f$ (with respect to $z$) near $0$ is 
    \begin{gather}\label{commutator9}
        \beta_\zeta(z)=\<w',Y_i(u,\zeta)Y_i(v,z)w\>\in \Cbb((z)).
    \end{gather}
    \end{subequations}
(This is well-known when $u,v$ are constant sections and $U=\Cbb$.  See \cite[Prop. 5.1.2]{FHL93} or \cite{GuiLec} Subsec. 7.13 and 9.4.)

    Choose circles $C_1,C_2,C_3\subset U$ centered at $0$ with radii $r_1,r_2,r_3$ respectively satisfying $r_1<r_2<r_3$. For each $z\in C_2$, choose a circle $C(z)$ centered at $z$ with radius less than $r_2-r_1$ and $r_3-r_2$. Then by Cauchy's theorem in complex analysis, for each fixed $z\in C_2$,
    \begin{align}\label{commutator10}
        \oint_{C_1}f(z,\zeta)\frac{d\zeta}{2\pi \im}+\oint_{C(z)}f(z,\zeta)\frac{d\zeta}{2\pi \im}=\oint_{C_3}f(z,\zeta)\frac{d\zeta}{2\pi \im}.
    \end{align}
    Note that 
    \begin{align*}
        \<w',vdz*_i (udz*_i w)\>\xlongequal{\eqref{commutator7}} \Res_{z=0}(\Res_{\zeta=0}\alpha_z(\zeta)d\zeta)dz&=\oint_{C_2}\Big(\oint_{C_1}f(z,\zeta)\frac{d\zeta}{2\pi \im}\Big)\frac{dz}{2\pi \im}\\
        \<w',udz*_i (vdz*_i w)\>\xlongequal{\eqref{commutator9}} \Res_{\zeta=0}(\Res_{z=0}\alpha_z(\zeta)d\zeta)dz&=\oint_{C_3}\Big(\oint_{C_2}f(z,\zeta)\frac{dz}{2\pi \im}\Big)\frac{d\zeta}{2\pi \im}\\
        &=\oint_{C_2}\Big(\oint_{C_3}f(z,\zeta)\frac{d\zeta}{2\pi \im}\Big)\frac{dz}{2\pi \im}\\
        \<w',(\Lbf_{udz}vdz)*_i w\>\xlongequal{\eqref{commutator8}} \Res_{z=0}(\Res_{\zeta-z}\gamma_z(\zeta-z)d\zeta)dz&=\oint_{C_2}\Big(\oint_{C(z)}f(z,\zeta)\frac{d\zeta}{2\pi \im}\Big)\frac{dz}{2\pi \im}.
    \end{align*}
    These identities, together with $\oint_{C_2}\eqref{commutator10}\frac{dz}{2\pi \im}=0$, prove
    \begin{align}\label{commutator11}
        \<w',udz*_i (vdz*_i w)\>-\<w',vdz*_i(udz*_i w)\>=\<w',(\Lbf_{udz}vdz)*_i w\>.
    \end{align}
    Since \eqref{commutator11} holds for all $w'\in \Wbb'$, we get \eqref{commutator12}.
\end{proof}
\begin{rem}
    If $U_1,\cdots,U_N$ are disjoint neighborhoods of $\sgm_1(\MB),\cdots,\sgm_N(\MB)$ and $\sigma\in H^0\big(U_1\cup \cdots \cup U_N,\SV_\fx\otimes \omega_{\MC/\MB}(\blt S_\fx)\big)$, Prop. \ref{commutator12345} tells us 
    \begin{align}\label{commutator666}
        \sigma\cdot (udz*_i w)-udz*_i(\sigma\cdot w)=(\Lbf_{\sigma}udz)*_i w,
    \end{align}
    where $\Lbf_{\sigma}udz$ is understood as $\Lbf_{\sigma\vert_{U_i}}(udz)$. It is because $*_i$ commutes with $*_j$ when $i\ne j$.
\end{rem}
The following lemma is also useful when defining connections.
\begin{lm}\label{commutator777}
    For any $vdz\in \Vbb\otimes_\Cbb \omega_{\MC/\MB}(\blt S_\fx)(U)$ and $w\in \Wbb\otimes_\Cbb \MO(\MB)$, 
    \begin{align}
        \big((\partial_z+L(-1)) vdz\big)*_i w=0.
    \end{align}
where the partial derivative $\partial_z$ is assumed to be perpendicular to $d\pi$.
\end{lm}
\begin{proof}
    It follows from
    \begin{align*}
        \partial_z Y_i(v,z)=Y_i(\partial_z v,z)+Y_i(L(-1)v,z)
    \end{align*}
    and $\Res_z \partial_z(\cdots)\cdot dz=0$.
\end{proof}

\subsection{The differential operators $\nabla_\yk$}\label{lbb17}

Recall \eqref{eqb45} that $\Delta\subset\mc B$ is the discriminant locus of $\mc B$.

\begin{df}\label{lbb18}
Let $\scr E$ be an $\mc O_\MB$-module. Let $\yk\in\Theta_\MB(\mc B)$. A \textbf{differential operator} $\nabla_\yk$ on $\scr E$ denotes a sheaf map $\scr E\rightarrow\scr E$ satisfying the Leibniz rule. In other words, for each open $U\subset\mc B$ we have a map $\nabla_\yk:\scr E(U)\rightarrow\scr E(U)$ satisfying $\nabla_\yk(fs)=(\yk f)\cdot s+f\cdot \nabla_\yk s$ for each $f\in\mc O(U)$ and $s\in\scr E(U)$. Moreover, if $V\subset U$ is open, then $\nabla_\yk(s|_V)=(\nabla_\yk s)|_V$.

If $\nabla_\yk$ is a differential operator on $\scr E$, then we also have a differential operator on $\scr E^\vee$, called the \textbf{dual differential operator} and also denoted by $\nabla_\yk$, defined by
    \begin{align}\label{notation1}
        \<\nabla_\yk \varphi,s\>=\yk\<\varphi,s\>-\<\varphi,\nabla_\yk s\>
    \end{align}
    where $\varphi\in \scr E^\vee(U)=\Hom_{\MO_U}(\scr E_U,\MO_U),\yk\in \Theta(U)$ and $s\in \scr E_U$.

A \textbf{connection} (resp. \textbf{logarithmic connection}) $\nabla$ on $\scr E$ associates to each open $U\subset\mc B$ and each $\yk\in\Theta_\MB(U)$ (resp. each $\yk\in\Theta_\MB(-\log\Delta)(U)$) a differential operator $\nabla_\yk$ on $\scr E|_U$ satisfying $\nabla_\yk|_V=\nabla_{\yk|_V}$ for each open $V\subset U$, and $\nabla_{f\yk}=f\nabla_\yk$ for each $f\in\mc O(U)$. \hfill\qedsymbol
\end{df}

%In case that $\Delta=\emptyset$, a logarithmic connection is simply called \textbf{connection}, which associates to each open set $U\subset\mc B$ a map $\Theta_\MB(U)\times\scr E(U)\rightarrow\scr E(U)$. (For example, even if $\Delta\neq\emptyset$, we can talk about connections on $\scr E|_{\MB-\Delta}$.)

The goal of this section is to define a logarithmic connection $\nabla$ on $\scr W(\Wbb)$, and show that it descends to a connection on the sheaf of coinvariants $\ST_\fx(\Wbb)\vert_{\MB-\Delta}$. The connection $\nabla$ on $\scr W(\Wbb)$ is important for our later proof of the convergence of sewing conformal blocks. The definition of $\nabla$ on $\scr W(\Wbb)$ is well-known, cf. \cite{TUY,FB04,DGT1}. However, our (differential geometric) proof that $\nabla$ descends to $\ST_\fx(\Wbb)\vert_{\MB-\Delta}$ is new and, in particular, differs from the algebraic proof in \cite{FB04}.

\begin{ass}\label{lbb8}
Throughout this section, we make the following assumptions.
\begin{enumerate}
\item[(1)] We assume that $\wtd\fx=\eqref{eqb1}$ has local coordinates $\eta_1,\cdots,\eta_N$ at the marked points $\sgm_1(\wtd \MB),\cdots,\sgm_N(\wtd \MB)$ where each $\eta_i$ is defined on a neighborhood $\wtd U_i$ of $\sgm_i(\wtd\MB)$. We assume that $\wtd U_1,\dots,\wtd U_N,V_1',\dots,V_R',V_1'',\dots,V_R''$ are mutually disjoint. By constant extension, $\eta_i$ becomes local coordinates of $\fx$ at the marked point $\sgm_i(\MB)$ defined on $U_i=\wtd U_i\times\mc D_{r_\blt\rho_\blt}$, cf. Rem. \ref{lbb9}. 
\item[(2)] We do not assume that $\wtd\MB$ has a set of coordinates. However, if a set of coordinates $\tau_\blt=(\tau_1,\cdots,\tau_m):\wtd \MB\rightarrow \Cbb^m$ has been chosen, by abuse of notations, we write $(q_\blt,\tau_\blt)=(q_\blt,\tau_\blt)\circ \pi$ for simplicity. So $(\eta_i,q_\blt,\tau_\blt)$ becomes a set of coordinate of $U_i$.
\end{enumerate}
\end{ass}

\subsubsection{The action of $\nabla_\yk$ on $\SW_\fx(\Wbb)$}
The map $d\pi:\Theta_\MC(-\log\MC_\Delta)\rightarrow\pi^*\Theta_\MB(-\log\Delta)$ (cf. \eqref{shortexact}) gives rise to a map
\begin{align}\label{eqb46}
H^0\big(\MC,\Theta_\MC(-\log \MC_\Delta+\blt S_\fx)\big)\xlongrightarrow{d\pi} H^0\big(\MC,\pi^* \Theta_\MB(-\log \Delta)(\blt S_\fx)\big)
\end{align}
Choose $\yk\in H^0\big(\MB,\Theta_\MB(-\log \Delta)\big)$. Its pullback $\pi^* \yk$ is in $H^0\big(\MC,\pi^*\Theta_\MB(-\log\Delta)\big)$ and hence in $H^0\big(\MC,\pi^* \Theta_\MB(-\log \Delta)(\blt S_\fx)\big)$.

\begin{df}\label{lbb50}
 We say that $\wtd\yk\in H^0\big(\MC,\Theta_\MC(-\log \MC_\Delta+\blt S_\fx)\big)$ is a \textbf{lift} of $\yk$ if it is sent by \eqref{eqb46} to $\pi^*\yk$.
\end{df}

\begin{rem}\label{lbb12}
Suppose that $\wtd\MB$ is a Stein manifold (and hence $\MB$ is Stein). Then any $\yk$ has a lift. This is because  \eqref{shortexact} gives rise to a short exact sequence 
\begin{equation*}
\begin{aligned}
    0\rightarrow H^0\big(\MB,\pi_*\Theta_{\MC/\MB}(\blt S_\fx)\big)\rightarrow &H^0\big(\MB,\pi_*\Theta_\MC(-\log \MC_\Delta+\blt S_\fx)\big)\\
    \xrightarrow{d\pi}&H^0\big(\MB,\pi_*\big(\pi^* \Theta_\MB(-\log \Delta)(\blt S_\fx)\big)\big)\rightarrow 0.
\end{aligned}
\end{equation*}
namely, a short exact sequence
\begin{align}\label{eqb29}
\begin{aligned}
    0\rightarrow H^0\big(\MC,\Theta_{\MC/\MB}(\blt S_\fx)\big)\rightarrow &H^0\big(\MC,\Theta_\MC(-\log \MC_\Delta+\blt S_\fx)\big)\\
    \xrightarrow{d\pi}&H^0\big(\MC,\pi^* \Theta_\MB(-\log \Delta)(\blt S_\fx)\big)\rightarrow 0
\end{aligned}
\end{align}
See \cite[Sec. 11]{Gui-sewingconvergence} the paragraph around (11.3).
\end{rem}

Choose 
\begin{align*}
\yk\in H^0(\mc B,\Theta_\MB(-\log\Delta))\qquad\text{with a lift}\qquad\wtd\yk\in H^0\big(\MC,\Theta_\MC(-\log \MC_\Delta+\blt S_\fx)\big)
\end{align*}
We shall define a differential operator $\nabla_\yk$ on $\scr W_\fx(\Wbb)$. 

First consider the special case that $\wtd\MB$ has coordinates $\tau_\blt$ as in Asmp. \ref{lbb8}-(2). Then  $U_i$ has a set of coordinates $(\eta_i,q_\blt,\tau_\blt)$. 
%Identify $U_i$ with its image via $(\eta_i,q_\blt,\tau_\blt)$. Then $\eta_i$ is identifies with the standard coordinate $z$ of $\Cbb$. 
Write 
\begin{align}\label{eqb30}
    \yk=\sum_{j=1}^m g_j(q_\blt,\tau_\blt)\partial_{\tau_j}+\sum_{r=1}^R f_r(q_\blt,\tau_\blt)\partial_{q_r}.
\end{align}
where $g_j,f_r\in\mc O((q_\blt,\tau_\blt)(\mc B))$. Then we can find $h_i\in \mc O((\eta_i,q_\blt,\tau_\blt)(U_i-\SX))$ such that
\begin{align}\label{eqb31}
    \wtd \yk\vert_{U_i}=h_i(\eta_i,q_\blt,\tau_\blt)\partial_{\eta_i} +\sum_{j=1}^m g_j(q_\blt,\tau_\blt)\partial_{\tau_j}+\sum_{r=1}^R f_r(q_\blt,\tau_\blt)\partial_{q_r},
\end{align}
where $\eta_i^k h_i(\eta_i,q_\blt,\tau_\blt)$ is holomorphic on $U_i$ for some $k\in \Nbb$. Recall that $\cbf$ is the conformal vector. Define
\begin{gather}\label{eqb32}
\begin{gathered}
\upnu(\wtd \yk)\in H^0(U_1\cup \cdots \cup U_N,\SV_\fx\otimes \omega_{\MC/\MB}(\blt S_\fx))\\
 \MU_\varrho(\eta_i)\upnu(\wtd \yk)\vert_{U_i}=h_i(\eta_i,q_\blt,\tau_\blt)\cbf d\eta_i
\end{gathered}
\end{gather}

\begin{rem}
It is easy to see that the definition of $\upnu(\wtd\yk)$ is independent of the coordinates $\tau_\blt$ of $\wtd\MB$. Thus, $\upnu(\wtd \yk)$ can be defined globally if $\wtd\MB$ does not have a set of coordinates. Therefore, in the following, we do not assume that $\wtd\MB$ has a set of coordinates $\tau_\blt$.
\end{rem}

\begin{df}\label{lbb16}
Identify 
\begin{align}\label{eqb33}
    \SW_\fx(\Wbb)=\Wbb\otimes_\Cbb \MO_\MB\qquad \text{via }\MU(\eta_\blt).
\end{align}
The \textbf{differential operator} $\pmb{\nabla_\yk}:\scr W_\fx(\Wbb)\rightarrow\scr W_\fx(\Wbb)$ \index{zz@$\nabla_\yk$} (which replies on $\wtd\yk$ and $\eta_\blt$) is defined as follows. For each open $V\subset\mc B$ and $w\in\Wbb\otimes_\Cbb\mc O(V)$,
\begin{align}\label{connectiondef1}
   \boxed{~\nabla_\yk w=\yk (w)-\upnu(\wtd \yk)\cdot w~}
\end{align}
where $\upnu(\wtd \yk)\cdot w$ is the residue action of $\upnu(\wtd\yk)$ on $w$ (Subsec. \ref{lbb10}). In particular, $\nabla_\yk$ restricts to a differential operator on $\scr W_\fx(\Wbb)|_{\MB-\Delta}$.
\end{df}

\begin{df}\label{lbb13}
Assume that $\wtd\MB$ has a set of coordinates $(\tau_1,\dots,\tau_m)$. Recall that $q_i\partial_{q_1},\dots,q_R\partial_{q_R},\partial_{\tau_1},\dots,\partial_{\tau_m}$ are a set of free generators of $\Theta_\MB(-\log(\Delta))$ (cf. Def. \ref{lbb11}). Assume that each of them has a lift so that we have differential operators
    \begin{align*}
    \nabla_{q_1\partial_{q_1}},\cdots,\nabla_{q_R\partial_{q_R}},\nabla_{\partial_{\tau_1}},\cdots,\nabla_{\partial_{\tau_m}}
    \end{align*}
(This is true e.g. when $\wtd\MB$ is Stein, cf. Rem. \ref{lbb12}.) We define the \textbf{logarithmic connection $\pmb\nabla$ on $\pmb{\scr W_\fx(\Wbb)}$} to be the unique one extending the above differential operators. It clearly restricts to a connection on $\scr W_\fx(\Wbb)|_{\mc B-\Delta}$.
\end{df}

\begin{comment}
\eqref{connectiondef} and \eqref{connectiondef3} are clearly independent of the choice of local coordinates of $\MB$. In order to prove that $\nabla$ defined in \eqref{connectiondef3} actually gives a connection on the $\MO_{\MB-\Delta}$-module $\ST_\fx(\Wbb)$, it remains to answer the following questions.
    \begin{enumerate}
        \item[Q1.] Can \eqref{connectiondef3} decend to $\ST_\fx(\Wbb)$?
        \item[Q2.] Is \eqref{connectiondef3} independent of the choice of the lift $\wtd \yk$?
        \item[Q3.] Is \eqref{connectiondef3} independent of the choice of $\eta_\blt$?
    \end{enumerate}
\end{comment}

\subsubsection{Descending $\nabla$ to $\ST_{\fx}(\Wbb)$}\label{lbb66}

\begin{thm}\label{independent1}
Choose $\yk\in H^0(\mc B,\Theta_\MB(-\log\Delta))$ with a lift $\wtd\yk$. Then the restriction of $\nabla_\yk$ to $\scr W_\fx(\Wbb)|_{\MB-\Delta}$ descends to a differential operator on $\ST_\fx(\Wbb)\vert_{\MB-\Delta}$.
\end{thm}

It follows that in the setting of Def. \ref{lbb13}, $\nabla$ descends to a connection on $\scr T_\fx(\Wbb)|_{\mc B-\Delta}$.

\begin{proof}
Since $\fx$ is smooth outside $\Delta$, by replacing $\wtd\fx$ with $\fk X|_{\MB-\Delta}$, it suffices to assume that $R=0$ and hence $\fx=\wtd\fx$. (In particular, $\Delta=\emptyset$.) So there are no $q_\blt$-variables. Assume the identification \eqref{eqb33}. Choose an open subset $V\subset \MB$, a section $w\in \Wbb\otimes \MO(V)$ and 
\begin{align*}
    v\in \pi_*\big(\SV_\fx\otimes \omega_{\MC/\MB}(\blt S_\fx)\big)(V)=H^0\big(\MC_V,\SV_\fx\otimes \omega_{\MC/\MB}(\blt S_\fx)\big).
\end{align*}
Let us show that $[\nabla_\yk,v]=\ML_{\wtd \yk}v$ when acting on $w$, i.e.,
\begin{align}\label{eqb34}
    \nabla_\yk(v\cdot w)=v\cdot \nabla_\yk w+\ML_{\wtd \yk}v\cdot w.
\end{align}
Then, since $\mc L_{\wtd\yk}v$ belongs to $H^0\big(\MC_V,\SV_\fx\otimes \omega_{\MC/\MB}(\blt S_\fx)\big)$ (cf. Rem. \ref{lieder1}), $\nabla_\yk$ preserves the denominator sheaf ${\pi_*\big(\SV_{\fx}\otimes \omega_{\MC/\MB}(\blt S_\fx)\big)\cdot \SW_\fx(\Wbb)}$ in the definition of $\scr T_\fx(\Wbb)$ (cf. Def. \ref{lbb14}). Then the theorem clearly follows.

%In order to prove \eqref{eqb34}, let us consider the local expression of both sides. Since we only need to consider the smooth family $\fx_{\MB-\Delta}$, we may assume $\fx$ is a smooth family (i.e. $\Delta=\emptyset$, $R=0$, so that we can suppress the symbols $q_\blt$) and $U=\MB$.

It suffices to prove the theorem locally with respect to $\mc B$. Therefore, we assume that $\mc B$ admits coordinates $\tau_\blt$, and that $\mc B=V$. Make the following identifications via $\mc U_\varrho(\eta_i)$:
\begin{gather*}
%   U_i=(\eta_i,\tau_\blt)(U_i)\qquad\text{via }(\eta_i,\tau_\blt)\\
   \SV_\fx\vert_{U_i}= \Vbb\otimes_\Cbb \MO_{U_i}\qquad
   \SV_{\fx}\otimes \omega_{\MC/\MB}\vert_{U_i}= \Vbb\otimes_\Cbb \MO_{U_i} d\eta_i\qquad 
\end{gather*}  
So $v\vert_{U_i}=u_id\eta_i$ for some $u_i\in \Vbb\otimes_\Cbb \MO_{U_i}(\blt S_\fx)(U_i)$. By \eqref{eqb30}, \eqref{eqb31}, and $R=0$, we have
\begin{align*}
\yk=\sum_{j=1}^m g_j\partial_{\tau_j}\qquad \wtd\yk|_{U_i}=h_i\partial_{\eta_i}+\sum_{j=1}^m g_j\partial_{\tau_j}
\end{align*}
where we have abbreviated $g_j(\tau_\blt)$ to $g_j$ and $h_i(\eta_i,\tau_\blt)$ to $h_i$. Then by \eqref{connectiondef1},
\begin{align}
&\nabla_\yk(v*_i w)=\nabla_\yk(u_id\eta_i*_i w)=\sum_j g_j \partial_{\tau_j} (u_id\eta_i*_i w)-\upnu(\wtd \yk)\cdot ((u_id\eta_i)*_i w)\nonumber\\
    \xlongequal{\eqref{commutator666}}&\sum_j g_j\big((\partial_{\tau_j}u_i)d\eta_i *_i w\big)+\sum_jg_j\big((u_i d\eta_i) *_i (\partial_{\tau_j}w)\big)\nonumber\\
    &- u_id\eta_i*_i (\upnu(\wtd \yk)\cdot w)-(\Lbf_{\upnu(\wtd \yk)}u_id\eta_i)*_i w\nonumber\\
    =&v*_i \nabla_\yk w+\sum_j g_j\big((\partial_{\tau_j}u_i)d\eta_i *_i w\big)-(\Lbf_{\upnu(\wtd \yk)}u_id\eta_i)*_i w\label{eqb35}
\end{align}
where $v*_i \nabla_\yk w=(u_id\eta_i)*_i\eqref{connectiondef1}$ equals the sum of the second and the third terms of the second last expression of \eqref{eqb35}. We claim that
\begin{align}\label{eqb36}
(\ML_{\wtd \yk}u_id\eta_i)*_i w=\sum_j g_j\big((\partial_{\tau_j}u_i)d\eta_i *_i w\big)-(\Lbf_{\upnu(\wtd \yk)}u_id\eta_i)*_i w.
\end{align}
where $\ML_{\wtd \yk}u_id\eta_i$ is the Lie derivative of $u_id\eta_i$, cf. Rem. \ref{lbb7}. If \eqref{eqb36} can be proved, then \eqref{eqb34} follows immediately by summing up \eqref{eqb35} over all $1\leq i\leq N$.

By \eqref{eqb32}, we have $\upnu(\wtd\yk)|_{U_i}=h_i\cbf d\eta_i$. Thus, by \eqref{commutator5},
\begin{align*}
    \Lbf_{\upnu(\wtd\yk)}u_id\eta_i=\sum_{k\geq 0}\frac{1}{k!}\partial_{\eta_i}^k h_i\cdot L(k-1)\cdot u_id\eta_i.
\end{align*}
Now \eqref{eqb44} reads
\begin{align*}
\begin{aligned}
    \ML_{\wtd \yk}(u_id\eta_i) =h_i\partial_{\eta_i} u_i d\eta_i+\sum_{j=1}^m g_j\partial_{\tau_j} u_id\eta_i-\sum_{k\geq 1}\frac{1}{k!}\partial_{\eta_i}^k h_i\cdot L(k-1)u_id\eta_i+\partial_{\eta_i}h_i\cdot u_id\eta_i
\end{aligned}
\end{align*}
So $\Lbf_{\upnu(\wtd\yk)}u_id\eta_i+\ML_{\wtd \yk}u_id\eta_i$ equals
\begin{align*}
\big(L(-1)+\partial_{\eta_i}\big)(h_iu)d\eta_i+\sum_{j=1}^m g_j\partial_{\tau_j} u_id\eta_i
\end{align*}
This, together with Lem. \ref{commutator777}, proves \eqref{eqb36}.
\end{proof}

\begin{rem}\label{threeques}
Assume that $R=0$. Suppose that  $\MB$ is Stein (so that the lift $\wtd\yk$ always exists) and have a set of coordinates $\tau_\blt$, the differential operator $\nabla_\yk$ relies on the lift $\wtd\yk$ and the local coordinates $\eta_\blt$. Thus, the connection $\nabla$ on $\scr T_\fx(\Wbb)$ defined by Def. \ref{lbb13}  relies not only on the lifts and on $\eta_\blt$, but also on the coordinates $\tau_\blt$ (since $\nabla$ relies on the choice of free generators $\partial_{\tau_1},\dots,\partial_{\tau_m}$ of $\Theta_\MB$).
\end{rem}

\subsection{Dependence of $\nabla_\yk$ on the lift $\wtd{\yk}$}
\label{lift1}

In this section, we assume that $R=0$. Hence $\fx=\wtd\fx$ and $\Delta=\emptyset$. We assume Asmp. \ref{lbb8}. (In particular, the local coordinates $\eta_\blt$ of $\fx$ are chosen.)

Suppose that $\yk\in H^0(\MB,\Theta_\MB)$ has two lifts $\wtd\yk,\wtd\yk'\in H^0(\MC,\Theta_\MC(\blt\SX))$. Then $\wtd\obf:=\wtd\yk'-\wtd\yk'$ is a lift of the zero vector field $0$.  Since $\wtd\obf$ is tangent to $d\pi$, we have
\begin{align*}
\wtd\obf\in H^0(\MC,\Theta_{\MC/\MB}(\blt\SX))
\end{align*}
(This also follows from \eqref{eqb29} if $\MB$ is Stein. Hence it holds for general $\MB$ by gluing local data.) Conversely, if $\wtd\yk$ is a lift of $\yk$, and if $\wtd\obf\in H^0(\MC,\Theta_{\MC/\MB}(\blt\SX))$, then $\wtd\yk':=\wtd\yk+\wtd\obf$ is clearly also a lift of $\yk$. Therefore, to study the dependence of $\nabla_\yk$ on the lift, it suffices to study the action of $H^0(\MC,\Theta_{\MC/\MB}(\blt\SX))$ on $\scr T_\fx(\Wbb)$.

\begin{df}
Choose an open subset $U\subset \MC$ and  $\eta\in\mc O(U)$ univalent on each fiber. If $f\in \MO(U)$ satisfies that $\partial_\eta f$ is nowhere zero, then the \textbf{Schwarzian derivative of $f$ over $\eta$} is defined to be 
\begin{align}
\Sbf_\eta f=\frac{\partial_\eta^3 f}{\partial_\eta f}-\frac{3}{2}\Big(\frac{\partial_\eta^2 f}{\partial_\eta f}\Big)^2
\end{align}
where the partial derivative $\partial_\eta$ is defined with respect to $(\eta,\pi)$, i.e., it is annihilated by $d\pi$ and restricts to $d/d\eta$ on each fiber. 
\end{df}

\begin{df}
    An open cover $(U_\alpha,\eta_\alpha)_{\alpha\in \FA}$ of $\MC$, where each open set $U_\alpha$ is equipped with a holomorphic function $\eta_\alpha\in \MO(U_\alpha)$ univalent on each fiber, is called a \textbf{projective chart} if for any $\alpha,\beta\in \FA$, we have $\Sbf_{\eta_\beta}\eta_\alpha=0$ on $U_\alpha\cap U_\beta$. Two projective charts are called equivalent if their union is a projective chart. A maximal projective atla (equivalently, an  equivalence class of projective charts) is called a \textbf{projective structure}.
\end{df}

\begin{thm}\label{lbb47}
Suppose that $\MB$ is Stein. Then $\fx$ has at least one projective structure.
\end{thm}

\begin{proof}
See \cite[Thm. B.2]{Gui-sewingconvergence}.
\end{proof}

\begin{df}
Let $\mbf P$ be a projective chart on $\fx$. Choose an open subset $U\subset \MC$ and $\eta\in \MO(U)$ univalent on each fiber. One can check that for each $(V,\mu),(W,\nu)\in\mbf P$ we have $\Sbf_\eta\mu=\Sbf_\eta\nu$ on $U\cap V\cap W$. (See \cite[Rem. 8.3]{Gui-sewingconvergence}.) Thus there is a unique \index{SP@$\Sbf_\eta\mbf P$}
    \begin{align*}
        \Sbf_\eta \mbf P \in \MO(U)
    \end{align*}
such that $\Sbf_\eta\mbf P|_{U\cap V}=\Sbf_\eta\mu|_{U\cap V}$ for each $(V,\mu)\in\mbf P$.  
\end{df}

Recall that $c$ is the central charge of $\Vbb$.

\begin{thm}\label{lift2}
    Suppose that $\fx$ has a projective structure $\mbf P$. Choose an element $\wtd\obf\in H^0\big(\MB,\pi_*\Theta_{\MC/\MB}(\blt S_\fx)\big)$. Let $a_i\in H^0(U_i,\mc O_{\MC}(\blt\SX))$ such that
\begin{align}\label{localexpression1}
    \wtd\obf|_{U_i}=a_i\partial_{\eta_i}
\end{align}
Then the action of $\upnu(\wtd\obf)$ on $\scr W_\fx(\Wbb)(\MB)/\SJ^\pre_\fx(\MB)$ equals the multiplication by  $\#(\wtd\obf)\in\mc O(\MB)$ where
    \begin{align}\label{eqb79}
       \#(\wtd\obf):= \frac{c}{12} \sum_{i=1}^N \Res_{\eta_i=0} \Sbf_{\eta_i}\mbf P \cdot a_id\eta_i
    \end{align}
\end{thm}

Note that $\upnu(\wtd\obf)\in H^0(U_1\cup \cdots \cup U_N,\SV_\fx\otimes \omega_{\MC/\MB}(\blt S_\fx))$ is defined by \eqref{eqb32}, i.e., $\mc U_\varrho(\eta_i)\upnu(\wtd\obf)|_{U_i}=a_i\cbf d\eta_i$. So it relies on $\eta_\blt$.

\begin{proof}
When $\Wbb$ is a tensor product of $\Vbb$-modules, this theorem was proved in \cite[Prop. 9.2]{Gui-sewingconvergence}. The same argument applies to the current general case.
\end{proof}

\begin{df}
We say that the local coordinates $\eta_1,\cdots,\eta_N$ of $\fx$ \textbf{admit a projective structure} of $\fx$ if there is a projective structure containing $(U_1,\eta_1),\cdots,(U_N,\eta_N)$.
\end{df}

\begin{co}\label{lbb15}
Assume that  $\eta_1,\cdots,\eta_N$ admit a projective structure of $\fx$. Then for each $\yk\in H^0(\mc B,\Theta_\MB)$ having a lift, the differential operator $\nabla_\yk$ is independent of the choice of lifts when acting of $\scr T_\fx(\Wbb)$. 
\end{co}

\begin{proof}
This is clear from Thm. \ref{lift2} and the discussion at the beginning of this section.
\end{proof}

\begin{co}
Assume that  $\eta_1,\cdots,\eta_N$ admit a projective structure of $\fx$. Then there is a connection $\nabla$ on $\scr T_\fx(\Wbb)$ such that for each open $V\subset\mc B$, and for each $\yk\in H^0(V,\Theta_\MB)$ with lift $\wtd\yk\in H^0(\MC_V,\Theta_\MC(\blt\SX))$, $\nabla_\yk$ is the differential operator on $\scr T_\fx(\Wbb)|_V$ defined as in Def. \ref{lbb16}.
\end{co}

\begin{proof}
Choose any Stein open set $W\subset\MB$ admitting a set of coordinates $\tau_\blt$. We can use Def. \ref{lbb13} and Thm. \ref{independent1} to define a connection $\nabla^W$ on $\scr T_\fx(\Wbb)|_W$ (by first defining the differential operator $\nabla^W_{\partial_{\tau_i}}$ for each $i$ using a lift $\wtd{\partial_{\tau_i}}$). Now, if $\yk=\sum_i g_i\partial_{\tau_i}\in H^0(W,\Theta_\MB)$ (where $g_i\in\mc O(W)$), then  $\nabla^W_\yk$ is clearly equal to the differential operator defined by the lift
$\sum_i g_i\wtd{\partial_{\tau_i}}$. Therefore, by Cor. \ref{lbb15}, $\nabla^W_\yk$ is also equal to the differential operator $\nabla_\yk$ defined by any other lift $\wtd\yk$. Thus, if for any other Stein open set $W'$ with coordinates we define $\nabla^{W'}$ in a similar way, then $\nabla^W=\nabla^{W'}$ on $W\cap W'$. Therefore, by gluing these locally defined connections, we get a (global) connection on $\scr T_\fx(\Wbb)$ satisfying the desired requirement.
\end{proof}

\subsection{Curvature and projective flatness}
Assume again that $R=0$. Assume  Asmp. \ref{lbb8}.  Assume that $\mc B=\wtd\MB$ has a set of coordinates $\tau_\blt=(\tau_1,\dots,\tau_m)$. %As usual, we abbreviate $\tau_j\circ\pi$ as $\tau_j$ when no confusion arises so that $(\eta_i,\tau_\blt)$ is a set of coordinates of $U_i$. 

Our goal in this section is to calculate the curvature of the connections constructed in Sec. \ref{lbb17}. Choose $\yk,\zk\in H^0(\mc B,\Theta_\MB)$ admitting lifts $\wtd\yk,\wtd\zk$. Then $\nabla_\yk$ and $\nabla_\zk$ are defined in terms of these lifts and $\eta_\blt$. Then $[\wtd\yk,\wtd\zk]$ gives a lift of $[\yk,\zk]$ (cf. \eqref{eqb49}). We use this lift (together with $\eta_\blt$) to define $\nabla_{[\yk,\zk]}$ as in  \eqref{connectiondef1}. Our goal is to calculate the \textbf{curvature}
\begin{align*}
R(\yk,\zk)=\nabla_\yk \nabla_\zk-\nabla_\zk \nabla_\yk -\nabla_{[\yk,\zk]}
\end{align*}
on $\scr T_\fx(\Wbb)$ and on $\scr T^*_\fx(\Wbb)$. (Note that $\nabla_\yk$ on $\scr T_\fx^*(\Wbb)$ is defined by the dual differential operator of $\nabla_\yk$ on $\scr T_\fx(\Wbb)$, cf. Def. \ref{lbb18}.)

On $U_i$ we can write the lifts in the form of \eqref{eqb31}, namely
\begin{subequations}\label{curvature2}
\begin{gather}
    \wtd \yk|_{U_i}=h_i(\eta_i,\tau_\blt)\partial_{\eta_i}+\sum_j g_j(\tau_\blt) \partial_{\tau_j}\\
    \wtd \zk|_{U_i}=k_i(\eta_i,\tau_\blt)\partial_{\eta_i}+\sum_j l_j(\tau_\blt)\partial_{\tau_j}
\end{gather}
\end{subequations}
since there are no variables $q_1,\dots,q_R$. Thus $\yk=\sum_j g_j(\tau_\blt)\partial_{\tau_j}$ and $\zk=\sum_j l_j(\tau_\blt)\partial_{\tau_j}$.

The following formula for $R(\yk,\zk)$ is essentially Prop. 4.2.2-(2) of Ueno's monograph  \cite{Ueno97} in the setting of affine Lie algebras. (See also Prop. 4.6 of \cite{Ueno08}.) We provide detailed calculations for the reader's convenience. (They will also help the reader understand how the geometric setting in \cite{Ueno97} and \cite{Ueno08} matches ours.)

\begin{thm}
    Let $f\in \MO(\MB)$ be 
    \begin{align}\label{eqb50}
        f=-\frac{c}{12}\sum_{i=1}^N \big(\Res_{\eta_i=0} \partial_{\eta_i}^3 h_i(\eta_i,\tau_\blt)\cdot k_i(\eta_i,\tau_\blt)d\eta_i\big).
    \end{align}
Then we have $R(\yk,\zk)=-f$ on $\scr W_\fx(\Wbb)$ and hence on $\ST_\fx(\Wbb)$. Consequently, we have $R(\yk,\zk)=f$ on $\ST^*_\fx(\Wbb)$.
\end{thm}

% Therefore, $\nabla$ is projectively flat and the curvature depend only on $c$ and $\fx$.
\begin{proof}
Let $Y,Z\in H^0(\bigcup_i U_i,\Theta_\MC)$ be defined by $Y=\sum_j g_j(\tau_\blt)\partial_{\tau_j}$ and $Z=\sum_j l_j(\tau_\blt)\partial_{\tau_j}$. Thus $Y$ and $\yk$ (resp. $Z$ and $\zk$) are different elements although they have the same expression due to our convention of abbreviating $\tau_j\circ\pi$ to $\tau_j$. Now \eqref{curvature2} reads
\begin{gather*}
  \wtd \yk|_{U_i}=h_i(\eta_i,\tau_\blt)\partial_{\eta_i}+Y|_{U_i}\qquad
    \wtd \zk|_{U_i}=k_i(\eta_i,\tau_\blt)\partial_{\eta_i}+Z|_{U_i}  \tag{$ \divideontimes$}\label{eqb47}
\end{gather*}
(Thus, although $\yk$ and $Y$ have the same expression, $Y$ is an element of $H^0(\bigcup_i U_i,\Theta_\MC)$ and is not the same as $\yk$. Similarly, $\zk\in H^0(\bigcup_i U_i,\Theta_\MC)$.) By \eqref{connectiondef1}, we have $\nabla_\yk s=\yk s-\sum_i (h_i\cbf d\eta_i)*_is$ and $\nabla_\zk s=\zk s-\sum_i (k_i \cbf d\eta_i)*_i s$. So
\begin{gather*}
        \nabla_\yk \nabla_\zk s=\yk\zk s-\yk\sum_i (k_i \cbf d\eta_i)*_i s-\sum_i (h_i \cbf d\eta_i)*_i\zk s+\sum_{i,j}(h_i\cbf d\eta_i)*_i(k_j \cbf d\eta_j) *_js\\
\nabla_\zk \nabla_\yk s=\zk\yk s-\zk\sum_i (h_i \cbf d\eta_i)*_i s-\sum_i (k_i \cbf d\eta_i)*_i\yk s+\sum_{i,j}(k_j\cbf d\eta_j)*_j(h_i \cbf d\eta_i)*_i s
\end{gather*}
Since $Y,Z$ are orthogonal to $d\eta_i$, we have 
\begin{gather*}
\yk \big((k_i \cbf d\eta_i)*_i s\big)-(k_i \cbf d\eta_i)*_i\yk s=((Y k_i)\cbf d\eta_i)*_is\\
\zk \big((h_i \cbf d\eta_i)*_i s\big)-(h_i \cbf d\eta_i)*_i\zk s=((Z k_i)\cbf d\eta_i)*_is
\end{gather*}
Clearly, if $i\ne j$, the $*_i$-action of $h_i \cbf d\eta_i$ and the $*_j$-action on $k_j\cbf d\eta_j$ on $s$ commute. By Prop. \ref{commutator12345}, the $*_i$-action of $[h_i \cbf d\eta_i,k_i\cbf d\eta_i]$ on $\ST_\fx(\Wbb)$ equals that of
    \begin{gather*}
        [h_i \cbf d\eta_i,k_i\cbf d\eta_i]=\Lbf_{h_i \cbf d\eta_i}k_i\cbf d\eta_i\xlongequal{\eqref{commutator5}}\sum_{n\geq 0} \frac{1}{n!}(\partial_{\eta_i}^n h_i)k_i L(n-1)\cbf d\eta_i\\
        =h_i k_i L(-1)\cbf d\eta_i+2(\partial_{\eta_i} h_i)k_i \cbf d\eta_i+\frac{c}{12}(\partial_{\eta_i}^3 h_i)k_i \ibf d\eta_i.
    \end{gather*}
since $L(0)\cbf=2\cbf,L(1)\cbf=0,L(2)\cbf=L(2)L(-2)\idt=\frac 12\cbf$.    By Lem. \ref{commutator777}, the action of the above expression equals that of 
    \begin{gather*}
-\partial_{\eta_i}(h_i k_i) \cbf d\eta_i+2(\partial_{\eta_i} h_i)k_i \cbf d\eta_i+\frac{c}{12}(\partial_{\eta_i}^3 h_i)k_i \ibf d\eta_i\\
        =(\partial_{\eta_i}h_i) k_i \cbf d\eta_i-h_i(\partial_{\eta_i}k_i) \cbf d\eta_i+\frac{c}{12}(\partial_{\eta_i}^3 h_i)k_i \ibf d\eta_i
    \end{gather*}

Summing up the above calculations, we get
    \begin{equation*}
    \begin{aligned}\label{curvature3}
        [\nabla_\yk,\nabla_\zk]s=&[\yk ,\zk ]s-\sum_i ((Y k_i) \cbf d\eta_i)*_is+\sum_i ((Z h_i)\cbf d\eta_i)*_is \\
        &+\sum_i((\partial_{\eta_i}h_i) k_i \cbf d\eta_i)*_is-\sum_i(h_i(\partial_{\eta_i}k_i) \cbf d\eta_i)*_is-fs   
        \end{aligned}\tag{$\star$}
        \end{equation*}
        On the other hand, by \eqref{eqb47}, the local expression of $[\wtd \yk,\wtd \zk]$ on $U_i$ equals
        \begin{align}\label{eqb49}
            [\wtd \yk,\wtd \zk]|_{U_i}=(h_i\partial_{\eta_i}k_i-k_i\partial_{\eta_i} h_i+Yk_i-Zh_i)\partial_{\eta_i}+[Y,Z]|_{U_i}
        \end{align}
Now, applying \eqref{connectiondef1} to $[\yk,\zk]$ and its lift $[\wtd\yk,\wtd\zk]$, we have (due to \eqref{eqb49})
\begin{align*}
\nabla_{[\yk,\zk]}s=[\yk,\zk]s-\sum_i \big((h_i\partial_{\eta_i}k_i-k_i\partial_{\eta_i} h_i+Yk_i-Zh_i)\cbf d\eta_i\big)*_is
\end{align*}
which equals $\eqref{curvature3}+fs$. Therefore $R(\yk,\zk)s=-fs$.
\end{proof}

\begin{comment}
\begin{rem}\label{oneflat}
    If $\MB$ is 1-dimensional, then $\nabla$ is always flat.
\end{rem}
\end{comment}

\section{Convergence of Virasoro uniformization and local freeness}

Throughout this chapter, we always assume that $\Vbb$ is $C_2$-cofinite, and $\Wbb$ is a grading restricted $\Vbb^{\otimes N}$-module (equivalently, a finitely admissible $\Vbb^{\times N}$-module) associated to the ordered marked points $\sgm_\blt(\MB)$. We assume $R=0$ so that $\fx=\wtd\fx=(\pi:\MC\rightarrow\MB|\sgm_1,\dots,\sgm_N)$ is smooth.

\subsection{Convergence of formal parallel transports}\label{lbb30}

In this section, we fix $0<r\leq+\infty$ and $\mc D_r=\{z\in\Cbb:|z|<r\}$, and assume that
\begin{align*}
\MB=\mc D_r\times\MB_0
\end{align*}
where $\MB_0$ is a complex manifold. Assume that $\fx$ is equipped with local coordinates, i.e.
\begin{align*}
\fx=(\pi:\MC\rightarrow\MB|\sgm_1,\dots,\sgm_N;\eta_1,\dots,\eta_N )\qquad \eta_i\in\mc O(U_i)
\end{align*}
where $U_1,\dots,U_N$ are mutually disjoint neighborhoods of $\sgm_1(\MB),\dots,\sgm_N(\MB)$.

Let $\MC_0=\pi^{-1}(\{0\}\times\mc B_0)$ so that
\begin{align}\label{eqb65}
\fx_0=(\pi_0:\MC_0\rightarrow\MB_0|\sigma_1,\dots,\sigma_N;\mu_1,\dots,\mu_N )
\end{align}
a family of $N$-pointed compact Riemann surfaces with local coordinates. Here, $\sigma_i,\mu_i$ are the restrictions of $\sgm_i,\eta_i$ to $\fx_0$. Similarly, $\pi_0$ is the restriction of $\pi$.

We always let $q$ be the standard coordinate of $\mc D_r$. Assume that $\MB_0$ has a set of coordinates $\tau_1,\dots,\tau_m$. By constant extension, $(q,\tau_\blt)=(q,\tau_1,\dots,\tau_m)$ is a set of coordinates of $\MB$, i.e., it maps $\MB$ biholomorphically to an open subset of $\Cbb^{m+1}$. As usual, we abbreviate $\tau_j\circ\wtd\pi$ and $\tau_j\circ\pi$ to $\tau_j$ when no confusion arises. Fix identifications
\begin{align*}
\scr W_{\fx_0}(\Wbb)=\Wbb\otimes_\Cbb\mc O_{\MB_0}\quad\text{via }\mc U(\mu_\blt)\qquad\qquad  \scr W_\fx(\Wbb)=\Wbb\otimes_\Cbb\mc O_\MB\quad\text{via }\mc U(\eta_\blt)
\end{align*}

Now $\partial_q$ is an element of $H^0(\mc D_r,\Theta_{\mc D_r})$. By constant extension, $\partial_q$ becomes an element of $H^0(\MB,\Theta_{\MB})$. Suppose that 
\begin{align}\label{eqb95}
\text{$\partial_q$ has a lift }\xk\in H^0(\MC,\Theta_\MC(\blt\SX))
\end{align}
(Recall Rem. \ref{lbb12} that the existence of $\xk$ is automatic when $\MB$ is Stein.) By \eqref{eqb31}, we can find $h_i\in \mc O((\eta_i,q,\tau_\blt)(U_i-\SX))$ such that
\begin{align}\label{eqb63}
\xk|_{U_i}=h_i(\eta_i,q,\tau_\blt)\partial_{\eta_i}+\partial_q
\end{align}
So $h_i$ has finite poles at $\eta_i=0$. The differential operator $\nabla_{\partial_q}$ on $\scr W_\fx(\Wbb)$ is defined by the lift $\xk$ as in Def. \ref{lbb16}. Write
\begin{align*}
h_i(\eta_i,q,\tau_\blt)=\sum_{k\in\Zbb}h_{i,k}(q,\tau_\blt)\eta_i^k\qquad\text{where }h_{i,k}\in \mc O((q,\tau_\blt)(\MB))
\end{align*}
noting that $h_{i,k}=0$ for sufficiently negative $k$. Then for each $w\in\Wbb\otimes_\Cbb\mc O(\MB)$ we have
\begin{subequations}
\begin{equation}\label{uniformization33}
    \nabla_{\partial_q}w=\partial_q w-A(q,\tau_\blt)w\qquad\in\Wbb\otimes_\Cbb\mc O(\MB)
\end{equation}
where
\begin{align}\label{eqb56}
    A(q,\tau_\blt)w=\sum_{i=1}^N \sum_{k\in \Zbb} h_{i,k}(q,\tau_\blt)L_i(k-1)w \qquad\in\Wbb\otimes_\Cbb\mc O(\MB)
\end{align}
\end{subequations}
Note that RHS of \eqref{eqb56} is a finite sum. We write
\begin{align*}
 A(q,\tau_\blt)w=\sum_{n\in \Nbb}A_n(\tau_\blt)w\cdot q^n
\end{align*}
where $A_n(\tau_\blt)w\in\Wbb\otimes_\Cbb\mc O(\MB_0)$. 

\begin{comment}
In other words,
\begin{align}\label{eqb55}
A_n(\tau_\blt)w=\sum_{i=1}^N \sum_{k\in \Zbb}\Res_{q=0} h_{i,k}(q,\tau_\blt)L_i(k-1)w\cdot q^{-n-1}dq
\end{align}
\end{comment}

\begin{rem}\label{lbb27}
Recall from \cite[Prop. 2.21]{GZ1} that we have an explicit description for conformal blocks associated to $\fx$: An element $\upphi\in H^0(\MB,\scr T^*_\fx(\Wbb))$ is precisely an $\mc O_\MB$-module morphism $\upphi:\Wbb\otimes_\Cbb\mc O_\MB\rightarrow\mc O_\MB$ such that $\upphi|_b\in\scr T_{\fx_b}^*(\Wbb)$ for each $b\in\MB$. 

Note that a morphism $\uppsi:\Wbb\otimes_\Cbb\mc O_\MB\rightarrow\mc O_\MB$ (i.e., an element $\uppsi\in H^0(\MB,(\Wbb\otimes_\Cbb\mc O_\MB)^*)$) is determined by its values $\uppsi(w)$ at the constant sections $w\in\Wbb$. Thus, a morphism $\Wbb\otimes_\Cbb\mc O_\MB\rightarrow\mc O_\MB$ is equivalent to a linear map $\Wbb\rightarrow\mc O(\MB)$, and is also equivalent to an $\mc O(\MB)$-module morphism $\Wbb\otimes_\Cbb\mc O(\MB)\rightarrow \mc O(\MB)$. Therefore:
\begin{itemize}
\item An element $\upphi\in H^0(\MB,\scr T^*_\fx(\Wbb))$ is equivalently a linear map $\Wbb\rightarrow\mc O(\MB)$ such that $\upphi(\cdot)|_b\in\scr T_{\fx_b}^*(\Wbb)$ for each $b\in\MB$.
\end{itemize}
A similar description holds for conformal blocks associated to $\fx_0$. \hfill\qedsymbol
\end{rem}

We shall always view $\Wbb$ as a subspace of $\Wbb\otimes_\Cbb\mc O(\MB)$ by identifying $w$ with $w\otimes 1$.

\begin{df}\label{lbb31}
    Fix an element $\upphi_0\in H^0(\MB_0,\ST_{\fx_0}^*(\Wbb))$, viewed as a linear map $\Wbb\rightarrow\mc O(\MB_0)$. The \textbf{formal parallel transport} of $\upphi_0$ (with respect to the lift $\xk$, cf. \eqref{eqb95}) is the linear map 
    \begin{align*}
        \upphi:\Wbb\rightarrow \mc O(\MB_0)[[q]]
    \end{align*}
extended $\mc O(\MB_0)[[q]]$-linearly to 
   \begin{align}\label{eqb57}
        \upphi:\Wbb\otimes_\Cbb\mc O(\MB_0)[[q]]\rightarrow \mc O(\MB_0)[[q]]
    \end{align}
such that for each $w\in\Wbb$ we have 
\begin{align}\label{eqb53}
\partial_q \upphi(w)=-\upphi(A(q,\tau_\blt)w)\qquad \upphi(w)\big|_{q=0}=\upphi_0(w)
\end{align}
where $A(q,\tau_\blt)w=\eqref{eqb56}$.   
\end{df}

\begin{rem}\label{lbb24}
By Def. \ref{lbb18}, the first relation of \eqref{eqb53} simply says for each $w\in\Wbb$ that 
\begin{align}\label{eqb59}
\partial_q\upphi(w)=\upphi(\nabla_{\partial_q}w)
\end{align}
By Leibniz's rule, \eqref{eqb59} also holds when $w\in \Wbb\otimes\mc O(\MB)\subset\Wbb\otimes\mc O(\MB_0)[[q]]$. So
\begin{align*}
(\nabla_{\partial_q}\upphi)(w)=0
\end{align*}
for all $w\in \Wbb\otimes\mc O(\MB)$, justifying the name ``formal parallel transport".
\end{rem}

\begin{rem}\label{lbb23}
Formal parallel transports always exist and are unique. To see this, we write any linear map $\upphi:\Wbb\rightarrow\mc O(\MB_0)[[q]]$ as
\begin{align*}
\upphi(w)=\sum_{n\in\Nbb}\wht\upphi_n(w)q^n
\end{align*}
where $\wht\upphi_n:\Wbb\rightarrow\mc O(\MB_0)$ is linear, extended $\mc O(\MB_0)$-linearly to a linear map
\begin{align*}
\wht\upphi_n:\Wbb\otimes_\Cbb\mc O(\MB_0)\rightarrow\mc O(\MB_0)
\end{align*}
Then condition \eqref{eqb53} is equivalent to the following relations in $\mc O(\MB_0)$:
\begin{align}\label{eqb54}
 n\wht\upphi_n(w) =-\sum_{l=0}^{n-1} \wht\upphi_l\big( A_{n-l-1}(\tau_\blt)w\big)\qquad(\text{if }n\geq 1),\qquad\qquad \wht\upphi_0(w)=\upphi_0(w)
\end{align}
The existence and uniqueness of $(\wht\upphi_n)_{n\in\Nbb}$ satisfying \eqref{eqb54} is clear. 
\end{rem}

\begin{eg}\label{lbb32}
In Rem. \ref{lbb23}, if each $h_i(\eta_i,q,\tau_\blt)$ is independent of $q$ (and hence can be written as $h_i(\eta_i,\tau_\blt)=\sum_{k\in\Zbb}h_{i,k}(\tau_\blt)$), then \eqref{eqb53} can be solved by
\begin{align}\label{eqb64}
\upphi(w)=\upphi_0\big(e^{-qA(\tau_\blt)}w\big)\qquad\text{ where } A(\tau_\blt)=\sum_{i=1}^N \sum_{k\in\Zbb}h_{i,k}(\tau_\blt)L_i(k-1)
\end{align} 
\end{eg}

\begin{lm}\label{lbb25}
Let $\upphi_0\in H^0(\MB_0,\scr T_{\fx_0}^*(\Wbb))$. Then the formal parallel transport    $\upphi$ is a \textbf{formal conformal block}  associated to $\fx$ and $\Wbb$, namely, viewed as a map \eqref{eqb57}, $\upphi$ vanishes on $\SJ_\fx^\pre(\MB)=\eqref{eqb58}$.
\end{lm}

\begin{proof}
Choose $w\in \SJ_\fx^\pre(\MB)\subset \Wbb\otimes\mc O(\MB)$. By Rem. \ref{lbb24} we have
    \begin{align*}
        \upphi(w)=\sum_{n\in \Nbb}\frac{q^n}{n!}\partial_q^n \upphi(w)\big|_{q=0}=\sum_{n\in \Nbb}\frac{q^n}{n!}\upphi(\nabla_{\partial_q}^n w)\big|_{q=0} 
    \end{align*}
By the proof of Thm. \ref{independent1} (more precisely, by \eqref{eqb34}), $\nabla_{\partial_q}^n w\in \SJ_\fx^\pre(\MB)$ for all $n\in \Nbb$. Thus $\nabla^n_{\partial_q}w\big|_{q=0}\in\SJ_{\fx_0}^\pre(\MB_0)$. Since $\upphi_0$ is a conformal block for $\fx_0$ and hence vanishes on $\SJ_{\fx_0}^\pre(\MB_0)$, we have $\upphi(\nabla_{\partial_q}^n w)\big|_{q=0}=\upphi_0\big(\nabla_{\partial_q}^n w\big|_{q=0}\big)=0$. So $\upphi(w)=0$.
\end{proof}

\begin{thm}\label{lbb26}
Let $\upphi_0\in H^0(\MB_0,\scr T_{\fx_0}^*(\Wbb))$ with formal parallel transport $\upphi$. Then for each $w\in\Wbb$, the formal power series $\upphi(w)\in\mc O(\MB_0)[[q]]$ converges a.l.u. on $\mc D_r\times\MB_0=\MB$. This defines a linear map $\upphi:\Wbb\rightarrow\mc O(\MB)$. Moreover, $\upphi$ belongs to $H^0(\MB,\scr T_\fx^*(\Wbb))$.
\end{thm}

\begin{proof}
Since the properties to be proved can be checked locally (with respect to $\MB_0$), by shrinking $\MB_0$, we assume that $\MB_0$ is Stein and connected. So $\MB$ is also Stein and connected. Therefore, by Thm. \ref{lbb21}, we can find finitely many elements $s_1,s_2,\cdots\in \Wbb\otimes \MO(\MB)$ generating $\Wbb\otimes \MO(\MB)$ mod $\SJ_\fx^\pre(\MB)$. Fix $\mu_\blt\in \Cbb^N$ such that $s_1,s_2,\cdots\in \Wbb_{[\leq\mu_\blt]}\otimes \MO(\MB)$. Since each $s_i$ is an $\mc O(\MB)$-linear combination of elements of $\Wbb_{[\leq\mu_\blt]}$, one has that $\Wbb_{[\leq\mu_\blt]}$ generates the $\MO(\MB)$-module $\Wbb\otimes\mc O(\MB)$ mod $\SJ_\fx^\pre(\MB)$.

Choose any $\lambda_\blt\geq \mu_\blt$. Since $\Wbb$ is grading-restricted and $\Vbb^{\otimes N}$ is $C_2$-cofinite,   $\Wbb_{[\leq\lambda_\blt]}$ is finite dimensional. Let $(e_j)_{j\in J}$ be a (finite) basis of $\Wbb_{[\leq\lambda_\blt]}$. By Rem. \ref{lbb24}, we have $\partial_q \upphi(e_i)=\upphi(\nabla_{\partial_q}e_i)$. Since $\Wbb_{[\leq\lambda_\blt]}$ generates $\Wbb\otimes\mc O(\MB)$ mod $\SJ_\fx^\pre(\MB)$, we can find $\Omega_{i,j}\in \MO(\MB)$ such that
    \begin{align}
        \nabla_{\partial_q}e_i=\sum_{j\in J}\Omega_{i,j}e_{j}\mod \SJ_\fx^\pre(\MB)
    \end{align}
for all $i,j\in J$.    Thus, by Lem. \ref{lbb25},
    \begin{align}\label{uniformization28}
        \partial_q \upphi(e_i)=\upphi(\nabla_{\partial_q}e_i)=\sum_{j\in J}\Omega_{i,j}\upphi(e_{j}).
    \end{align}
By \eqref{uniformization28}, as an element of $\mc O(\MB_0)[[q]]^J$, $f:=\oplus_{j\in J}\upphi(e_j)$  is a formal solution of the differential equation $\partial_q f=\Omega f$ where $\Omega$ is a $\Cbb^{J\times J}$-valued holomorphic function on $\MB$. Therefore, by \cite[Thm. A.1]{Gui-sewingconvergence}, $\upphi(w)$ converges a.l.u. on $\MB$ whenever $w=e_i$, and hence whenever $w\in\Wbb_{[\leq\lambda_\blt]}$. Since $\lambda_\blt\geq \mu_\blt$ is arbitrary, $\upphi(w)$ converges a.l.u. for all $w\in\Wbb$.

By $\mc O_\MB$-linear extension, we can view $\upphi$ as an $\mc O_\MB$-module morphism $\upphi:\Wbb\otimes\mc O_\MB\rightarrow\mc O_\MB$ (cf. Rem. \ref{lbb27}). By Lem. \ref{lbb25}, $\upphi$ vanishes on $\SJ_\fx^\pre(\MB)$. Therefore, by Thm. \ref{lbb28}, $\upphi$ is a conformal block associated to $\fx$ and $\Wbb$.
\end{proof}

\subsection{Virasoro uniformization by non-autonomous flows}\label{lbb67}

In this section, we explain how Thm. \ref{lbb26} can be interpreted from the perspective of Virasoro uniformization. The material of this section will not be used elsewhere in the series of papers. However, it might have potential applications in the future. It also provides a differential geometric background for the construction of $\nabla$ on $\scr T_{\fx}^*(\Wbb)$ comparable to the algebraic geometric background in \cite{FB04} Sec. 17.3 (especially Rem. 17.3.3).

An \textbf{open annulus} denotes a set of the form $\{z\in\Cbb:r_1<|z|<r_2\}$ where $0<r_1<r_2$. For each $r>0$ we let $r\Sbb^1=\{z\in\Cbb:|z|=r\}$. Consider a family of $N$-pointed compact Riemann surfaces with local coordinates
\begin{align*}
\fx_0=(\pi_0:\MC_0\rightarrow\MB_0|\sigma_1,\dots,\sigma_N;\mu_1,\dots,\mu_N )
\end{align*}
where $\MB_0$ is a connected open subset of $\Cbb^m$ with standard coordinates $\tau_\blt=(\tau_1,\dots,\tau_m)$. We abbreviate $\tau_j\circ\pi_0$ to $\tau_j$ as usual. Let $U_{1,0},\dots,U_{N,0}$ be mutually disjoint neighborhoods of $\sigma_1(\MB_0)\dots,\sigma_N(\MB_0)$ such that $\mu_i\in\mc O(U_{i,0})$, and that
\begin{align*}
(\mu_i,\tau_\blt)(U_{i,0})=\mc D_{p_i}\times\MB_0
\end{align*}
for some $p_1,\dots,p_N>0$.

Fix $r>0$ and
\begin{align*}
\MB=\mc D_r\times \MB_0
\end{align*}
For each $i$, choose a Laurent series in $\mc O(\MB)((z))$: 
\begin{align*}
h_i(z,q,\tau_\blt)=\sum_{k\in\Zbb}h_{i,k}(q,\tau_\blt)z^k\qquad\text{converging a.l.u. on }\mc D_{p_i}^\times\times\MB
\end{align*}
So it is the expansion of an element of $\mc O(\mc D_{p_i}^\times\times\MB)$ with finite poles at $\{0\}\times\MB$. Fix
\begin{align*}
0<r_i<p_i
\end{align*}
By the basics of differential equations, after making $r$ smaller, for each $i$ we can find
\begin{gather*}
\beta^i\in\mc O(\mc A^i\times\MB)\qquad\text{where $\mc A^i$ is an open annulus containing $r_i\Sbb^1$}
\end{gather*}
satisfying the following properties. (We write $\beta^i(z,q,\tau_\blt)$ as $\beta^i_{q,\tau_\blt}(z)$.)
\begin{enumerate}[label=(\arabic*)]
\item For each $(q,\tau_\blt)\in\MB$ we have
\begin{align}\label{eqb60}
\beta^i_{q,\tau_\blt}(\mc A^i)\subset \mc D_{p_i}^\times
\end{align}
Moreover, for each $(z,q,\tau_\blt)\in\mc A^i\times\MB$, we have
\begin{align}\label{eqb62}
\partial_q\beta^i_{q,\tau_\blt}(z)=h\big(\beta^i_{q,\tau_\blt}(z),q,\tau_\blt\big)\qquad \beta^i_{0,\tau_\blt}(z)=z
\end{align}
(Thus, when $q=0$, \eqref{eqb60} reads $\mc A^i\subset \mc D_{p_i}^\times$.)
\item For each $(q,\tau_\blt)\in\MB$, we have $0\notin\beta^i_{q,\tau_\blt}(r_i\Sbb^1)$.
\end{enumerate} 
Note that \eqref{eqb62} simply says that when $\tau_\blt$ is fixed, $q\in\mc D_r\mapsto \beta^i_{q,\tau_\blt}$ is the non-autonomous flow generated by the time-dependent (i.e. $q$-dependent) vector field $h_i\partial_z$. (More precisely, if $q\mapsto \alpha^i_{q,\tau_\blt}$ is the (autonomous) flow in $\Cbb\times\Cbb$ generated by $h_i\partial_z+\partial_q$, then $\alpha^i_{q,\tau_\blt}(z,0)=(\beta^i_{q,\tau_\blt}(z),q)$.) Therefore,  the basic properties of autonomous flows implies:
\begin{enumerate}
\item[(3)] For each $(q,\tau_\blt)\in\mc B$, the map $\beta^i_{q,\tau_\blt}:\mc A^i\rightarrow\Cbb$ is injective (and hence is a biholomorphism onto an open subset of $\Cbb$).
\end{enumerate}

For each $(q,\tau_\blt)$, let $\Gamma^i_{q,\tau_\blt}$ denote the oriented simple closed curve $\beta^i_{q,\tau_\blt}:r_i\Sbb^1\rightarrow\Cbb$. By the Jordan curve theorem, $\Pbb^1-\Gamma^i_{q,\tau_\blt}$ has exactly two components
\begin{align*}
\Pbb^1-\Gamma^i_{q,\tau_\blt}=\Omega^i_{q,\tau_\blt}\sqcup\wtd\Omega^i_{q,\tau_\blt}
\end{align*}
where $\infty\in\wtd\Omega^i_{q,\tau_\blt}$. Define
\begin{align}
O^i=\{(z,q,\tau_\blt)\in\Pbb^1\times\MB:z\in\Omega^i_{q,\tau_\blt}\}\qquad \wtd O^i=\{(z,q,\tau_\blt)\in\Pbb^1\times\MB:z\in\wtd\Omega^i_{q,\tau_\blt}\}
\end{align}

\begin{rem}\label{lbb29}
$O^i$ and $\wtd O^i$ are open subsets of $\Pbb^1\times\MB$. Moreover,  for each $(q,\tau_\blt)\in\MB$ we have $0\in\Omega_{q,\tau_\blt}^i$ (i.e. $0$ is inside $\Gamma^i_{q,\tau_\blt}$).
\end{rem}

\begin{proof}
The continuity of $\beta^i$ implies that $O^i\cup \wtd O^i$ is open. The function
\begin{align*}
\mc W:\quad (z,q,\tau_\blt)\in O^i\cup \wtd O^i\quad\mapsto\quad \frac 1{2\im\pi}\oint\nolimits_{\Gamma^i_{q,\tau_\blt}}\frac {d\zeta}{\zeta-z}
\end{align*}
is continuous and takes values in $\Nbb$. So it must be locally constant. For each $(q,\tau_\blt)\in\MB$, $\mc W(\cdot,q,\tau_\blt)$ is constant on $\Omega^i_{q,\tau_\blt}$ and on $\wtd\Omega^i_{q,\tau_\blt}$. Since $\mc W(\infty,q,\tau_\blt)=0$, we have $\mc W(\cdot,q,\tau_\blt)|_{\wtd\Omega^i_{q,\tau_\blt}}=0$. Since $0\notin\Gamma^i_{q,\tau_\blt}$ for all $(q,\tau_\blt)\in\MB$, and since $\mc W(0,0,\tau_\blt)=1$, we have $\mc W(0,q,\tau_\blt)=1$ for all $q$. So $\mc W(\cdot,q,\tau_\blt)|_{\Omega^i_{q,\tau_\blt}}=1$. It follows that $O^i=\mc W^{-1}(1)$ and $\wtd O^i=\mc W^{-1}(0)$. Thus $O^i$ and $\wtd O^i$ are open. Since $\mc W(0,q,\tau_\blt)=1$, we have $0\in\Omega_{q,\tau_\blt}^i$.
\end{proof}

\begin{rem}
According to Rem. \ref{lbb29}, for each $z\in\mc A^i-r_i\Sbb^1$, the subset of all $(q,\tau_\blt)\in\MB$ such that $\beta^i_{q,\tau_\blt}(z)$ is inside (resp. outside) $\Gamma^i_{q,\tau_\blt}$ is an open subset of $\MB$, and hence is also closed, and must be either $\emptyset$ or $\MB$. Therefore, for each $z\in\mc A^i$ we have
\begin{gather}\label{eqb61}
\begin{gathered}
|z|<r_i\qquad\Longleftrightarrow\qquad \beta^i_{q,\tau_\blt}(z)\text{ is inside $\Gamma^i_{q,\tau_\blt}$ for all }(q,\tau_\blt)\in\MB\\
|z|>r_i\qquad\Longleftrightarrow\qquad \beta^i_{q,\tau_\blt}(z)\text{ is outside $\Gamma^i_{q,\tau_\blt}$ for all }(q,\tau_\blt)\in\MB
\end{gathered}
\end{gather}
\end{rem}

\begin{df}
We now construct a family $\fx=(\pi:\MC\rightarrow\MB|\sgm_1,\dots,\sgm_N;\eta_1,\dots,\eta_N)$, called the \textbf{Virasoro uniformization of $\fx_0$} by the non-autonomous flows $\beta^1,\dots,\beta^N$ (or by $h_1,\dots,h_N$). $\MB$ has already been defined.  For each $(q,\tau_\blt)\in\MB$, let
\begin{align*}
\mc R^i_{q,\tau_\blt}=\beta^i_{q,\tau_\blt}(\mc A^i)\cup\Omega^i_{q,\tau_\blt}
\end{align*}
Then the fiber $\MC_{q,\tau_\blt}$ is defined by (setting $\eps_i=\inf\{|z|:z\in\mc A^i\}$)
\begin{gather*}
\mc C^+_{q,\tau_\blt}=\pi_0^{-1}(\tau_\blt)\Big\backslash \bigcup_{i=1}^N\mu_i^{-1}(\ovl \MD_{\eps_i})\qquad \mc C^-_{q,\tau_\blt}=\bigsqcup_{i=1}^N\mc R^i_{q,\tau_\blt}\\ \MC_{q,\tau_\blt}=(\MC^+_{q,\tau_\blt}\sqcup \MC^-_{q,\tau_\blt})\big/\sim
\end{gather*}
where the gluing $\sim$ is given by the following biholomorphism between open subsets of $\MC^+_{q,\tau_\blt}$ and $\MC^-_{q,\tau_\blt}$:
\begin{align*}
\beta^i_{q,\tau_\blt}\circ\mu_i:\pi^{-1}_0(\tau_\blt)\cap\mu_i^{-1}(\mc A^i)\xlongrightarrow{\simeq}\beta^i_{q,\tau_\blt}(\mc A^i)
\end{align*}
(One needs \eqref{eqb61} to show that $\MC_{q,\tau_\blt}$ is (sequentially) compact Hausdorff.) Assembling these fibers together, one gets $\MC$. The map $\pi:\MC\rightarrow\MB$ is defined by sending each $\MC_{q,\tau_\blt}$ to $(q,\tau_\blt)$.

Recall that $0$ belongs to $\Omega_{q,\tau_\blt}^i$ (Rem. \ref{lbb29}). Then $\sgm_i$ is defined such that $\sgm_i(q,\tau_\blt)$ is $0\in\mc R^i_{q,\tau_\blt}$. Let
\begin{align*}
U_i=\bigcup_{(q,\tau_\blt)\in\MB}\mc R^i_{q,\tau_\blt}
\end{align*}
which is open in $\MC$ and contains $\sgm_i(\MB)$. The local coordinate $\eta_i\in\mc O(U_i)$ at $\sgm_i(\MB)$ is defined such that its restriction to $\mc R^i_{q,\tau_\blt}$ is the standard coordinate $z\mapsto z$. (Therefore, its restriction to $\pi^{-1}_0(\tau_\blt)\cap\mu_i^{-1}(\mc A^i)\subset\MC_{q,\tau_\blt}^+$ is $\eta_i=\beta^i_{q,\tau_\blt}\circ\mu_i$.)  \hfill\qedsymbol
\end{df}

\begin{rem}
Now one can define a lift $\xk\in H^0(\MC,\Theta_\MC(\blt\SX))$ of $\partial_q$ as follows. On the open subset $\dps\MC^+:=\bigcup_{(q,\tau_\blt)\in\MB}\MC^+_{q,\tau_\blt}$ of $\MC$, the vector field $\xk$ is $\partial_q$. In particular, on the open subset $\dps\bigcup_{(q,\tau_\blt)}\big(\pi^{-1}_0(\tau_\blt)\cap\mu_i^{-1}(\mc A^i)\big)$ of $\MC^+$ and under the coordinate system $(\mu_i,q,\tau_\blt)$, the vector field $\xk$ is $\partial_q$. Using \eqref{eqb62} and the change of variable formula for tangent vectors, one easily shows that under the coordinate system $(\eta_i=\beta^i_{q,\tau_\blt}\circ\mu_i,q,\tau_\blt)$ one has $\xk=h_i(\eta_i,q,\tau_\blt)\partial_{\eta_i}+\partial_q$ on this open set. Thus, the construction of $\xk$ is finished if we define $\xk$ on $\dps\bigcup_{(q,\tau_\blt)}\mc R^i_{q,\tau_\blt}$ (under the standard coordinates $(z,q,\tau_\blt)$) to be $h_i(z,q,\tau_\blt)\partial_z+\partial_q$.
\end{rem}

The above construction shows that $\xk$ satisfies \eqref{eqb63} in Sec. \ref{lbb30}, i.e. $\xk|_{U_i}=h_i(\eta_i,q,\tau_\blt)\partial_{\eta_i}+\partial_q$. Therefore, we have shown that the non-autonomous Virasoro uniformization gives rise to a family $\fk X$ and a lift $\xk$ of $\partial_q$ as in Sec. \ref{lbb30}. Identify $\MB_0$ with $\{0\}\times\MB_0$ in $\mc D_r\times\MB_0=\MB$. Then the family $\fx_0$ in Sec. \ref{lbb30} is the restriction of $\fx$ to $\MB_0$. Choose $\upphi_0\in H^0(\MB_0,\scr T_{\fx_0}^*(\Wbb))$. Let $\upphi$ be its formal parallel transport (Def. \ref{lbb31}), called the \textbf{Virasoro uniformization} of $\upphi_0$. Then Thm. \ref{lbb26} says that \textit{the Virasoro uniformization of $\upphi_0$ converges a.l.u. to a conformal block associated to $\fx$ and $\Wbb$}. 

\begin{rem}
Intuitively, the fiber $\fx_{q,\tau_\blt}$ is obtained from $\fx_{0,\tau_\blt}$ by changing each of the local coordinates $\mu_i$ to the parametrization $\beta^i_{q,\tau_\blt}\circ\mu_i$. Therefore, the formula \eqref{eqb53} (equivalently, \eqref{eqb54}) for $\upphi$ can be viewed as the \textbf{change of parametrization formula for conformal blocks}. In particular, if each $\beta^i$ is autonomous (i.e., $h_i$ is independent of $q$), Exp. \ref{lbb32} shows that $\upphi$ can be expressed as the exponential of a sum of Virasoro operators. See \cite{GuiLec} (especially Sec. 13) for more discussions on this topic.
\end{rem}

\subsection{Local freeness for smooth families}\label{lbb69}

We emphasize that in this section, unless otherwise stated, $\fx$ is not assumed to have local coordinates $\eta_\blt$ at $\sgm_\blt(\MB)$.

\begin{thm}\label{lbb35}
The $\MO_\MB$-module $\ST_\fx(\Wbb)$ is a locally free. Hence $\scr T_\fx^*(\Wbb)$ is locally free. 

Moreover, suppose that $\fx$ is equipped with local coordinates $\eta_\blt$ at $\sgm_\blt(\MB)$ so that we can make the identifications
\begin{gather}\label{eqb98}
\begin{gathered}
\SW_\fx(\Wbb)=\Wbb\otimes_\Cbb\MO_\MB\qquad\text{via }\mc U(\eta_\blt)\\
\SW_{\fx_b}(\Wbb)=\Wbb\qquad\text{via }\MU(\eta_\blt|_{\MC_b})
\end{gathered}
\end{gather}
for all $b\in\MB$. If we consider the stalk map
\begin{align}\label{eqb96}
\scr T_{\fx}^*(\Wbb)_b\rightarrow\scr T_{\fx_b}^*(\Wbb)
\end{align}
defined by the restriction maps $H^0(V,\scr T_\fx^*(\Wbb))\rightarrow\scr T_{\fx_b}^*(\Wbb)$ (for all open $V\ni b$) sending each $\upphi:\Wbb\rightarrow\mc O(V)$ to $\upphi(\cdot)|_b$, then \eqref{eqb96} descends to a linear isomorphism
\begin{align}\label{eqb97}
\scr T_\fx^*(\Wbb)\big|_b\xlongrightarrow{\simeq}\scr T_{\fx_b}^*(\Wbb)
\end{align}
Consequently, the function $b\in\MB\mapsto \dim \scr T_{\fx_b}^*(\Wbb)$ is locally constant. 
\end{thm}

Note that the second paragraph describes how the space of conformal blocks $\scr T_{\fx_b}^*(\Wbb)$ can be viewed canonically as the fiber of the vector bundle $\scr T_\fx^*(\Wbb)$ at $b$.

\begin{proof}
Step 1. It suffices to assume that $\fx$ has local coordinates and assume the identifications \eqref{eqb98}. Let us prove that $\ST_\fx(\Wbb)$ is locally free. By Lem. \ref{lbb34}, it suffices to prove that the rank function $\Rbf:b\in\MB\mapsto\dim\scr T_{\fx_b}^*(\Wbb)$ is locally constant. Thus we may assume that $\MB$ is the open ball $B_r=\{\tau_\blt\in\Cbb^m:|\tau_\blt|<r\}$ and prove that $\Rbf(a_\blt)=\Rbf(0)$ for each $a_\blt=(a_1,\dots,a_m)\in B_r$. Since $\zeta\in\mc D_{r/|a_\blt|}\mapsto \zeta\cdot a_\blt$ parametrizes a closed $1$-dimensional submanifold of $B_r$ containing $a_\blt$, it suffices to prove that $\Rbf$ is constant on this Riemann surface. Therefore, it suffices to prove that $\Rbf$ is locally constant under the assumption that $\dim\MB=1$.

We want to prove that each point of $\MB$ has a neighborhood on which $\Rbf$ is constant. Thus we may assume that $\fx$ has local coordinates $\eta_\blt$ at $\sgm_\blt(\MB)$, that $\MB=\mc D_{2r}$, and prove that $\Rbf(p)=\Rbf(0)$ for each $p\in\MD_r$. Let $q$ be the standard coordinate of $\MB$. Then by Rem. \ref{lbb12}, $\partial_q$ has a lift $\xk$.  Fix $p\in\MD_r$. By Thm. \ref{lbb26}, for each $\uppsi\in \scr T_{\fx_0}^*(\Wbb)$, we have $F\uppsi\in H^0(\mc D_r,\scr T_\fx^*(\Wbb))$ satisfying the differential equation \eqref{eqb53} (defined by $\xk$) with initial condition $F\uppsi|_{q=0}=\uppsi$. Similarly, let $\mc D_r(p)=\{z\in\Cbb:|z-p|<r\}$, then for each $\upomega\in\scr T_{\fx_p}^*(\Wbb)$ we have $G\upomega\in H^0(\mc D_r(p),\scr T_\fx^*(\Wbb))$ satisfying the same differential equation \eqref{eqb53} but with initial condition $G\upomega|_{q=p}=\upomega$. 

Define linear maps
\begin{align*}
F_p:\scr T_{\fx_0}^*(\Wbb)\rightarrow\scr T_{\fx_p}^*(\Wbb)\qquad\text{ resp.}\qquad G_0:\scr T_{\fx_p}^*(\Wbb)\rightarrow\scr T_{\fx_0}^*(\Wbb)
\end{align*}
by $F_q\uppsi=F\uppsi|_p$ resp. $G_0\upomega=G\upomega|_0$. Then for each $\upomega\in\scr T_{\fx_p}^*(\Wbb)$, clearly $F(G_0\upomega)\in H^0(\mc D_r,\scr T_\fx^*(\Wbb))$ and $G\upomega\in H^0(\mc D_r(p),\scr T_\fx^*(\Wbb))$ satisfy the same differential equation and the same initial condition at $q=0$. So the series expansions of $FG_0\upomega$ and $G\upomega$ at $0$ both satisfy \eqref{eqb54}, and hence are equal. So $FG_0\upomega$ equals $G\upomega$ on $\mc D_r\cap\mc D_r(p)$, and hence at $p$. This shows $F_pG_0\upomega=G\upomega|_p=\upomega$. Thus $F_pG_0=\id$. A similar analysis for $GF_p\uppsi$ and $F\uppsi$ shows that $G_0F_p=\id$. So $F_p$ is a linear isomorphism, and hence $\Rbf(p)=\Rbf(0)$. \\[-1ex]

Step 2. The map \eqref{eqb96} clearly sends $\mk_{\MB,b}\ST^*_\fx(\Wbb)_b$ to zero, and hence descends to \eqref{eqb97}.  Let us prove that \eqref{eqb97} is an isomorphism. Since $\ST_\fx(\Wbb)$ and $\ST^*_\fx(\Wbb)$ are locally free, we have $\dim\ST_\fx^*(\Wbb)|_b=\dim\ST_\fx(\Wbb)|_b$. By \cite[Prop. 2.19]{GZ1}, we have $\dim\ST_\fx(\Wbb)|_b=\dim\ST_{\fx_b}(\Wbb)$. Clearly $\ST_{\fx_b}(\Wbb)$ and its dual space $\ST^*_{\fx_b}(\Wbb)$ have the same dimension. Therefore, the domain of the codomain of \eqref{eqb97} have the same dimension.

Let us prove that \eqref{eqb96} is surjective. Then \eqref{eqb97} is also surjective, and hence is a linear isomorphism. Without loss of generality, we assume that $\MB=\MD_{\eps_1}\times\cdots\times\MD_{\eps_m}$ and $b=0=(0,\dots,0)$. Choose any $\uppsi_0\in\ST^*_{\fx_0}(\Wbb)$. Let $0\leq j\leq m$ and $\MB_j=\MD_{\eps_1}\times\cdots\times\MD_{\eps_j}$. (We understand $\MB_0$ as $\{0\}$.) Suppose that $\uppsi_j:\Wbb\rightarrow\MO(\MB_j)$ is a conformal block (i.e., an element of $H^0(\MB_j,\ST^*_{\fx_{\MB_j}}(\Wbb))$) such that $\uppsi_j|_0=\uppsi_0$. Then, by Thm. \ref{lbb26}, $\uppsi_j$ can be extended to a conformal block $\uppsi_{j+1}:\Wbb\rightarrow\MO(\MB_{j+1})$. It follows that we have $\uppsi\in H^0(\MB,\ST_\fx^*(\Wbb))$ such that $\uppsi|_0=\uppsi_0$. This proves the surjectivity of \eqref{eqb96}.
\end{proof}

\begin{co}
Let $(C;x_1,\dots,x_N)$ be an $N$-pointed compact Riemann surface. Associate $\Wbb$ to $x_\blt$. Then the finite number $\dim\scr T_{(C;x_\blt)}^*(\Wbb)$ depends only on the \textbf{topology of $(C;x_\blt)$}, namely,  the number of components of $C$, the topology of each component of $C$, and the subset of all $i\in\{1,\dots,N\}$ such that $x_i$ is contained in a given component.
\end{co}

Thus, $\dim\scr T_\fx^*(\Wbb)$ is independent of the complex structure of $C$, the position of each $x_i$ (as long as it is always on the same component), and the local coordinates at $x_\blt$.

\begin{proof}
All $N$-pointed compact Riemann surfaces (with certain extra structures) that have the same topology as $(C;x_\blt)$ form a family $\fx$ whose base manifold is a product of Teichm\"uller spaces $\mathcal T_{g,n}$ and hence is connected. Therefore, by Thm. \ref{lbb35}, all fibers of $\fx$ have the same dimension of space of conformal blocks.
\end{proof}

\subsection{Application: $\bbs_\fx(\Wbb)$ for a smooth family $\fx$}

The goal of this section is to establish the existence of dual fusion products of $\Wbb$ along a family of compact Riemann surfaces, generalizing the corresponding result from \cite{GZ1} for a single compact Riemann surface. The results of this section will be needed in the third paper of the series to prove the sewing-factorization for families.

We consider a setting different from Asmp. \ref{lbb1}. Choose a family of $(M,N)$-pointed compact Riemann surfaces with local coordinates (recall Def. 2.8 in \cite{GZ1})
$$
\fx=(\tau_\star;\theta_\star\big| \pi:\MC\rightarrow \MB\big| \varsigma_\blt;\eta_\blt)=(\tau_1,\cdots,\tau_M;\theta_1,\cdots,\theta_M\big| \pi:\MC\rightarrow \MB\big| \varsigma_1,\cdots,\varsigma_N).
$$
where $M,N\geq0$. Assume that for each $b\in\MB$, each connected component of $\MC_b$ intersects one of $\sgm_1(\MB),\dots,\sgm_N(\MB)$. 

Recall that $\Vbb$ is $C_2$-cofinite. Associate the grading restricted $\Vbb^{\otimes N}$-module $\Wbb$ to the ordered incoming marked points $\sgm_\blt(\MB)$. When $\MB$ is a single point, we have proved in \cite[Thm. 3.31]{GZ1} that $\Wbb$ has a dual fusion product along $\fx$ in the following sense:

\begin{df}\label{lbb36}
Assume that $\MB$ is a single point. A \textbf{dual fusion product} denotes a pair $(\bbs_\fx(\Wbb),\gimel)$ where $\bbs_\fx(\Wbb)$ is a grading restricted $\Vbb^{\otimes M}$-module associated to $\tau_\star(\MB)$, and $\gimel\in\scr T_\fx^*(\Wbb\otimes\bbs_\fx(\Wbb))$ satisfying the universal property:
\begin{itemize}
\item For each $\upphi\in\scr T_\fx^*(\Wbb\otimes\Mbb)$ (where $\Mbb$ is a grading restricted $\Vbb^{\otimes M}$-module associated to $\tau_\star(\MB)$) there exists a unique $T\in\Hom_{\Vbb^{\otimes M}}(\Mbb,\bbs_\fx(\Wbb))$ such that $\upphi=\gimel\circ(\idt_\Wbb\otimes T)$
\end{itemize}
We abbreviate $(\bbs_\fx(\Wbb),\gimel)$ to $\bbs_\fx(\Wbb)$ when no confusion arises. The contragredient $\Vbb^{\otimes M}$-module of $\bbs_\fx(\Wbb)$ is denoted by $\boxtimes_\fx(\Wbb)$ and called the \textbf{fusion product} of $\Wbb$ along $\fx$.
\end{df}

Here, $\Wbb\otimes\bbs_\fx(\Wbb)$ and $\Wbb\otimes\Mbb$ are both associated to the ordered marked points $\sgm_1(\MB),\dots,\sgm_N(\MB),\tau_1(\MB),\dots,\tau_M(\MB)$. (Note the change in the order of marked points.)

\begin{eg}
When $M=0$, then $\Vbb^{\otimes M}=\Cbb$. A grading restricted $\Cbb$-module is understood as a finite-dimensional $\Cbb$-vector space. Thus, if $\MB$ is a single point, then
\begin{align*}
(\scr T_\fx^*(\Wbb),\gimel)\text{ is a dual fusion product of $\Wbb$ along $\fx$}
\end{align*}
where $\gimel\in \scr T_\fx^*(\Wbb\otimes_\Cbb\scr T_\fx^*(\Wbb))$ is defined by
\begin{align*}
\gimel:\Wbb\otimes\scr T_\fx^*(\Wbb)\rightarrow\Cbb\qquad w\otimes\upphi\mapsto\upphi(w)
\end{align*}
Equivalently, noting the linear isomorphism $\scr T_\fx^*(\Wbb\otimes\scr T_\fx^*(\Wbb))\simeq \scr T_\fx^*(\Wbb)\otimes_\Cbb \scr T_\fx(\Wbb)$, then $\gimel=\sum_i\upphi_i\otimes\wch\upphi^i$ where $(\upphi_i)$ is a (finite) basis of $\scr T_\fx^*(\Wbb)$ with dual basis $(\wch\upphi^i)$ in $\scr T_\fx(\Wbb)$.
\end{eg}

\begin{rem}\label{lbb55}
Assume that $\MB$ is a single point, and let $(\bbs_\fx(\Wbb),\gimel)$ be a dual fusion product. Then $\gimel:\Wbb\otimes\bbs_\fx(\Wbb)\rightarrow\Cbb$ is \textbf{partially injective} in the sense that
\begin{align}\label{eqb67}
\{\xi\in\bbs_\fx(\Wbb):\gimel(w\otimes\xi)=0\text{ for all }w\in\Wbb\}
\end{align}
is zero. This follows from \cite[Ch. 3]{GZ1}, where $\bbs_\fx(\Wbb)$ is constructed as a linear subspace of $\Wbb^*$, and the standard pairing $\Wbb\otimes\Wbb^*\rightarrow\Cbb$ restricts to $\gimel$. There is an equivalent formulation of partial injectivity injectivity. Recall \eqref{eqb66} for the meaning of $P_{\leq\lambda_\blt}$. Note that $\gimel$ gives rise to a linear map
\begin{align*}
\gimel^\sharp:\Wbb\rightarrow\bbs_\fx(\Wbb)^*=\ovl{\boxtimes_\fx(\Wbb)}\qquad w\mapsto \gimel(w\otimes-)
\end{align*}
Then for each $\lambda_\blt\in\Cbb^M$, the map $P_{\leq\lambda_\blt}\circ\gimel^\sharp:\Wbb\rightarrow \boxtimes_\fx(\Wbb)_{[\leq\lambda_\blt]}$ is surjective.
\end{rem}

We now generalize Def. \ref{lbb36} to the case that $\fx$ is a family.

\begin{df}\label{lbb38}
The \textbf{dual fusion product} of $\Wbb$ along $\fx$ is a pair $(\bbs_\fx(\Wbb),\gimel)$ where $\bbs_\fx(\Wbb)$ is a grading restricted $\Vbb^{\otimes M}$-module, $\gimel\in H^0(\MB,\scr T^*_\fx(\Wbb\otimes\bbs_\fx(\Wbb)))$, such that for each $b\in\MB$, the pair $(\bbs_\fx(\Wbb),\gimel|_b)$ is a dual fusion product of $\Wbb$ along $\fx_b$. The contragredient $\Vbb^{\otimes M}$-module of $\bbs_\fx(\Wbb)$ is denoted by $\boxtimes_\fx(\Wbb)$ and called the \textbf{fusion product} of $\Wbb$ along $\fx$.
\end{df}

The following proposition will be a special case of the sewing-factorization theorem for families to be proved in the third paper of this series.

\begin{df}
Let $\Wbb_1,\Wbb_2$ be $\Vbb^{\otimes N}$-modules such that $\Hom_{\Vbb^{\otimes N}}(\Wbb_1,\Wbb_2)<+\infty$ (e.g. when $\Wbb_1,\Wbb_2$ are grading restricted). Then for each complex manifold $X$, the \textbf{holomorphicity} of each map $X\rightarrow \Hom_{\Vbb^{\otimes N}}(\Wbb_1,\Wbb_2)$ can be defined in the obvious way as a vector-valued holomorphic function. 
\end{df}

\begin{pp}\label{lbb39}
Let $(\bbs_\fx(\Wbb),\gimel)$ be a dual fusion product of $\Wbb$ along $\fx$. Then for each grading restricted $\Vbb^{\otimes M}$-module $\Mbb$ and each $\upphi\in H^0(\MB,\scr T_\fx^*(\Wbb\otimes\Mbb))$, there exists a unique holomorphic map $b\in\MB\mapsto T_b\in \Hom_{\Vbb^{\otimes M}}(\Mbb,\bbs_\fx(\Wbb))$ such that $\upphi|_b=\gimel|_b\circ(\idt\otimes T_b)$ for each $b\in\MB$.
\end{pp}

\begin{proof}
The uniqueness is obvious. As for the existence, consider the linear map
\begin{align*}
\Psi:\Hom_{\Vbb^{\otimes M}}(\Mbb,\bbs_\fx(\Wbb))\rightarrow H^0(\MB,\scr T_\fx^*(\Wbb\otimes\Mbb))\qquad S\mapsto \gimel\circ(\idt\otimes S)
\end{align*}
which (by Def. \ref{lbb38}) restricts to a linear isomorphism
\begin{align*}
\Psi|_b:\Hom_{\Vbb^{\otimes M}}(\Mbb,\bbs_\fx(\Wbb))\rightarrow \scr T_{\fx_b}^*(\Wbb\otimes\Mbb)\qquad S\mapsto\gimel|_b\circ(\idt\otimes S)
\end{align*}
for each $b\in\MB$. Therefore, if $S_1,\dots,S_n$ are a basis of $\Hom_{\Vbb^{\otimes M}}(\Mbb,\bbs_\fx(\Wbb))$, then $\Psi|_b(S_1),\dots,\Psi|_b(S_n)$ are a basis of $\scr T_{\fx_b}^*(\Wbb\otimes\Mbb)$. Thus, in view of the isomorphism \eqref{eqb97}, we see that $\Psi(S_1),\dots,\Psi(S_n)$ are a free generator of the $\mc O_\MB$-module $\scr T_\fx^*(\Wbb\otimes\Mbb)$, i.e., the $\MO_\MB$-module morphism $\MO_\MB^n\rightarrow \scr T_\fx^*(\Wbb\otimes\Mbb)$ sending each $(f_1,\dots,f_n)$ to $\sum_j f_j\Psi(S_j)$ is an isomorphism. Thus $\upphi$ can be uniquely written as $\sum_{j=1}^n f_j\Psi(S_j)$ where $f_j\in\mc O(\MB)$. The existence of $T$ follows from setting $T_b=\sum_j f_j(b)S_j$ for each $b$.
\end{proof}

\begin{co}\label{lbb40}
Let $(\bbs_\fx(\Wbb),\gimel)$ and $(\wtd\bbs_\fx(\Wbb),\wtd\gimel)$ be dual fusion products of $\Wbb$ along $\fx$. Then there exists a unique holomorphic map $b\in\MB\mapsto \Phi_b\in\Hom_{\Vbb^{\otimes M}}(\wtd\bbs_\fx(\Wbb),\bbs_\fx(\Wbb))$ such that for each $b\in\MB$, we have that $\wtd\gimel|_b=\gimel|_b\circ(\idt\otimes \Phi_b)$, and that $\Phi_b$ is a $\Vbb^{\otimes\Mbb}$-module isomorphism.
\end{co}

We call $\Phi$ the \textbf{canonical isomorphism} from $(\wtd\bbs_\fx(\Wbb),\wtd\gimel)$ to $(\bbs_\fx(\Wbb),\gimel)$.

\begin{proof}
The existence and uniqueness of $\Phi$ satisfying the desired properties, except that $\Phi_b$ is bijective, is clear from Prop. \ref{lbb39}. Similarly, there is a unique $\Psi:\bbs_\fx(\Wbb)\rightarrow\wtd\bbs_\fx(\Wbb)$ satisfying all the properties of a canonical isomorphism except that $\Psi_b$ is bijective. Thus $\gimel_b=\gimel_b\circ(\idt\otimes(\Phi_b\circ\Psi_b))$, which shows that $\Phi_b\circ\Psi_b=\idt$. Similarly $\Psi_b\circ\Phi_b=\idt$. Therefore $\Phi_b$ is bijective. 
\end{proof}

When $\dim\MB>0$, a dual fusion product does not necessarily exist.  However, the following theorem shows that the existence holds locally. Therefore, using the canonical isomorphisms in Cor. \ref{lbb40}, one can glue all $\bbs_{\fx_V}(\Wbb)\otimes_\Cbb \mc O_V$ together (for all open $V\subset\MB$ such that a dual fusion product $\bbs_{\fx_V}(\Wbb)$ exists for $\Wbb$ along $\fx_V$) and get a \textbf{bundle of dual fusion products}. (More precisely, it is an infinite-rank locally free $\mc O_\MB$-module.) We will not explore this topic further in this paper.

\begin{thm}\label{lbb37}
Assume that $\MB=\mc D_r\times\MB_0$ where $\MB_0$ is a Stein manifold and $0<r\leq +\infty$. Let
\begin{align*}
\fx_0=(\tau_\star;\theta_\star|\pi_0:\MC_0\rightarrow\MB_0|\sgm_\blt;\eta_\blt)
\end{align*}
be the restriction of $\fx$ to $\MB_0\simeq\{0\}\times\MB_0$ as in \eqref{eqb65}. Assume that $\Wbb$ has a dual fusion product $(\bbs_{\fx_0}(\Wbb),\gimel_0)$ along $\fx_0$. Then $\Wbb$ has a dual fusion product $(\bbs_\fx(\Wbb),\gimel)$ along $\fx$.
\end{thm}

\begin{proof}
Let $q$ be the standard coordinate of $\MD_r$. Consider $\partial_q$ as a constant vector field on $\MB$. Set $\bbs_\fx(\Wbb)=\bbs_{\fx_0}(\Wbb)$. Since $\MB_0$ and $\MD_r$ are Stein, $\MB$ is Stein. So by Rem. \ref{lbb12}, $\partial_q$ has a lift $\xk$, by which one can define the differential operator $\nabla_{\partial_q}$ on $\scr W_\fx(\Wbb),\scr T_\fx(\Wbb),\scr T_\fx^*(\Wbb)$. Therefore, since $\gimel_0\in H^0(\MB_0,\scr T_{\fx_0}^*(\Wbb\otimes\bbs_\fx(\Wbb)))$, by Thm. \ref{lbb26}, the formal parallel transport $\gimel$ of $\gimel_0$ converges a.l.u. to an element of $H^0(\MB,\scr T_\fx^*(\Wbb\otimes\bbs_\fx(\Wbb)))$. In other words, $\gimel$ can be viewed as a linear map $\Wbb\rightarrow\mc O(\MB)$ whose restriction to each $(p,b)\in\MB$ is a conformal block associated to $\Wbb$ and $\fx_{(p,b)}$.

We want to show that for each $(p,b)\in\MB$, $(\bbs_\fx(\Wbb),\gimel|_{(p,b)})$ is a dual fusion product of $\Wbb$ along $\fx_{(p,b)}$. For that purpose, it suffices to assume that $\MB_0$ is a single point, say $\{0\}$. Then $\MB=\MD_r\times\{0\}\simeq\MD_r$. Therefore, $p\in\MD_r$ denotes a general element of $\MB$. By Thm. \ref{lbb35}, for each grading restricted $\Vbb^{\otimes M}$-module $\Mbb$ associated to $\tau_\star(\MB)$, $\scr T_\fx^*(\Wbb\otimes\Mbb)$ is a (finite-rank) holomorphic vector bundle on $\MD_r$. Then $\nabla_{\partial_q}$ defines an (obviously flat) connection $\nabla$ on $\scr T_\fx^*(\Wbb\otimes\Mbb)$, and $\gimel$ is parallel under this connection if $\Mbb=\bbs_\fx(\Wbb)$. 

Choose any $\upphi_p\in\scr T_{\fx_p}^*(\Wbb\otimes\Mbb)$. Let $\upphi_0\in\scr T_{\fx_0}^*(\Wbb\otimes \Mbb)$ be the parallel transport of $\upphi_p$ (under $\nabla$) to $0$. Since $(\bbs_\fx(\Wbb),\gimel_0)$ is a dual fusion product along $\fx_0$, there exists $T\in\Hom_{\Vbb^{\otimes M}}(\Mbb,\bbs_\fx(\Wbb))$ such that $\upphi_0=\gimel_0\circ(\idt\otimes T)=\gimel\circ(\idt\otimes T)|_0$. Since $\gimel$ is parallel under $\nabla$ and $T$ intertwines the actions of $\Vbb^{\otimes M}$, it is clear that $\gimel\circ(\idt\otimes T)\in H^0(\MB,\scr T_\fx^*(\Wbb\otimes \Mbb))$ is parallel under $\nabla$. So $\gimel|_p\circ(\idt\otimes T)$ and $\upphi_p$ are both the parallel transport of $\upphi_0$ to $p$. Thus $\gimel|_p\circ(\idt\otimes T)=\upphi_p$.

On the other hand, if $S\in\Hom_{\Vbb^{\otimes M}}(\Mbb,\bbs_\fx(\Wbb))$ also satisfies $\gimel|_p\circ(\idt\otimes S)=\upphi_p$, then $\upphi_0$ is the parallel transport of $\gimel|_p\circ(\idt\otimes S)$ to $0$, and hence $\upphi_0=\gimel|_0\circ(\idt\otimes S)$. Similarly $\upphi_0=\gimel|_0\circ(\idt\otimes T)$, which shows $T=S$ because $(\bbs_\fx(\Wbb),\gimel_0)$ is a dual fusion product. This finishes the proof that $(\bbs_\fx(\Wbb),\gimel|_p)$ is a dual fusion product of $\Wbb$ along $\fx_p$.
\end{proof}

\begin{co}\label{lbb54}
Suppose that $\MB$ is an open polydisk, i.e., it is of the form $\MD_{r_1}\times\cdots\times\MD_{r_m}\subset\Cbb^m$ where $m\in\Nbb$. Then there exists a dual fusion product $(\bbs_\fx(\Wbb),\gimel)$ of $\Wbb$ along $\fx$.
\end{co}

\begin{proof}
This follows immediately from Thm. \ref{lbb37} and from induction on $m$.
\end{proof}

\section{Convergence of sewing conformal blocks}

In this chapter, we continue to assume Asmp. \ref{lbb1}, and use freely the notations in Subsec. \ref{lbb44}. In particular,
\begin{align*}
\fx=(\pi:\MC\rightarrow \MB\big| \sgm_\blt)=(\pi:\MC\rightarrow \MB\big|\sgm_1,\cdots,\sgm_N)
\end{align*}
is obtained by sewing the family
\begin{align*}
    \wtd \fx=(\wtd \pi:\wtd \MC\rightarrow \wtd \MB\big|\sgm_\blt\big\Vert \sgm_\blt',\sgm_\blt'')
    =(\wtd \pi:\wtd \MC\rightarrow \wtd \MB\big|\sgm_1,\cdots,\sgm_N\big\Vert \sgm_1',\cdots,\sgm_R',\sgm_1'',\cdots,\sgm_R'')
\end{align*}
satisfying \eqref{eqb52}. Assume  that $\wtd\fx$ has local coordinates $\eta_\blt,\xi_\blt,\varpi_\blt$ at $\sgm_\blt(\wtd\MB),\sgm_\blt'(\wtd\MB),\sgm''_\blt(\wtd\MB)$, and extend $\eta_\blt$ constantly to local coordinates of $\fx$ at $\sgm_\blt(\MB)$ (also denoted by $\eta_\blt$), cf. Asmp. \ref{lbb8}.

We always assume that $\Vbb=\bigoplus_{n\in\Nbb}\Vbb(n)$ is a $C_2$-cofinite VOA, and $\Wbb$ is a grading-restricted $\Vbb^{\otimes N}$-module associated to $\sgm_\blt(\wtd\MB)$ and also to $\sgm_\blt(\MB)$. Let $\Mbb$ be a grading-restricted $\Vbb^{\otimes R}$-module with (automatically grading-restricted) contragredient $\Mbb'$. Associate $\Mbb,\Mbb'$ to $\sgm_\blt'(\wtd\MB)$ and $\sgm_\blt''(\wtd\MB)$ respectively. We use $\mc U(\eta_\blt,\xi_\blt,\varpi_\blt)$ and $\mc U(\eta_\blt)$ to make the identifications
\begin{align}\label{eqb85}
\scr W_{\wtd\fx}(\Wbb\otimes\Mbb\otimes\Mbb')=\Wbb\otimes\Mbb\otimes\Mbb'\otimes_\Cbb\MO_{\wtd\MB}\qquad  \scr W_\fx(\Wbb)=\Wbb\otimes_\Cbb\MO_\MB
\end{align}

\subsection{Sewing conformal blocks}

\subsubsection{The decomposition $L_j(0)=\Ljss+\Ljni$}

Recall the grading $\Mbb=\bigoplus_{\lambda_\blt\in\Cbb^R}\Mbb_{[\lambda_\blt]}$, cf. \eqref{eqb68}.   Since $\Mbb$ is grading-restricted, by the first part of the proof of \cite[Thm. A.14]{GZ1}, there exists a finite subset $E\subset\Cbb^R$ such that
\begin{gather}\label{eqb72}
\Mbb=\bigoplus_{\lambda_\blt\in E+\Nbb^R}\Mbb_{[\lambda_\blt]}\qquad\text{and}\qquad \dim\Mbb_{[\lambda_\blt]}<+\infty\text{ for each }\lambda_\blt
\end{gather}
For each $1\leq j\leq R$, we let $\Ljss\in\End(\Mbb)$ be defined such that $\Ljss\big|_{\Mbb_{[\lambda_\blt]}}=\lambda_j$, and let $\Ljni=L_j(0)-\Ljss$. 

Since for each $v\in\Vbb$ and $1\leq i,j\leq R$ and $n\in\Zbb$ we have on $\Mbb$ that
\begin{align*}
[L_j(0),Y_i(v)_n]=[\Ljss,Y_i(v)_n]=\delta_{i,j}Y_i((L_j(0)-n-1)v)_n
\end{align*}
we see that $\Ljni\in\End_{\Vbb^{\otimes R}}(\Mbb)$ when $\Ljni$ acts on $\Mbb$. In particular, we have $[\Ljni,L_i(0)]=0$. So $\Ljni$ preserves each generalized eigenspace of $L_i(0)$, i.e.
\begin{align*}
[\Ljni,\Liss]=0\qquad\text{and hence}\qquad[\Ljni,\Lini]=0
\end{align*}
Thus $\Mbb_{[\lambda_\blt]}$ is $\Ljni$-invariant. Therefore, on the finite-dimensional space $\Mbb_{[\lambda_\blt]}$,
\begin{align*}
L_j(0)\big|_{\Mbb_{[\lambda_\blt]}}=\Ljss\big|_{\Mbb_{[\lambda_\blt]}}+\Ljni\big|_{\Mbb_{[\lambda_\blt]}}
\end{align*}
is the (unique) Jordan-Chevalley decomposition of $L_j(0)\big|_{\Mbb_{[\lambda_\blt]}}$. Hence $\Ljni\big|_{\Mbb_{[\lambda_\blt]}}$ is nilpotent. Since $\Mbb$ is finitely-generated as a $\Vbb^{\otimes R}$-module (\cite{Hua-projectivecover}), it follows that
\begin{align}\label{eqb70}
\Ljni\text{ is nilpotent on }\Mbb
\end{align}

\subsubsection{The element $q_\blt^{L_\blt(0)}\btr\otimes \btl$ of $(\Mbb\otimes\Mbb')\{q_\blt\}[\log q_\blt]$}

We treat $q_1,\dots,q_R$ and $\log q_1,\dots,\log q_R$ as mutually commuting and independent formal variables. Each $A\in\End(\Mbb)$ acts on $(\Mbb\otimes\Mbb')\{q_\blt\}$ resp. $(\Mbb\otimes\Mbb')\{q_\blt\}[\log q_\blt]$ by acting on the coefficient before each power of $q_\blt$ and $\log q_\blt$. Define
\begin{align*}
\btr\otimes\btl\in(\Mbb'\otimes\Mbb)^*\qquad\bk{\btr\otimes\btl,m'\otimes m}=\bk{m',m}
\end{align*}
for all $m\in\Mbb,m'\in\Mbb'$. For each $A\in\End(\Mbb)$ such that $A^\tr\in\End(\Mbb')$ exists (i.e. $\bk{Am,m'}=\bk{m,A^\tr m'}$ for all $m\in\Mbb,m'\in\Mbb'$), we have a linear map
\begin{align*}
A:(\Mbb'\otimes\Mbb)^*\rightarrow (\Mbb'\otimes\Mbb)^*\qquad (A\phi)(m'\otimes m)=\phi(A^\tr m'\otimes m)
\end{align*}
for each $\phi\in(\Mbb'\otimes\Mbb)^*$. We set
\begin{align*}
\btr\otimes A^\tr\btl:=A\btr\otimes\btl
\end{align*}

Note that $\Mbb\otimes\Mbb'$ is canonically a subspace of $(\Mbb'\otimes\Mbb)^*$. Consider $P_{\lambda_\blt}$, the projection of $\Mbb$ onto $\Mbb_{[\lambda_\blt]}$. Then one checks easily that
\begin{align}\label{eqb69}
P_{\lambda_\blt}\btr\otimes A^\tr\btl=\sum_{\alpha\in\fk A_{\lambda_\blt}} m_{(\lambda_\blt,\alpha)}\otimes A^\tr \wch m_{(\lambda_\blt,\alpha)}\qquad\in \Mbb\otimes\Mbb'
\end{align}
where $(m_{(\lambda_\blt,\alpha)})_{\alpha\in\fk A_{\lambda_\blt}}$ is any (finite) basis of $\Mbb_{[\lambda_\blt]}$ with dual basis $(\wch m_{(\lambda_\blt,\alpha)})_{\alpha\in\fk A_{\lambda_\blt}}$.

\begin{df}
For each $A\in\End(\Mbb)$ such that $A^\tr\in\End(\Mbb')$ exists, define
\begin{gather*}
q_\blt^\Lbss\btr\otimes A^\tr\btl=\sum_{\lambda_\blt\in\Cbb^R} q_\blt^{\lambda_\blt}\cdot P_{\lambda_\blt}\btr\otimes A^\tr\btl\qquad \in(\Mbb\otimes\Mbb')\{q_\blt\}
\end{gather*}
where $q_\blt^{\lambda_\blt}=q_1^{\lambda_1}\cdots q_R^{\lambda_R}$. Thus, its evaluation with each $m'\otimes m\in\Mbb'\otimes\Mbb$ is
\begin{align*}
\bk{A^\tr q_\blt^{\Lbss}m',m}=\bk{q_1^{L_1(0)_{\mathrm s}}\cdots q_R^{L_R(0)_{\mathrm s}}m',Am}
\end{align*}
\end{df}

\begin{df}
Define a linear map
\begin{gather*}
q_\blt^{\Lbni}:\Mbb\otimes\Mbb'\rightarrow(\Mbb\otimes\Mbb')[\log q_\blt] \qquad q_\blt^{\Lbni}(m\otimes m')=(q_\blt^{\Lbni}m\otimes m')
\end{gather*}
where
\begin{align*}
q_\blt^{\Lbni}m=\sum_{k_\blt\in\Nbb^R}(\log q_1)^{k_1}\cdots(\log q_R)^{k_R}\cdot\frac{L_1(0)_{\mathrm n}^{k_1}}{k_1!}\cdots\frac{L_R(0)_{\mathrm n}^{k_R}}{k_R!} m
\end{align*}
By \eqref{eqb70}, the order of power of each $\log q_i$ in $q_\blt^{\Lbni}m$ has an upper bound independent of the choice of $m\in\Mbb$. Thus, we can define
\begin{gather}\label{eqb71}
\begin{gathered}
q_\blt^{\Lbni}:(\Mbb\otimes\Mbb')\{q_\blt\}\rightarrow (\Mbb\otimes\Mbb')\{q_\blt\}[\log q_\blt]\\[0.5ex] \sum_{\lambda_\blt\in\Cbb^R}\chi_{\lambda_\blt}\cdot q_\blt^{\lambda_\blt}\mapsto \sum_{\lambda_\blt\in\Cbb^R}(q_\blt^{\Lbni}\chi_{\lambda_\blt})\cdot q_\blt^{\lambda_\blt}
\end{gathered}
\end{gather}
Finally, for each $A\in\End(\Mbb)$ such that $A^\tr\in\End(\Mbb')$ exists, define
\begin{gather*}
q_\blt^{L_\blt(0)} \btr\otimes A^\tr\btl\in(\Mbb\otimes\Mbb')\{q_\blt\}[\log q_\blt]\qquad q_\blt^{L_\blt(0)} \btr\otimes A^\tr\btl=q_\blt^{\Lbni}\big(q_\blt^{\Lbss}\btr\otimes A^\tr\btl\big)
\end{gather*}
In particular, $q_\blt^{L_\blt(0)} \btr\otimes\btl$ is an element of $(\Mbb\otimes\Mbb')\{q_\blt\}[\log q_\blt]$. 
\end{df}

Note that by \eqref{eqb69} we have
\begin{subequations}\label{eqb73}
\begin{gather}
Aq_\blt^{L_\blt(0)} \btr\otimes\btl=\sum_{\lambda_\blt\in\Cbb^R}q_\blt^{\lambda_\blt}\sum_{\alpha\in\fk A_{\lambda_\blt}} Aq_\blt^{\Lbni}m_{(\lambda_\blt,\alpha)}\otimes \wch m_{(\lambda_\blt,\alpha)}\\
q_\blt^{L_\blt(0)} \btr\otimes A^\tr\btl=\sum_{\lambda_\blt\in\Cbb^R}q_\blt^{\lambda_\blt}\sum_{\alpha\in\fk A_{\lambda_\blt}} q_\blt^{\Lbni}m_{(\lambda_\blt,\alpha)}\otimes A^\tr \wch m_{(\lambda_\blt,\alpha)}
\end{gather}
\end{subequations}

\begin{rem}\label{lbb43}
Since $\Ljni\in\End_{\Vbb^{\otimes R}}(\Mbb)$, the map \eqref{eqb71} intertwines the action of $Y_i(v)_n$ for each $1\leq i\leq R,v\in\Vbb,n\in\Zbb$.
\end{rem}

\begin{rem}\label{lbb42}
By \eqref{eqb72}, each $\Mbb$ is a finite direct sum of submodules that are $\Lbss$-simple. Now assume that $\Mbb$ is \pmb{$\Lbss$}\textbf{-simple}, i.e., the set $E$ in \eqref{eqb72} can be chosen to be a single point set $\{\kappa_\blt\}$. Define
\begin{align}
\wtd L_j(0)=\Ljss-\kappa_j
\end{align}
So the eigenvalues of $\wtd L_j(0)$ are in $\Nbb$. Then $\wtd L_\blt(0)$ makes $\Mbb$ a finitely-admissible $\Vbb^{\times R}$-module.
\end{rem}

\subsubsection{Sewing conformal blocks}\label{lbb61}

For each $\uppsi\in \ST_{\wtd \fx}^*(\Wbb\otimes \Mbb\otimes \Mbb')(\wtd \MB)$ and $w\in \Wbb$, define 
\begin{align}
    \MS\uppsi (w)=\uppsi\big(w\otimes q_\blt^{L_\blt(0)}\btr\otimes \btl\big)\in \MO(\wtd \MB)\{q_\blt\}[\log q_\blt].
\end{align}
$\MS\uppsi$ is called \textbf{the sewing of \pmb{$\uppsi$}} along pairs of points $\sgm_\blt'(\wtd\MB),\sgm_\blt''(\wtd\MB)$. Write 
\begin{align}
    \MS\uppsi(w)=\sum_{n_\blt\in \Cbb^R,l_\blt\in \Nbb^R} \MS\uppsi(w)_{n_\blt,l_\blt} q_\blt^{n_\blt}(\log q_\blt)^{l_\blt}
\end{align}
Then $\mc S\uppsi:\Wbb\rightarrow \MO(\wtd \MB)\{q_\blt\}[\log q_\blt]$ can be extended in an obvious way to an $\mc O(\wtd\MB)[[q_\blt]]$-module morphism
\begin{align}\label{eqb83}
\mc S\uppsi:\Wbb\otimes\mc O(\wtd\MB)[[q_\blt]]\rightarrow \MO(\wtd \MB)\{q_\blt\}[\log q_\blt]
\end{align}
In particular, $\mc S\uppsi$ can be defined on $\Wbb\otimes\mc O(\MB)$.

\begin{df}\label{lbb51}
    We say $\MS \uppsi$ \textbf{converges a.l.u.} if for each $w\in\Wbb$ and each compact subsets $K\subset \wtd\MB$ and $Q\subset \MD_{r_\blt \rho_\blt}^\times$, there exists $C>0$ such that 
    \begin{align*}
        \sum_{n_\blt\in \Cbb^R}|\MS\uppsi(w)_{n_\blt,l_\blt}(b)|\cdot |q_\blt^{n_\blt}|\leq C
    \end{align*}
holds for any $b\in K$, $q_\blt=(q_1,\cdots,q_R)\in Q$, and $l_\blt\in \Nbb^R$.
\end{df}

Note that the a.l.u. convergence of $\MS\uppsi$ is slightly stronger than the condition that for each $w\in\Wbb$, the series of functions $\sum_{n_\blt\in \Cbb^R,l_\blt\in \Nbb^R} \MS\uppsi(w)_{n_\blt,l_\blt} q_\blt^{n_\blt}(\log q_\blt)^{l_\blt}$ converge a.l.u. on $\MB-\Delta=\MD_{r_\blt\rho_\blt}^\times\times\wtd\MB$.

\subsection{The sewing of a conformal block is a formal conformal block}       

If $1\leq k\leq R$, we set $q_{\blt\setminus k}=(q_1,\dots,q_{k-1},q_{k+1},\dots,q_R)$.
 
\begin{pp}\label{multisew7}
    Let $\scr A$ be a unital commutative $\Cbb$-algebra. For any $u\in \Vbb$, $1\leq k\leq R$ and $f\in \scr A[[\xi_k,\varpi_k,q_{\blt\backslash k}]]$, the following elements of $(\Mbb\otimes \Mbb'\otimes \scr A)\{q_\blt\}[\log q_\blt]$ are equal:
    \begin{equation}\label{multisew2}
    \begin{aligned}
        &\Res_{\xi_k=0}~Y_{\Mbb,k}(\xi_k^{L(0)}u,\xi_k)q_\blt^{L_\blt(0)}\btr\otimes \btl \cdot f\big(\xi_k,\frac{q_k}{\xi_k},q_{\blt\backslash k}\big)\frac{d\xi_k}{\xi_k}\\
        =&\Res_{\varpi_k=0}~q_\blt^{L_\blt(0)}\btr\otimes Y_{\Mbb',k}(\varpi_k^{L(0)}\MU(\upgamma_1)u,\varpi_k) \btl\cdot f\big(\frac{q_k}{\varpi_k},\varpi_k,q_{\blt\backslash k}\big)\frac{d\varpi_k}{\varpi_k}
    \end{aligned}
    \end{equation}
\end{pp}

\begin{rem}
Using \eqref{eqb73}, one easily sees that
\begin{gather*}
Y_{\Mbb,k}(\xi_k^{L(0)}u,\xi_k)q_\blt^{L_\blt(0)}\btr\otimes \btl\qquad\in (\Mbb\otimes\Mbb')((\xi_k))\{q_\blt\}[\log q_\blt]\\
q_\blt^{L_\blt(0)}\btr\otimes Y_{\Mbb',k}(\varpi_k^{L(0)}\MU(\upgamma_1)u,\varpi_k) \btl\qquad \in(\Mbb\otimes\Mbb')((\varpi_k))\{q_\blt\}[\log q_\blt]
\end{gather*}
Multiplying these two elements respectively by
\begin{gather*}
f\big(\xi_k,\frac{q_k}{\xi_k},q_{\blt\backslash k}\big)\quad \in \scr A((\xi_k))[[q_\blt]]\qquad\text{and}\qquad f\big(\frac{q_k}{\varpi_k},\varpi_k,q_{\blt\backslash k}\big)\quad\in \scr A((\varpi_k))[[q_\blt]]
\end{gather*}
they become elements of
\begin{gather*}
(\Mbb\otimes\Mbb'\otimes \scr A)((\xi_k))\{q_\blt\}[\log q_\blt]\qquad\text{resp.}\qquad (\Mbb\otimes\Mbb'\otimes \scr A)((\varpi_k))\{q_\blt\}[\log q_\blt]
\end{gather*}
Thus, the residues in \eqref{multisew2} make sense.
\end{rem}

\begin{proof}[Proof of Prop. \ref{multisew7}]
Without loss of generality, we assume $\Mbb$ is $\Lbss$-simple. When $L_\blt(0)$ is replaced by $\wtd L_\blt(0)$ and $R=2$, Eq. \eqref{multisew2} holds in $(\Mbb\otimes \Mbb'\otimes \scr A)[[q_\blt]]$ by Step 1 of the proof of \cite[Prop. 2.32]{GZ1}; for general $R$ the proof is similar. Since $\wtd L_j(0)$ and $\Ljss$ differ by a scalar, Eq. \eqref{multisew2} holds in $(\Mbb\otimes \Mbb'\otimes \scr A)\{q_\blt\}$ when $L_\blt(0)$ is replaced by $\Lbss$. Multiplying both sides of this result by the map $q_\blt^{\Lbni}=\eqref{eqb71} $ and noting Rem. \ref{lbb43}, we obtain \eqref{multisew2}. (See also the proof of \cite[Lem. 10.2]{Gui-sewingconvergence}.)
\end{proof}

\begin{pp}\label{lbb48}
    If $\uppsi\in H^0\big(\wtd \MB,\ST_{\wtd \fx}^*(\Wbb\otimes \Mbb\otimes \Mbb')\big)$, then $\MS \uppsi$ is a formal conformal block, i.e., $\MS\uppsi=\eqref{eqb83}$ vanishes on $\SJ_\fx^\pre(\MB)$ (defined in \eqref{eqb58}).
\end{pp}

When $\wtd\MB$ is a single point, this proposition becomes a special case of \cite[Prop. 2.32]{GZ1}. Although the following proof is similar to the one of \cite[Prop. 2.32]{GZ1}, we have included it here because the setting and the notations in the proof will be useful for the discussions in the third paper of our series.

\begin{proof}
Step 1. Choose $v\in H^0\big(\MC,\SV_{\fx}\otimes \omega_{\MC/\MB}(\blt S_\fx)\big)$. We claim that outside the nodes we have power series expansion
    \begin{align}\label{multisew3}
        v=\sum_{n_\blt \in \Nbb^R} v_{n_\blt}q_\blt^{n_\blt}\qquad\text{where }v_{n_\blt}\in H^0\big(\wtd \MC,\SV_{\wtd \fx}\otimes \omega_{\wtd\MC/\wtd \MB}(\blt S_{\wtd \fx})\big)
    \end{align}
We use freely the notations in Subsec. \ref{lbb44}. Choose a precompact open subset $\wtd U$ of $\wtd \MC$ disjoint from $\sgm_\blt'(\wtd\MB)$ and $\sgm_\blt''(\wtd \MB)$. Choose small enough $\eps_\blt,\kappa_\blt$ such that $\wtd U\times \MD_{\eps_\blt \kappa_\blt}$ is an open subset of $\wtd \MC\times \MD_{r_\blt \rho_\blt}-\bigcup_{i=1}^R (F_i'\cup F_i'')$. Then $\pi:\MC\rightarrow \MB$, when restricted to $\wtd U\times \MD_{\eps_\blt\kappa_\blt}$, becomes $\wtd \pi\times \id:\wtd U\times \MD_{\eps_\blt\kappa_\blt}\rightarrow\wtd\MB\times\MD_{\eps_\blt\kappa_\blt}$. Thus $v|_{\wtd U\times \MD_{\eps_\blt \kappa_\blt}}$ can be viewed as a global section of $\SV_{\wtd \fx\times \MD_{r_\blt \rho_\blt}}\otimes \omega_{\wtd \MC\times \MD_{\eps_\blt \kappa_\blt}/\wtd \MB\times \MD_{\eps_\blt \kappa_\blt}}(\blt S_\fx)$, and hence an element in $H^0\big(\wtd U,\SV_{\wtd \fx}\otimes \omega_{\wtd \MC/\wtd \MB}(\blt S_{\wtd \fx})\big)[[q_\blt]]$. The coefficient before $q_\blt^{n_\blt}$ defines $v_{n_\blt}|_{\wtd U}$ which are clearly compatible for different $\wtd U$. Therefore, by considering all such $\wtd U$, we obtain a section $v_{n_\blt}$ of $\SV_{\wtd \fx}\otimes \omega_{\wtd \MC/\wtd \MB}(\blt S_{\wtd \fx})$ on $\wtd \MC-\bigcup_{i=1}^R \big(\sgm_i'(\wtd \MB)\cup \sgm_i''(\wtd \MB)\big)$.

To finish proving \eqref{multisew3},  it remains to prove that each $v_{n_\blt}$ has finite poles at $\sgm_i'(\wtd \MB)$ and $\sgm_i''(\wtd \MB)$ for each $i$. By the description of $\SV_\fx$ and $\omega_{\MC/\MB}$ in Def. \ref{lbb45} and Rem. \ref{lbb4}, $v|_{W_i-\Sigma}$ is a finite sum of elements whose restrictions to $W_i'$ (resp. $W_i''$) under trivializations $\MU_\varrho(\xi_i)$ (resp. $\MU_\varrho(\varpi_i)$) are 
    \begin{align}\label{multisew4}
        f(\xi_i,q_i/\xi_i,q_{\blt\backslash i})\xi_i^{L(0)}u\cdot\frac{d\xi_i}{\xi_i} \qquad \text{resp.}\qquad -f(q_i/\varpi_i,\varpi_i,q_{\blt\backslash i})\varpi_i^{L(0)}\MU(\upgamma_1)u\cdot \frac{d\varpi_i}{\varpi_i}
    \end{align}
    where $u\in \Vbb$ and $f\in \MO(W_i)$. (Here, we have suppressed the local coordinates $\tau_\blt$ of $\wtd\MB$ in the parentheses of $f$. The choice of $\tau_\blt$ is clearly irrelevant.) By taking power series expansion of \eqref{multisew4}, we see the term before $q_\blt^{n_\blt}$ has poles of orders at most $n_i+1$ at $\xi_i=0$ (resp. $\varpi_i=0$). The same can be said about $v_{n_\blt}$.\\[-1ex]

Step 2. By Prop. \ref{multisew7} and \eqref{multisew4}, in $(\Mbb\otimes\Mbb'\otimes \mc O(\wtd\MB))\{q_\blt\}[\log q_\blt]$  we have
    \begin{align}\label{multisew5}
        \sum_{n_\blt \in \Nbb^R} \big(v_{n_\blt}\cdot q_\blt^{L_\blt(0)}\btr\otimes \btl+q_\blt^{L_\blt(0)}\btr\otimes v_{n_\blt}\btl\big)q_\blt^{n_\blt}=0
    \end{align}
    Since $v_{n_\blt}\in H^0\big(\wtd \MC,\SV_{\wtd \fx}\otimes \omega_{\wtd\MC/\wtd \MB}(\blt S_{\wtd \fx})\big)$ and $\uppsi$ is a conformal block, we have 
    \begin{align}\label{multisew6}
        \uppsi\big(v_{n_\blt}\cdot w\otimes q_\blt^{L_\blt (0)}\btr\otimes \btl+w\otimes v_{n_\blt}\cdot q_\blt^{L_\blt (0)}\btr\otimes \btl+w\otimes  q_\blt^{L_\blt (0)}\btr\otimes v_{n_\blt}\cdot \btl\big)=0
    \end{align}
    Therefore, for each $w\in\Wbb$ we have
    \begin{align*}
        &\MS\uppsi(v\cdot w)=\sum_{n_\blt\in \Nbb^R}\uppsi\big(v_{n_\blt}\cdot w\otimes q_\blt^{L_\blt(0)}\btr\otimes \btl\big)q_{\blt}^{n_\blt}\\
        &\xlongequal{\eqref{multisew5}}\sum_{n_\blt\in \Nbb^R}\uppsi\big(v_{n_\blt}\cdot w\otimes q_\blt^{L_\blt(0)}\btr\otimes \btl+w\otimes v_{n_\blt}\cdot q_\blt^{L_\blt (0)}\btr\otimes \btl+w\otimes  q_\blt^{L_\blt (0)}\btr\otimes v_{n_\blt}\cdot \btl\big)q_{\blt}^{n_\blt}
    \end{align*}
which equals $0$ by \eqref{multisew6}. Since any element of $\SJ_\fx^\pre(\MB)$ is an $\mc O(\MB)$-linear combination of such $v\cdot w$, we conclude that $\mc S\uppsi$ vanishes on $\SJ_\fx^\pre(\MB)$.
\end{proof}

\subsection{Convergence of sewing conformal blocks}

The following theorem specializes to \cite[Thm. 13.1]{Gui-sewingconvergence} in the special case that $\Mbb$ is a tensor product of grading-restricted $\Vbb$-modules. In fact, the proof for that special case applies almost directly to the present general situation. However, in \cite{Gui-sewingconvergence}, the proof of Thm. 13.1 was only outlined in Sec. 13, with the detailed proof provided solely for the case that $R=1$ in Sec. 11. Therefore, we now provide the detailed proof of Thm. \ref{lbb49} for arbitrary $R$ and arbitrary grading-restricted $\Vbb^{\otimes R}$-module $\Mbb$.

Recall that $\Delta=\eqref{eqb45}$ is the discriminant locus of $\fx$, and hence $\MB-\Delta=\MD_{r_\blt\rho_\blt}^\times\times\wtd\MB$. We let $\wht\MD_{r_\blt\rho_\blt}^\times$ be the universal cover of $\MD_{r_\blt\rho_\blt}^\times$. Let
\begin{align}\label{eq1}
\wht\MB^\times=\wht\MD_{r_\blt\rho_\blt}^\times\times\wtd\MB
\end{align}

\begin{thm}\label{lbb49}
Let $\uppsi\in H^0\big(\wtd \MB,\ST_{\wtd \fx}^*(\Wbb\otimes \Mbb\otimes \Mbb')\big)$. Then $\MS\uppsi$ converges a.l.u. (in the sense of Def. \ref{lbb51}) to a linear map $\Wbb\rightarrow \MO(\wht\MB^\times)$, which belongs to $H^0\big(\wht\MB^\times,\ST_\fx^*(\Wbb)\big)$ (cf. Rem. \ref{lbb27}).
\end{thm}

\begin{proof}
Step 1. The theorem can be proved locally with respect to $\wtd\MB$. Thus, we may assume that $\wtd\MB$ is an open ball in $\Cbb^m$ with standard coordinates $\tau_\blt$. Therefore, by Thm. \ref{lbb47}, $\wtd\fx$ has a projective structure $\mbf P$.  We abbreviate each $\tau_j\circ\pi$ to $\tau_j$.   Moreover, once we can prove the a.l.u. convergence, then the fact that $\mc S\uppsi$ belongs to $H^0\big(\wht\MB^\times,\ST_\fx^*(\Wbb)\big)$ will follow from Prop. \ref{lbb48} and Thm. \ref{lbb28}. 

Fix $1\leq k\leq R$. By Rem. \ref{lbb12}, we have a lift $\wtd\yk\in H^0\big(\MC,\Theta_\MC(-\log \MC_\Delta+\blt S_\fx)\big)$ for $q_k\partial_{q_k}$ and use $\wtd\yk$ to define the differential operator $\nabla_{q_k\partial_{q_k}}$, cf. Def. \ref{lbb16}. Let 
    \begin{align*}
        \Gamma=\bigcup_{i=1}^R\big(\sgm_i'(\wtd \MB)\cup \sgm_i''(\wtd\MB)\big)
    \end{align*}
    Choose any precompact open subset $\wtd U\subset \wtd \MC-\Gamma$ equipped with a fiberwise univalent $\eta\in \MO(\wtd U)$. We can find a sufficiently small sub-polydisk $\MD_{\varepsilon_\blt\kappa_\blt}\subset\MD_{r_\blt\rho_\blt}$ such that $\wtd U\times \MD_{\varepsilon_\blt \kappa_\blt}\simeq \MD_{\varepsilon_\blt \kappa_\blt}\times \wtd U$ is an open subset of $\wtd \MC\times \MD_{r_\blt \rho_\blt}-\bigcup_{i=1}^R \big(F_i'\cup F_i''\big)$. (We will freely switch the order of product of manifolds.)  After extending $\eta$ constantly to a fiberwise univalent function on $U:=\MD_{\varepsilon_\blt \kappa_\blt}\times \wtd U$, as in \eqref{eqb31}, we have
\begin{subequations}\label{eqb80}
    \begin{align}\label{eqb80a}
        \wtd \yk |_U=h(q_\blt,\eta,\tau_\blt)\partial_\eta +q_k\partial_{q_k}
    \end{align}
where $h\in\mc O((q_\blt,\eta,\tau_\blt)(U-\SX))$ has finite poles at $(q_\blt,\eta,\tau_\blt)(\SX)=\MD_{\eps_\blt\kappa_\blt}\times (\eta,\tau_\blt)(S_{\wtd\fx})$. Write
\begin{align}\label{eqb80b}
h=\sum_{n_\blt \in \Nbb^R} h_{n_\blt}(\eta,\tau_\blt)q_\blt^{n_\blt}
\end{align}
where each $h_{n_\blt}\in \mc O((\eta,\tau_\blt)(\wtd U-S_{\wtd\fx}))$ has finite poles at $(\eta,\tau_\blt)(S_{\wtd\fx})$, and set
    \begin{align}\label{eqb80c}
        \wtd \yk_{n_\blt}^\perp =h_{n_\blt}(\eta,\tau_\blt)\partial_\eta\qquad\in H^0(\wtd U,\Theta_{\wtd \MC/\wtd \MB}(\blt S_{\wtd \fx}))
    \end{align}
\end{subequations}
Then $\yk_{n_\blt}^\perp$ is independent of the choice of $\eta$, and hence can be extended to an element of $H^0\big(\wtd \MC-\Gamma,\Theta_{\wtd \MC/\wtd \MB}(\blt S_{\wtd\fx})\big)$. To see this, suppose we have another $\mu\in \MO(\wtd U)$ univalent on each fiber, extended constantly to $U$. So $\partial_{q_j}\mu=0$ for $1\leq j\leq R$, and hence \eqref{eqb80a} implies $\wtd \yk |_U=h(q_\blt,\eta,\tau_\blt)\cdot \partial_\eta \mu\cdot \partial_\mu +q_k\partial_{q_k}$. Note that $\partial_\eta\mu$ is constant over $q_\blt$. So if we define $\wtd\yk_{n_\blt}^\perp$ using $\mu$, then $\wtd \yk_{n_\blt}^\perp=h_{n_\blt}(\eta,\tau_\blt)\cdot \partial_\eta\mu\cdot \partial_\mu$, which coincides with \eqref{eqb80c}.\\[-1ex]

Step 2. In this step, we calculate $\wtd\yk_{n_\blt}^\perp$ on each $V_j',V_j''$. Our result \eqref{eqb77} will imply
\begin{align}\label{eqb78}
\wtd\yk_{n_\blt}^\perp\in H^0\big(\wtd \MC, \Theta_{\wtd\MC/\wtd\MB}(\blt S_{\wtd \fx})\big)
\end{align}
i.e., that $\wtd\yk_{n_\blt}^\perp$ has finite poles at $\Gamma$. We use the notations in Asmp. \ref{lbb8}. By the description of $d\pi$ in Subsec. \ref{lbb46}, for each $1\leq i\leq N$ and $1\leq j\leq R$, we can write
    \begin{subequations}\label{eqb76}
    \begin{gather}
        \wtd\yk|_{U_i}=h^i(q_\blt,\eta_i,\tau_\blt)\partial_{\eta_i}+q_k\partial_{q_k}\\
        \wtd \yk|_{W_j}=a^j(\xi_j,\varpi_j,q_{\blt\backslash j},\tau_\blt)\xi_j \partial_{\xi_j}+b^j(\xi_j,\varpi_j,q_{\blt\backslash j},\tau_\blt)\varpi_j\partial_{\varpi_j}+(1-\delta_{j,k})q_k\partial_{q_k}\label{eqb76b}
    \end{gather}      
    \end{subequations}
where $h^i\in\mc O((q_\blt,\eta_i,\tau_\blt)(U-\SX))$ has finite poles at $(q_\blt,\eta,\tau_\blt)(\SX)$, and $a^j,b^j\in \MO(W_j)$ (where $W_j=\eqref{eqWi1}$ has standard coordinates $(\xi_j,\varpi_j,q_{\blt\backslash j},\tau_\blt)$) satisfy
\begin{align}
a^j+b^j=\delta_{j,k}
\end{align}
The tangent vectors $\partial_{\xi_j},\partial_{\varpi_j},\partial_{q_k}$ are defined using the coordinates $(\xi_j,\varpi_j,q_{\blt\backslash j},\tau_\blt)$. 

When restricted to $W_j'=\eqref{eqWi2}$ resp. $W_j''=\eqref{eqWi3}$ and using the coordinates $(\xi_j,q_j/\xi_j,q_{\blt\backslash j},\tau_\blt)$ resp. $(q_j/\varpi_j,\varpi_j,q_{\blt\backslash j},\tau_\blt)$ to define the tangent vectors $\partial_{\xi_j},\partial_{q_k}$ resp. $\partial_{\varpi_j},\partial_{q_k}$, Eq. \eqref{eqb76b} becomes
   \begin{subequations}\label{multi3}
   \begin{gather}
       \wtd \yk|_{W_j'}=a^j(\xi_j,q_j/\xi_j,q_{\blt\backslash j},\tau_\blt)\xi_j \partial_{\xi_j}+q_k\partial_{q_k}\\
        \wtd \yk|_{W_j''}=b^j(q_j/\varpi_j,\varpi_j,q_{\blt\backslash j},\tau_\blt)\varpi_j \partial_{\varpi_j}+q_k\partial_{q_k}
   \end{gather}
   \end{subequations}
which are calculated from the change of coordinate formula for tangent vectors
\begin{gather*}
\left\{
\begin{array}{l}
\partial_{\xi_j}=\partial_{\xi_j}+\varpi_j\partial_{q_j}\\
\partial_{\varpi_j}=\xi_j\partial_{q_j}
\end{array}
\right.
\text{ on }W_j'\qquad
\left\{
\begin{array}{l}
\partial_{\xi_j}=\varpi_j\partial_{q_j}\\
\partial_{\varpi_j}=\partial_{\varpi_j}+\xi_j\partial_{q_j}
\end{array}
\right.
\text{ on }W_j''
\end{gather*}
due to $q_j=\xi_j\varpi_j$. Write
  \begin{gather*}
      a^j(\xi_j,\varpi_j,q_{\blt\backslash j},\tau_\blt)=\sum_{m,n\in \Nbb}a^j_{m,n}(q_{\blt\backslash j},\tau_\blt)\xi_j^m \varpi_j^n\\
      b^j(\xi_j,\varpi_j,q_{\blt\backslash j},\tau_\blt)=\sum_{m,n\in \Nbb}b^j_{m,n}(q_{\blt\backslash j},\tau_\blt)\xi_j^m \varpi_j^n
  \end{gather*}
Then the RHS below converge a.l.u. on $W_j$ to the LHS:
\begin{subequations}\label{multi4}
\begin{gather}
      a^j(\xi_j,q_j/\xi_j,q_{\blt\backslash j},\tau_\blt)=\sum_{n\geq 0,l\geq -n}a^j_{l+n,n}(q_{\blt\backslash j},\tau_\blt)\xi_j^l q_j^n\\
      b^j(q_j/\varpi_j,\varpi_j,q_{\blt\backslash j},\tau_\blt)=\sum_{m\geq 0,l\geq -m}b^j_{m,l+m}(q_{\blt\backslash j},\tau_\blt)\varpi_j^l q_j^m
  \end{gather}
\end{subequations}
By substituting \eqref{multi4} into \eqref{multi3}, we obtain
\begin{gather*}
    \label{multi6}\wtd \yk|_{W_j'}=\sum_{n\geq 0,l\geq -n}a^j_{l+n,n}(q_{\blt\backslash j},\tau_\blt)\xi_j^{l+1}q_j^n\partial_{\xi_j}  +q_k\partial_{q_k}\\
        \label{multi7}\wtd \yk|_{W_j''}=\sum_{m\geq 0,l\geq -m}b^j_{m,l+m}(q_{\blt\backslash j},\tau_\blt)\varpi_j^{l+1}q_j^m\partial_{\varpi_j}  +q_k\partial_{q_k}
\end{gather*}
Write
\begin{subequations}\label{jthexpansion}
\begin{gather}
    a^j_{l+n_j,n_j}(q_{\blt\backslash j},\tau_\blt)=\sum_{n_{\blt\backslash j}\in\Nbb^{R-1}}a^j_{l+n_j,n_j,n_{\blt\backslash j}}(\tau_\blt) q_{\blt\backslash j}^{n_{\blt\backslash j}}\\
    b^j_{n_j,l+n_j}(q_{\blt\backslash j},\tau_\blt)=\sum_{n_{\blt\backslash j}\in\Nbb^{R-1}}b^j_{n_j,l+n_j,n_{\blt\backslash j}}(\tau_\blt)q_{\blt\backslash j}^{n_{\blt\backslash j}}
\end{gather}
\end{subequations}
Then in $H^0(V_j'-S_{\wtd\fx},\Theta_{\wtd \MC/\wtd \MB})$ resp. $H^0(V_j''-S_{\wtd\fx},\Theta_{\wtd \MC/\wtd \MB})$ we have
\begin{subequations}\label{eqb77}
\begin{gather}
\wtd\yk_{n_\blt}^\perp\big|_{V_j'}=\sum_{l\geq -n_j}a^j_{l+n_j,n_j,n_{\blt\setminus j}}(\tau_\blt)\xi_j^{l+1}\partial_{\xi_j} \\
\wtd\yk_{n_\blt}^\perp\big|_{V_j''}=\sum_{l\geq -n_j}b^j_{n_j,l+n_j,n_{\blt\setminus j}}(\tau_\blt)\varpi_j^{l+1}\partial_{\varpi_j}
\end{gather}
\end{subequations}
where $\partial_{\xi_j}$ resp. $\partial_{\varpi_j}$ are with respect to the coordinates $(\xi_j,\tau_\blt)$ resp. $(\varpi_j,\tau_\blt)$ of $V_j'$ resp. $V_j''$ (cf. \eqref{eqb4}). So $\wtd\yk_{n_\blt}^\perp$ has poles of orders at most $n_j-1$ at $\sgm_j'(\wtd\MB)$ and $\sgm_j''(\wtd\MB)$.\\[-1ex]

Step 3. By \eqref{eqb78}, $\wtd\yk_{n_\blt}^\perp$ is a lift of the zero tangent field of $\wtd\MB$. Therefore
\begin{gather*}
\upnu(\wtd \yk_{n_\blt}^\perp)\in H^0\big(\wtd U_1\cup \cdots\cup \wtd U_N\cup V_1'\cup V_1'' \cup \cdots \cup V_R'\cup V_R'',  \scr V_{\wtd\fx}\otimes\omega_{\wtd\MC/\wtd\MB}(\blt S_{\wtd\fx})\big)\\
\upnu(\wtd\yk)\in H^0\big(U_1\cup\cdots\cup U_N,\scr V_\fx\otimes\omega_{\MC/\MB}(\blt\SX)\big)
\end{gather*}
can be defined by \eqref{eqb32}. The goal of this step is to prove \eqref{multi9}.

Note \eqref{eqb80} for the relation between $\wtd\yk$ and $\wtd\yk_{n_\blt}^\perp$. Choose any $w\in \Wbb\subset\Wbb\otimes \MO(\MB)$. By Def. \ref{lbb16}, in $\Wbb\otimes(\MO(\wtd\MB)[[q_\blt]])$ we have
\begin{align*}
\nabla_{q_k\partial_{q_k}}w=-\upnu(\wtd\yk)w=-\sum_{n_\blt\in\Nbb^R}\upnu(\wtd \yk_{n_\blt}^\perp)w\cdot q_\blt^{n_\blt}
\end{align*}
Therefore, by Thm. \ref{lift2}, in $(\Wbb\otimes\Mbb\otimes\Mbb'\otimes\MO(\wtd\MB))\{q_\blt\}[\log q_\blt]$ we have
\begin{align*}
    &\upnu(\wtd \yk_{n_\blt}^\perp)w\otimes q_\blt^{L_\blt(0)}\btr\otimes \btl+w\otimes \upnu(\wtd \yk_{n_\blt}^\perp) q_\blt^{L_\blt(0)}\btr\otimes \btl+w\otimes q_\blt^{L_\blt(0)}\btr\otimes  \upnu(\wtd \yk_{n_\blt}^\perp) \btl\\
=&\#(\wtd \yk_{n_\blt}^\perp)\cdot w\otimes q_\blt^{L_\blt(0)}\btr\otimes \btl+ \text{ an element of }\SJ^\pre_{\wtd\fx}(\wtd\MB)\{q_\blt\}[\log q_\blt]
\end{align*}
Note that $\#(\wtd \yk_{n_\blt}^\perp)\in\mc O(\wtd\MB)$. Since $\uppsi$ vanishes on $\SJ^\pre_{\wtd\fx}(\wtd\MB)$, in $\mc O(\wtd\MB)\{q_\blt\}[\log q_\blt]$ we have
\begin{equation}\label{multi9}
\begin{aligned}
    &\sum_{n_\blt\in \Nbb^R}q_\blt^{n_\blt}\uppsi\big(w\otimes \upnu(\wtd \yk_{n_\blt}^\perp) q_\blt^{L_\blt(0)}\btr\otimes \btl+w\otimes q_\blt^{L_\blt(0)}\btr\otimes  \upnu(\wtd \yk_{n_\blt}^\perp) \btl\big)\\
    =&\MS \uppsi (\nabla_{q_k\partial_{q_k}}w)+\sum_{n_\blt\in \Nbb^R}\#(\wtd \yk_{n_\blt}^\perp) q_\blt^{n_\blt}\cdot \MS\uppsi(w)
\end{aligned}
\end{equation}

Step 4. Let us prove in $(\Mbb\otimes\Mbb'\otimes\MO(\wtd\MB))\{q_\blt\}[\log q_\blt]$ that 
\begin{align}\label{multi8}
    \sum_{n_\blt\in \Nbb^R}q_\blt^{n_\blt}\big( \upnu(\wtd \yk_{n_\blt}^\perp) q_\blt^{L_\blt(0)}\btr\otimes \btl+ q_\blt^{L_\blt(0)}\btr\otimes  \upnu(\wtd \yk_{n_\blt}^\perp) \btl\big)=L_k(0)q_\blt^{L_\blt(0)}\btr\otimes \btl
\end{align}
This will imply that for each $w\in\Wbb$, we have in $\mc O(\wtd\MB)\{q_\blt\}[\log q_\blt]$ that
\begin{equation}\label{eqb82}
\begin{aligned}
   & q_k\partial_{q_k}\MS \uppsi(w)=\uppsi(w\otimes q_k \partial_{q_k}q_\blt^{L_\blt(0)}\btr\otimes \btl)=\uppsi(w\otimes L_k(0)q_\blt^{L_\blt(0)}\btr\otimes \btl)\\
    \xlongequal{\eqref{multi8}}&\sum_{n_\blt\in \Nbb^R}q_\blt^{n_\blt}\uppsi\big(w\otimes \upnu(\wtd \yk_{n_\blt}^\perp) q_\blt^{L_\blt(0)}\btr\otimes \btl+w\otimes q_\blt^{L_\blt(0)}\btr\otimes  \upnu(\wtd \yk_{n_\blt}^\perp) \btl\big)\\
    \xlongequal{\eqref{multi9}}&\MS \uppsi (\nabla_{q_k\partial_{q_k}}w)+\sum_{n_\blt\in \Nbb^R}\#(\wtd \yk_{n_\blt}^\perp) q_\blt^{n_\blt}\cdot \MS\uppsi(w)
\end{aligned}
\end{equation}
To prove \eqref{multi8}, note that by \eqref{eqb77}, the $j$-th residue actions $*_j$ of $\upnu(\wtd \yk_{n_\blt}^\perp)$ on $\Mbb,\Mbb'$ are 
\begin{align*}
    &\Res_{\xi_j=0}\sum_{l\geq -n_j}a^j_{l+n_j,n_j,n_{\blt\backslash j}}(\tau_\blt)Y_{\Mbb,j}(\cbf,\xi_j)\xi_j^{l+1}d\xi_j \\
    &\Res_{\varpi_j=0}\sum_{l\geq -n_j}b^j_{n_j,l+n_j,n_{\blt\backslash j}}(\tau_\blt)Y_{\Mbb',j}(\cbf,\varpi_j)\varpi_j^{l+1}d\varpi_j
\end{align*}
These formulas, together with \eqref{jthexpansion}, show that the LHS of \eqref{multi8} equals
\begin{align*}
    &\sum_{n_j\geq 0}\sum_{j=1}^R \Res_{\xi_j=0}\sum_{l\geq -n_j}a^j_{l+n_j,n_j}(q_{\blt\backslash j},\tau_\blt)q_j^{n_j}Y_{\Mbb,j}(\cbf,\xi_j)q_\blt^{L_\blt(0)}\btr\otimes \btl \xi_j^{l+1}d\xi_j\\
    &+\sum_{n_j\geq 0}\sum_{j=1}^R \Res_{\varpi_j=0}\sum_{l\geq -n_j}b^j_{n_j,l+n_j}(q_{\blt\backslash j},\tau_\blt)q_j^{n_j}q_\blt^{L_\blt(0)}\btr\otimes Y_{\Mbb',j}(\cbf,\varpi_j)\btl \varpi_j^{l+1}d\varpi_j\\
\xlongequal{\eqref{multi4}}&\sum_{j=1}^R \Res_{\xi_j=0}~a^j(\xi_j,q_j/\xi_j,q_{\blt\backslash j},\tau_\blt)Y_{\Mbb,j}(\cbf,\xi_j)q_\blt^{L_\blt(0)}\btr\otimes \btl \xi_jd\xi_j\\
    &+\sum_{j=1}^R \Res_{\varpi_j=0}~b^j(q_j/\varpi_j,\varpi_j,q_{\blt\backslash j},\tau_\blt)q_\blt^{L_\blt(0)}\btr\otimes Y_{\Mbb',j}(\cbf,\varpi_j)\btl \varpi_j d\varpi_j
\end{align*}
This result, together with $a^j+b^j=\delta_{j,k}$ and the formula
\begin{align*}
        &\Res_{\xi_j=0}~Y_{\Mbb,j}(\xi_j^{2}\cbf,\xi_j)q_\blt^{L_\blt(0)}\btr\otimes \btl \cdot b^j(\xi_j,q_j/\xi_j,q_{\blt\backslash j},\tau_\blt)\frac{d\xi_j}{\xi_j}\\
        =&\Res_{\varpi_j=0}~q_\blt^{L_\blt(0)}\btr\otimes Y_{\Mbb',j}(\varpi_j^{2}\cbf,\varpi_j) \btl\cdot b^j(q_j/\varpi_j,\varpi_j,q_{\blt\backslash j},\tau_\blt)\frac{d\varpi_j}{\varpi_j}
    \end{align*}
(due to Prop. \ref{multisew7} and the fact that $\MU(\upgamma_1)\cbf=\cbf$), implies \eqref{multi8}.\\[-1ex]

Step 5. In this step, we prove that
\begin{align}
    g_k:=\sum_{n_\blt \in \Nbb^R}\#(\wtd \yk_{n_\blt}^\perp)q_\blt^{n_\blt}\in \MO(\wtd \MB)[[q_\blt]]
\end{align}
belongs to $\MO(\MB)$. Recall \eqref{eqb4} that $V_i'=\MD_{r_i}\times\wtd\MB$ and $V_i''=\MD_{\rho_i}\times\wtd\MB$. Define $h^i_{n_\blt}$ from $h^i$ as in \eqref{eqb80}. By Thm. \ref{lift2} and \eqref{eqb77}, we have
\begin{align*}
\#(\wtd\yk_{n_\blt}^\perp)=\frac c{12}\Big(\sum_{1\leq j\leq R} A_{j,n_\blt}+\sum_{1\leq j\leq R} B_{j,n_\blt}+\sum_{1\leq i\leq N}C_{i,n_\blt}\Big)
\end{align*}
where
\begin{gather*}
A_{j,n_\blt}=\sum_{l\geq-n_j}\Res_{\xi_j=0}~ \Sbf_{\xi_j}\mbf P(\xi_j,\tau_\blt)\cdot a^j_{l+n_j,n_j,n_{\blt\setminus j}}(\tau_\blt) \xi_j^{l+1} d\xi_j\\
B_{j,n_\blt}=\sum_{l\geq-n_j}\Res_{\varpi_j=0}~ \Sbf_{\varpi_j}\mbf P(\varpi_j,\tau_\blt)\cdot b^j_{n_j,l+n_j,n_{\blt\setminus j}}(\tau_\blt)\varpi_j^{l+1} d\varpi_j\\
C_{i,n_\blt}=\Res_{\eta_i=0}~ \Sbf_{\eta_i} \mbf P(\eta_i,\tau_\blt)\cdot h^i_{n_\blt}(\eta_i,\tau_\blt)d\eta_i
\end{gather*}

Choose anticlockwise circles $\gamma_j'$ inside $\MD_{r_j}$ and $\gamma_j''$ inside $\MD_{\rho_j}$ around the origins with radii $0<\epsilon_j'<r_j$ resp. $0<\epsilon_j''<\rho_j$. Let
\begin{subequations}\label{eqb81}
\begin{gather}
    \label{eqb81a}A_j=\frac{1}{2\im\pi}\oint\nolimits_{\gamma_j'} \Sbf_{\xi_j}\mbf P(\xi_j,\tau_\blt)\cdot a^j(\xi_j,q_j/\xi_j,q_{\blt\backslash j},\tau_\blt) \cdot \xi_j d\xi_j\\
    \label{eqb81b}B_j=\frac{1}{2\im\pi}\oint\nolimits_{\gamma_j''} \Sbf_{\varpi_j}\mbf P(\varpi_j,\tau_\blt)\cdot b^j(q_j/\varpi_j,\varpi_j,q_{\blt\backslash j},\tau_\blt)\cdot \varpi_j d\varpi_j\\
    \label{eqb81c}C_i=\Res_{\eta_i=0} \Sbf_{\eta_i} \mbf P(\eta_i,\tau_\blt)\cdot h^i(q_\blt,\eta_i,\tau_\blt)d\eta_i
\end{gather} 
\end{subequations}
where in \eqref{eqb81a} and \eqref{eqb81b} we assume that $|q_j|/\rho_j<\epsilon_j'$ resp. $|q_j|/r_j<\epsilon_j''$ (cf. \eqref{eqb84}). Clearly $C_i\in\MO(\MB)$. It is not hard to see that the definitions of $A_j,B_j$ are consistent for different choices of $\epsilon_j',\epsilon_j''$. Therefore, by choosing $\epsilon_j',\epsilon_j''$ sufficiently close to $r_j,\rho_j$ respectively, Eq. \eqref{eqb81a} and \eqref{eqb81b} define $A_j,B_j\in\MO(\MB)$. It is also not hard to see that $\sum_{n_\blt} A_{j,n_\blt}q_\blt^{n_\blt}$, $\sum_{n_\blt} B_{j,n_\blt}q_\blt^{n_\blt}$, $\sum_{n_\blt} C_{j,n_\blt}q_\blt^{n_\blt}$ converge a.l.u. on $\MB$ to $A_j$, $B_j$, $C_j$ respectively. (See the proof of \cite[Prop. 11.12]{Gui-sewingconvergence}.) Note that the factor $\frac 1{2\im\pi}$ for \eqref{eqb81a} and \eqref{eqb81b} is missing in Prop. 11.12 and also in (13.8) of \cite{Gui-sewingconvergence}.) Thus $g_k$ converges a.l.u. to an element of $\MO(\MB)$, i.e., $\frac{c}{12}\Big(\sum_{1\leq j\leq R} A_j+\sum_{1\leq j\leq R} B_j+\sum_{1\leq i\leq N}C_i\Big)$.\\[-1ex]

Step 6. This final step is similar to the proof of Thm. \ref{lbb26}. By Thm. \ref{lbb21}, we can find finitely many elements $s_1,s_2,\dots\in\Wbb\otimes\MO(\MB)$ generating $\Wbb\otimes\MO(\MB)$ mod $\SJ^\pre_\fx(\MB)$. Fix $\mu_\blt\in\Cbb^N$ such that $s_1,s_2,\dots\in\Wbb_{[\leq\mu_\blt]}\otimes\MO(\MB)$. Thus $\Wbb_{[\leq\mu_\blt]}$ generates the $\MO(\MB)$-module $\Wbb\otimes\MO(\MB)$ mod $\SJ^\pre_\fx(\MB)$.

Choose any $\lambda_\blt\geq\mu_\blt$. Let $(e_j)_{j\in J}$ be a basis of the (finite-dimensional) vector space $\Wbb_{[\leq\lambda_\blt]}$. By \eqref{eqb82} in Step 4, for each $i\in J$ and $1\leq k\leq R$ we have
\begin{align*}
q_k\partial_q \MS\uppsi(e_i)=\MS\uppsi(\nabla_{q_k\partial_{q_k}}e_i)+g_k\cdot\MS\uppsi(e_i)
\end{align*}
where $g_k\in\MO(\MB)$ by Step 5. Since $\Wbb_{[\leq\lambda_\blt]}$ generates the $\MO(\MB)$-module $\Wbb\otimes\MO(\MB)$ mod $\SJ^\pre_\fx(\MB)$, and since $\nabla_{q_k\partial_{q_k}}e_i\in \Wbb\otimes\MO(\MB)$ (cf. Def. \ref{lbb16}), we can find $\Omega^k_{i,j}\in\MO(\MB)$ such that
\begin{align*}
\nabla_{q_k\partial_{q_k}}e_i=\sum_{j\in J}\Omega_{i,j}^ke_j \mod\ \SJ^\pre_\fx(\MB)
\end{align*}
for all $i,j\in J$. Therefore, by Prop. \ref{lbb48}, 
\begin{align*}
q_k\partial_{q_k}\MS\uppsi(e_i)=\sum_{j\in J}\Omega_{i,j}^k\MS\uppsi(e_j)+g_k\cdot \MS\uppsi(e_i)
\end{align*}
Thus, as an element of $\MO(\wtd\MB)[[q_\blt]][\log q_\blt]^J$, $f:=\oplus_{j\in J}\MS\uppsi(e_j)$ is a formal solution of the differential equation of the system of differential equations $q_k\partial_{q_k}f=\Lambda^kf$ (for all $1\leq k\leq R$) where $\Lambda^k$ is the $\Cbb^{J\times J}$-valued holomorphic functions on $\MB$ whose $i\times j$-th entry is $\Omega_{i,j}^k+\delta_{i,j}g_k$. Therefore, by \cite[Thm. A.1]{Gui-sewingconvergence}, $\MS\uppsi(w)$ converges a.l.u.  whenever $w=e_i$, and hence whenever $w\in\Wbb_{[\leq\lambda_\blt]}$. Since $\lambda_\blt\geq \mu_\blt$ is arbitrary, $\MS\uppsi(w)$ converges a.l.u. for all $w\in\Wbb$. 
\end{proof}

\begin{rem}\label{lbb56}
\textit{Thm. \ref{lbb49} holds if \eqref{eqb52} is replaced by the following weaker assumption: Each connected component of each fiber $\wtd\fx_b$ of $\wtd\fx$ (where $b\in\wtd\MB$) intersects  $\sgm_\blt(\wtd\MB)\cup\sgm'_\blt(\wtd\MB)\cup\sgm''_\blt(\wtd\MB)$, and each connected component of each smooth fiber $\fx_b$ of $\fx$ (where $b\in\MB-\Delta$) intersects $\sgm_\blt(\MB)$.}
\end{rem}

The primary purpose of this assumption is that, when defining sheaves of conformal blocks for smooth families, we always require each connected component of every fiber to intersect the marked points.

\begin{proof}
It suffices to prove that for each $b_0\in\wtd\MB$ and $0<\wtd r_i<r_i,0<\wtd\rho_i<\rho_i$, there is a neighborhood $W\subset\wtd\MB$ of $b_0$ such that Thm. \ref{lbb49} holds when restricted to the open subset $W\times \MD_{\wtd r_\blt\wtd\rho_\blt}$ of $\MB$. Therefore, by shrinking $r_i,\rho_i$ to each $\wtd r_i,\wtd \rho_i$, and by shrinking $\wtd\MB$ to a small enough neighborhood of $b_0$, we assume that some incoming marked points (i.e. sections) $\sigma_1,\dots,\sigma_K$ of $\wtd\fx$ can be added such that $\sigma_1(\wtd\MB),\dots,\sigma_K(\wtd\MB)$ are mutually disjoint and are also disjoint from each $\sgm_i(\wtd\MB)$ and $V_j'=\MD_{r_i}\times\wtd\MB$ and $V_j''=\MD_{\rho_i}\times\wtd\MB$, and each component of each $\wtd\fx_b$ intersects one of $\sigma_\varstar(\wtd\MB)$. 

Let $\wtd\fx'=(\wtd \pi:\wtd \MC\rightarrow \wtd \MB\big|\sgm_\blt,\sigma_\varstar\big\Vert \sgm_\blt',\sgm_\blt'')$ be the new family obtained from $\wtd\fx$ by adding $\sigma_\varstar$. Associate the vacuum module $\Vbb$ to each $\sigma_k$. By the propagation of conformal blocks (cf. \cite[Sec. 2.5]{GZ1}) which is applicable by our weaker assumption, there exists an element
\begin{gather*}
\wr\uppsi\in H^0(\wtd\MB,\scr T^*_{\wtd\fx'}(\Wbb\otimes\Vbb^{\otimes K}\otimes\Mbb\otimes\Mbb'))\\
\wr\uppsi(w\otimes\idt^{\otimes K}\otimes m\otimes m')=\uppsi(w\otimes m\otimes m')
\end{gather*}
for all $w\in\Wbb,m\in\Mbb,m'\in\Mbb'$. Then $\wtd\fx'$ satisfies \eqref{eqb52}. We sew $\wtd\fx'$ and get $\fx'$. By  Thm. \ref{lbb49}, the sewing $\MS{\wr\uppsi}$ converges a.l.u. to a conformal block associated to $\fx'$ outside the discriminant locus $\Delta$. Thus
\begin{align*}
\MS\uppsi(w)=\MS{\wr\uppsi}(w\otimes\idt^{\otimes K})
\end{align*}
converges a.l.u.; the limit is a conformal block associated to $\fx$ outside $\Delta$, again by the propagation of conformal blocks \footnote{Namely, if a module for the marked point is $\Vbb$, then inserting the vacuum vector $\idt$ to this marked point produces a conformal block associated to the Riemann surfaces with that marked point removed. This easy property does not require the condition that each connected component intersects the marked points. Therefore, we do not assume that each component of each fiber of $\fx|_{\MB-\Delta}$ intersects $\sgm_\blt(\MB)$.}. This proves Thm. \ref{lbb49}.
\end{proof}

The following Thm. \ref{lbb52} says roughly that the sewing of a conformal block is parallel ``in the direction of sewing" up to a projective term that can be calculated explicitly. In a future paper, we will use Thm. \ref{lbb52} to explain the appearance of $-\frac c{24}$ in the modular invariance formulas in \cite{Zhu-modular-invariance,Miy-modular-invariance,Hua-differential-genus-1,Hua-modular-C2}.

\begin{thm}\label{lbb52}
Let $\uppsi\in H^0\big(\wtd \MB,\ST_{\wtd \fx}^*(\Wbb\otimes \Mbb\otimes \Mbb')\big)$. Assume that $\wtd\MB$ is an open subset of $\Cbb^m$ with standard coordinates $\tau_\blt$. Fix $1\leq k\leq R$, assume that $\wtd\yk\in H^0(\MC,\Theta_\MC(-\log\MC_\Delta+\blt\SX))$ is a lift of $q_k\partial_{q_k}$ (cf. Def. \ref{lbb50}). We use $\wtd\yk$ to define $\nabla_{q_k\partial_{q_k}}$ on $\scr T_\fx^*(\Wbb)$ being the dual of   $\nabla_{q_k\partial_{q_k}}$ on $\scr T_\fx(\Wbb)$ defined by Def. \ref{lbb16}. Assume that $\mbf P$ is a projective structure on $\wtd\fx$. For each $1\leq j\leq R$ and $1\leq i\leq N$, let  $a^j,b^j,h^j$ be defined by \eqref{eqb76}, and let $A_j,B_j,C_i$ be defined by \eqref{eqb81}. Let
\begin{align*}
    g_k=\frac{c}{12}\Big(\sum_{1\leq j\leq R} A_j+\sum_{1\leq j\leq R} B_j+\sum_{1\leq i\leq N}C_i\Big)
\end{align*}
(which is in $\MO(\MB)$). Then $\MS\uppsi\in H^0\big(\wht\MB^\times,\ST_\fx^*(\Wbb)\big)$ satisfies
\begin{align*}
\nabla_{q_k\partial_{q_k}}\MS\uppsi=g_k\cdot \MS\uppsi
\end{align*}
\end{thm}

We repeat that $\wtd\yk$ exists due to Rem. \ref{lbb12}, and that $\mbf P$ exists due to Thm. \ref{lbb47}. Recall Def. \ref{lbb18} for the meaning of dual differential operators.

\begin{proof}
This is clear from the proof of Thm. \ref{lbb49}, especially Eq. \eqref{eqb82} and Step 5.
\end{proof}

\subsection{Application: convergence of higher genus pseudo-$q$-traces}\label{lbb63}

In this section, we assume for simplicity that $R=1$. We write $\sgm_1',\sgm_1''$ as $\sgm',\sgm''$. Write $r_1=r,\rho_1=\rho$. Associate a grading restricted $\Vbb^{\otimes N}$-module $\Wbb$ to $\sgm_\blt(\wtd\MB)$. Thus $\MB=\MD_{r\rho}\times\wtd\MB$. Choose a grading restricted $\Vbb$-module $\Mbb$ with contragredient $\Mbb'$. Associate $\Mbb'$ and $\Mbb$ to $\sgm'(\wtd\MB)$ and $\sgm''(\wtd\MB)$ respectively. $\Wbb$ is also associated to the marked points $\sgm_\blt(\MB)$ of the sewn family $\fx$. Recall from the beginning of this chapter that the local coordinates $\eta_\blt,\xi_1=\xi,\varpi_1=\varpi$ (at $\sgm_\blt(\wtd\MB),\sgm'(\wtd\MB),\sgm''(\wtd\MB)$) are chosen, and the identifications \eqref{eqb85} are assumed.

\subsubsection{The pseudo-trace $\Tr^\omega$ on $\End^0_A(\Mbb)$}

Note that $\dim\End_\Vbb(\Mbb)<+\infty$. We fix a unital $\Cbb$-subalgebra $A$ of $\End_\Vbb(\Mbb)^\mathrm{op}$ such that $\Mbb$ is projective as a right $A$-module. We fix a symmetric linear functional (\textbf{SLF}) $\omega$ on $A$. Namely, 
\begin{align*}
\omega:A\rightarrow\Cbb
\end{align*}
is a linear map satisfying $\omega(ab)=\omega(ba)$ for all $a,b\in A$. 

As usual, $\End_A(\Mbb)$ denotes the set of linear maps $T:\Mbb\rightarrow\Mbb$ satisfying $(Tm)a=T(m a)$ for all $m\in\Mbb,a\in A$. We let
\begin{gather*}
\End^0(\Mbb):=\bigcup_{\lambda\in\Cbb} \End(\Mbb_{[\leq\lambda]})=\{T\in\End(\Mbb):T=P_{\leq\lambda}TP_{\leq\lambda}\text{ for some }\lambda\in\Cbb\}\\
\End^0_A(\Mbb):=\End^0(\Mbb)\cap\End_A(\Mbb)=\bigcup_{\lambda\in\Cbb}\End_A(\Mbb_{[\leq\lambda]})
\end{gather*}
Since each $\Mbb_{[\leq\lambda]}$ is finite-dimensional, by the pseudo-trace (also called Hattori-Stallings trace) construction (cf. \cite{Ari10} and the references therein), $\omega$ defines canonically an SLF $\Tr^\omega$ on $\End_A(\Mbb_{[\leq\lambda]})$, restricting to the one on $\End_A(\Mbb_{[\leq\mu]})$ whenever $\Re\mu\leq\Re\lambda$. We thus obtain an SLF
\begin{align*}
\Tr^\omega:\End_A^0(\Mbb)\rightarrow\Cbb
\end{align*}
called the \textbf{pseudo-trace of \pmb{$\omega$} on \pmb{$\End_A^0(\Mbb)$}}. (In fact, the results in this section hold more generally when $\Tr^\omega$ is replaced by any SLF on $\End_A^0(\Mbb)$. Thus, one can forget about the fact that it arises from $\omega$.) Recall that $\ovl{\Mbb}$ is the algebraic completion of $\Mbb$.

\subsubsection{The pseudo-sewing $\MS^\omega\upphi$ of a conformal block $\upphi$}

\begin{df}
Let $\upphi\in H^0(\wtd\MB,\scr T_{\wtd\fx}^*(\Wbb\otimes\Mbb'\otimes\Mbb))$. So $\upphi$ can be viewed as a linear map $\Wbb\otimes\Mbb'\otimes\Mbb\rightarrow \MO(\wtd\MB)$. Define a map \pmb{$\upphi^\sharp$} by
\begin{gather*}
\upphi^\sharp: \wtd\MB\rightarrow \Hom_\Cbb(\Wbb,\Hom_\Cbb(\Mbb,\ovl\Mbb))
\end{gather*}
such that for each $b\in\wtd\MB$ and $w\in\Wbb$, the linear map $\upphi^\sharp_b(w)\in\Hom(\Mbb,\ovl\Mbb)$ satisfies
\begin{gather}\label{eqb88}
\upphi^\sharp_b(w):\Mbb\rightarrow\ovl\Mbb\qquad \bk{\upphi^\sharp_b(w)m,m'}=\upphi(w\otimes m'\otimes m)\big|_b
\end{gather}
for all $m\in\Mbb,m'\in\Mbb'$. We say that $\upphi$ \textbf{commutes with \pmb{$A$}} if for each $b\in\wtd\MB,w\in\Wbb$ and $\lambda,\mu\in\Cbb$, the following linear map 
\begin{align}\label{eqb86}
P_\lambda\circ \upphi^\sharp_b(w)\circ P_\mu:\Mbb_{[\mu]}\rightarrow\Mbb_{[\lambda]}
\end{align}
commutes with the action of $A$ (and hence belongs to $\End^0_A(\Mbb)$).
\end{df}

In the following, we fix $\upphi\in H^0(\wtd\MB,\scr T_{\wtd\fx}^*(\Wbb\otimes\Mbb'\otimes\Mbb))$ commuting with $A$. $\Tr^\omega$ can be defined on \eqref{eqb86}. Recall that $L(0)_{\mathrm n}$ preserves each $\Mbb_{[\lambda]}$ (because it commutes with $L(0)$). Note that $q^{L(0)}$ and $P_\lambda$ commute with $A$. Thus, for each $b\in\wtd\MB,w\in\Wbb$, we can define
\begin{align}\label{eqb90}
\begin{aligned}
\MS^\omega\upphi(w)\big|_b:=&\sum_{\lambda,\mu\in\Cbb} \Tr^\omega\big(q^{L(0)}P_\lambda\circ \upphi^\sharp_b(w)\circ P_\mu\big)\\
=&\sum_{\lambda\in\Cbb} q^\lambda\cdot\Tr^\omega \big(q^{L(0)_{\mathrm n}}P_\lambda\circ \upphi^\sharp_b(w)\circ P_\lambda\big)
\end{aligned}
\end{align}
which is in $\Cbb\{q\}[\log q]$. This gives a linear map
\begin{align*}
\MS^\omega\upphi:\Wbb\rightarrow\MO(\wtd\MB)\{q\}[\log q]\qquad w\mapsto\Tr^\omega\upphi^\sharp(w)
\end{align*}
%and hence can be extended uniquely to an $\mc O(\wtd\MB)[[q]]$-module morphism
%\begin{align*}
%\Tr_q^\omega\upphi:\Wbb\otimes\MO(\wtd\MB)[[q]]\rightarrow\MO(\wtd\MB)\{q\}[\log q]
%\end{align*}
We call $\MS^\omega\upphi$ the \textbf{pseudo-sewing} (or the pseudo-$q$-trace) of $\upphi$ with respect to $\omega$. 

The \textbf{a.l.u. convergence} of $\MS^\omega\upphi$ is understood as in Def. \ref{lbb51}. In the rest of this section, our goal is to prove the following theorem. Recall $\MB-\Delta=\MD_{r\rho}^\times\times\wtd\MB$. Recall \eqref{eq1}.

\begin{thm}\label{lbb53}
$\MS^\omega\upphi$ converges a.l.u. to a linear map $\Wbb\rightarrow \MO(\wht\MB^\times)$, which belongs to $H^0\big(\wht\MB^\times,\ST_\fx^*(\Wbb)\big)$ (cf. Rem. \ref{lbb27}).
\end{thm}

\subsubsection{$\Tr^\omega$ is a conformal block associated to the $\Vbb^{\otimes2}$-module $\End^0_A(\Mbb)$}

Recall \eqref{eqb87} for the meaning of $\mc U(\upgamma_z)$. Notice the linear isomorphism
\begin{align*}
\Mbb\otimes\Mbb'\xlongrightarrow{\simeq}\End^0(\Mbb) \qquad m\otimes m'\mapsto m\cdot \bk{m',-}
\end{align*}
Pushing forward the $\Vbb^{\otimes2}$-module structure of $\Mbb\otimes\Mbb'$, we see that $\End^0(\Mbb)$ is a grading-restricted $\Vbb^{\otimes 2}$-module whose module structure is determined by the fact that for each $v\in\Vbb,T\in\End^0(\Mbb)$, the following relation holds in $\End^0(\Mbb)[[z^{\pm1}]]$:
\begin{align}\label{eqb92}
Y(v\otimes \idt,z) T=Y_\Mbb(v,z)\circ T\qquad Y(\idt\otimes v,z)T=T\circ Y_\Mbb (\mc U(\upgamma_z)v,z^{-1})
\end{align}
Clearly $\End^0_A(\Mbb)$ is a (grading-restricted) $\Vbb^{\otimes2}$-submodule of $\End^0(\Mbb)$.

\begin{rem}
Define a $2$-pointed sphere with local coordinates
\begin{align*}
\fn=(\Pbb^1\Vert\infty,0;1/z,z)
\end{align*}
where $z$ denotes the standard coordinate of $\Cbb$. Let $\Xbb$ be a grading restricted $\Vbb^{\otimes2}$-module. Let $\uppsi:\Xbb\rightarrow\Cbb$ be a linear map. Then clearly $\uppsi$ belongs to $\scr T^*_{\fn}(\Xbb)$ iff for each $v\in\Vbb$ and $\chi\in\Xbb$, the following relation holds in $\Cbb[[z^{\pm1}]]$:
\begin{align}\label{eqb91}
\uppsi\big( Y(\mc U(\upgamma_z)v\otimes \idt,z^{-1}) \chi\big)=\uppsi\big(Y(\idt\otimes v,z)\chi\big)
\end{align}
This description is independent of whether $\Xbb$ is associated to $(0,\infty)$ or $(\infty,0)$ (since $\mc U(\upgamma_{z^{-1}})\mc U(\upgamma_z)=\id$).
\end{rem}

\begin{pp}\label{lbb57}
Let $\fn\times\wtd\MB$ be the constant extension of $\fn$ by $\wtd\MB$, i.e., it is the family of $2$-pointed spheres with local coordinates
\begin{align*}
\fn\times\wtd\MB=(\Pbb^1\times\wtd\MB\rightarrow\wtd\MB|\infty,0;1/z,z)
\end{align*}
where $\Pbb^1\times\wtd\MB\rightarrow\wtd\MB$ is the projection onto the $\wtd\MB$-component, and $0,\infty$ mean the constant sections $b\mapsto (0,b)$ and $b\mapsto (\infty,b)$. 

Associate $\End_A^0(\Mbb)$ to the marked points $(\infty,0)$ where the first component (i.e. $Y(-\otimes \idt,z)$) is associated to $\infty$ and the second one (i.e. $Y(\idt\otimes-,z)$) is associated to $0$. Then the linear functional $\Tr^\omega:\End^0_A(\Mbb)\rightarrow\Cbb$, viewed as a linear map $\Tr^\omega:\End^0_A(\Mbb)\rightarrow \MO(\wtd\MB)$ by enlarging the codomain, is an element of $H^0\big(\wtd\MB,\scr T^*_{\fn\times\wtd\MB}(\End^0_A(\Mbb))\big)$.
\end{pp}

\begin{proof}
Since the property of being a conformal block can be checked pointwise (cf. Rem. \ref{lbb27}), it suffices to prove that $\Tr^\omega:\End^0_A(\Mbb)\rightarrow\Cbb$ belongs to $\scr T_\fn^*(\End_A^0(\Mbb))$. This, in view of \eqref{eqb92} and \eqref{eqb91}, is equivalent to proving that in $\Cbb[[z^{\pm1}]]$ we have
\begin{align*}
\Tr^\omega (Y_\Mbb(v,z)\circ T)=\Tr^\omega(T\circ Y_\Mbb(v,z))
\end{align*}
for each $v\in\Vbb$ and $T\in\End^0_A(\Mbb)$. Take $\lambda\in\Cbb$ such that $T=P_{\leq\lambda}TP_{\leq\lambda}$. Then the above equation becomes
\begin{align*}
\Tr^\omega (P_{\leq\lambda}Y_\Mbb(v,z)P_{\leq\lambda}\circ T)=\Tr^\omega(T\circ P_{\leq\lambda} Y_\Mbb(v,z)P_{\leq\lambda})
\end{align*}
which holds because $\Tr^\omega$ is an SLF and $P_{\leq\lambda}Y_\Mbb(v,z)P_{\leq\lambda}$ belongs to $\End^0_A(\Mbb)[z^{\pm1}]$.
\end{proof}

\subsubsection{$\upphi$ is a conformal block associated to $\Wbb\otimes_\Cbb\End_A^0(\Mbb)'$}

Let $b\in\wtd\MB$ and $w\in\Wbb$. For each $\lambda,\mu\in\Cbb$, the map $P_\lambda\circ \upphi^\sharp_b(w)\circ P_\mu=\eqref{eqb86}$ belongs to the weight-$(\lambda,\mu)$ subspace
\begin{align*}
\End^0_A(\Mbb)_{[\lambda,\mu]}=P_\lambda\circ \End^0_A(\Mbb)\circ P_\mu=\Hom_A(\Mbb_{[\mu]},\Mbb_{[\lambda]})
\end{align*}
of $\End^0_A(\Mbb)$. Therefore, $\upphi^\sharp_b$ can be viewed as a linear map
\begin{gather*}
\upphi^\sharp_b:\Wbb\longrightarrow \ovl{\End_A^0(\Mbb)}=\prod_{\lambda,\mu\in\Cbb}\End^0_A(\Mbb)_{[\lambda,\mu]}
\end{gather*}
Thus $\upphi_b$ can be viewed as a linear functional $\Wbb\otimes\End_A^0(\Mbb)'\rightarrow\Cbb$, where
\begin{align*}
\End_A^0(\Mbb)'=\bigoplus_{\lambda,\mu\in\Cbb}\Hom_A(\Mbb_{[\mu]},\Mbb_{[\lambda]})^*
\end{align*}
is the contragredient $\Vbb^{\otimes2}$-module of $\End_A^0(\Mbb)$. Clearly $\upphi_b$ is still a conformal block associated to $\wtd\fx_b$. Thus:

\begin{pp}\label{lbb58}
$\upphi$ is naturally an element of $H^0\big(\wtd\MB,\scr T^*_{\wtd\fx}(\Wbb\otimes\End_A^0(\Mbb)')\big)$.
\end{pp}

\subsubsection{Reducing pseudo-sewing to sewing}

Let $\wtd{\fk Y}:=\wtd\fx\sqcup(\fn\times\wtd\MB)$. We sew $\wtd\fy$ along two pairs of marked points. The first pair is $(\sgm',\infty)$ with sewing radii $r,1$. The second pair is $(\sgm'',0)$ with sewing radii $\rho,1$. The result of sewing (cf. Subsec. \ref{lbb44}) gives a family $\fy$ with base manifold $\MD_r\times\MD_\rho\times\wtd\MB$.

By Prop. \ref{lbb57} and \ref{lbb58},
\begin{gather*}
\upphi\otimes\Tr^\omega:\Wbb\otimes\End^0_A(\Mbb)'\otimes\End^0_A(\Mbb)\rightarrow \MO(\wtd\MB)
\end{gather*}
is a conformal block associated to $\wtd\fy$. Its sewing is
\begin{gather}\label{eqb89}
\begin{gathered}
\MS(\upphi\otimes\Tr^\omega):\Wbb\rightarrow \MO(\wtd\MB)\{q_1,q_2\}[\log q_1,\log q_2]\\
w\mapsto \sum_{\lambda,\mu\in\Cbb}\Tr^\omega \big(q_1^{L(0)}P_\lambda\upphi^\sharp(w)P_\mu q_2^{L(0)}\big)
\end{gathered}
\end{gather} 
The RHS of \eqref{eqb89} is clearly in $\MO(\wtd\MB)\{q_1q_2\}[\log(q_1q_2)]$, i.e., it is of the form $f(q_1q_2)$ for some $f\in\MO(\wtd\MB)\{q\}[\log q]$.

\begin{thm}\label{lbb59}
As linear maps $\Wbb\rightarrow\MO(\wtd\MB)\{q\}[\log q]$ we have
\begin{align*}
\MS^\omega\upphi=\MS(\upphi\otimes\Tr^\omega)\big|_{q_1q_2=q}
\end{align*}
\end{thm}
\begin{proof}
This is clear from \eqref{eqb90} and \eqref{eqb89}.
\end{proof}

\begin{proof}[\textbf{Proof of Thm. \ref{lbb53}}]
By Thm. \ref{lbb49} and Rem. \ref{lbb56}, $\MS(\upphi\otimes\Tr^\omega)$ converges a.l.u. on $\MD_r^\times\times\MD_\rho^\times\times\wtd\MB$ to a conformal block associated to $\fy$ outside its discriminant locus. Therefore, by Thm. \ref{lbb59}, $\MS^\omega\upphi$ converges a.l.u. on $\MD_{r\rho}^\times\times\wtd\MB$. 

Moreover, for each fixed $q\in\MD_{r\rho}^\times$ and $\arg q$, we find $q_1\in\MD_r^\times,q_2\in\MD_\rho^\times$ with arguments such that $q=q_1q_2$ and $\arg q=\arg q_1+\arg q_2$. Then for each $b\in\MB$, the fiber $\fx_{(b,q)}$ is canonically equivalent to $\fy_{(b,q_1,q_2)}$. Thus $\MS^\omega\upphi|_{(b,q)}$ is a conformal block associated to $\fx_{(b,q)}$. Since the property of being a conformal block can be checked pointwise (cf. Rem. \ref{lbb27}), we conclude that $\MS^\omega\upphi$ is a conformal block associated to $\fx$ outside $\Delta=\{0\}\times\wtd\MB$. This finishes the proof of Thm. \ref{lbb53}.
\end{proof}

%\printindex	
%%%%%%%%%%%%%%%%%%%%%%%%%%%%%%%%%%%%%%%%%%%%%%%%%%%%%%%%%%%%%%%%
%  References
%%%%%%%%%%%%%%%%%%%%%%%%%%%%%%%%%%%%%%%%%%%%%%%%%%%%%%%%%%%%%%%%
\footnotesize
	\bibliographystyle{alpha}
    \bibliography{voa}

\noindent {\small \sc Yau Mathematical Sciences Center, Tsinghua University, Beijing, China.}

\noindent {\textit{E-mail}}: binguimath@gmail.com\qquad bingui@tsinghua.edu.cn\\

\noindent {\small \sc Yau Mathematical Sciences Center and Department of Mathematics, Tsinghua University, Beijing, China.}

\noindent {\textit{E-mail}}: zhanghao1999math@gmail.com \qquad h-zhang21@mails.tsinghua.edu.cn

\end{document}